\documentclass{amsart}
\usepackage{amssymb}
\usepackage{amscd}
\usepackage[latin1]{inputenc}
\usepackage{amsmath}
\usepackage{amsfonts}
\usepackage{latexsym}
\usepackage{verbatim}
\usepackage{graphicx}
\usepackage{epsfig}
\usepackage{faktor}
\usepackage{xcolor}
\usepackage{soul}
\usepackage{pgf, tikz}
\usepackage{caption}
\usetikzlibrary{arrows, automata}

\setstcolor{red}

\newtheorem{theorem}{Theorem}[section]
\newtheorem{corollary}[theorem]{Corollary}
\newtheorem{lemma}[theorem]{Lemma}

\newtheorem{proposition}[theorem]{Proposition}

\theoremstyle{definition}
\newtheorem{definition}[theorem]{Definition}
\newtheorem{example}[theorem]{Example}
\newtheorem{remark}[theorem]{Remark}

\newcommand{\zz}{\mathbb{Z}}

\newcommand{\aut}{\textrm{Aut}}
\newcommand{\A}{\mathcal{A}}

\newcommand{\rp}[1]{\textcolor{blue}{#1}}
\newcommand{\infin}{\textnormal{Inf}(\mathcal{G}_{T})}
\newcommand{\coinv}[1]{\mathcal{G}_{#1}}

\newcommand{\Scal}{\mathbf{S}}
\newcommand{\Scalprime}{\mathbf{S}^{\prime}}
\newcommand{\Tcal}{\mathbf{T}}
\newcommand{\Tcalprime}{\mathbf{T}^{\prime}}
\newcommand{\Tcalbar}{\overline{\mathbf{T}}}
\newcommand{\Tcalprimebar}{\overline{\mathbf{T}^{\prime}}}
\newcommand{\AtoZ}{\mathcal{A}^{\mathbb{Z}}}
\newcommand{\TTcal}{\mathbf{T}\mathbf{T}}
\newcommand{\TTcalprime}{\mathbf{T}\mathbf{T}^{\prime}}

\newcommand{\TTcalprimebar}{\overline{\mathbf{T}\mathbf{T}^{\prime}}}
\newcommand{\nmc}{\mathbf{NMC}}
\newcommand{\omc}{\mathbf{OMC}}

\begin{document}









\title[On the structure of generic subshifts]{On the structure of generic subshifts}

\author{Ronnie Pavlov}
\address{Ronnie Pavlov\\
Department of Mathematics\\
University of Denver\\
2390 S. York St.\\
Denver, CO 80208}
\email{rpavlov@du.edu}
\urladdr{http://www.math.du.edu/$\sim$rpavlov/}

\author{Scott Schmieding}
\address{Scott Schmieding\\
Department of Mathematics\\
University of Denver\\
2390 S. York St.\\
Denver, CO 80208}
\email{scott.schmieding@du.edu}
\urladdr{https://s-schmieding.github.io/}

\thanks{The first author gratefully acknowledges the support of a Simons Foundation Collaboration Grant.}
\renewcommand{\subjclassname}{MSC 2020}

\begin{abstract}
We investigate generic properties (i.e. properties corresponding to residual sets) in the space of subshifts with the Hausdorff metric. Our results deal with four spaces: the space $\Scal$ of all subshifts, the space $\Scalprime$ of non-isolated subshifts, the closure $\Tcalprimebar$ of the infinite transitive subshifts, and the closure $\TTcalprimebar$ of the infinite totally transitive subshifts.

In the first two settings, we prove that generic subshifts are fairly degenerate; for instance, all points in a generic subshift are biasymptotic to periodic orbits. In contrast, generic subshifts in the latter two spaces possess more interesting dynamical behavior. Notably, generic subshifts in both $\Tcalprimebar$ and $\TTcalprimebar$ are zero entropy, minimal, uniquely ergodic, and have word complexity which realizes any possible subexponential growth rate along a subsequence. In addition, a generic subshift in $\Tcalprimebar$ is a regular Toeplitz subshift which is strongly orbit equivalent to the universal odometer.
\end{abstract}

\maketitle

\section{Introduction}\label{intro}

One of the most well-studied classes of topological dynamical systems are the symbolically defined systems called \textit{subshifts}; a subshift is a closed subset of $\mathcal{A}^{\mathbb{Z}}$ for some finite alphabet $\mathcal{A}$ which is invariant under the shift map $\sigma$, and we denote the set of subshifts on $\mathcal{A}$ by $\Scal[\mathcal{A}]$.

In this work, we investigate the following question: what is the structure of a generic (or `typical') subshift? 
In order to treat all possible (finite) alphabet sizes, we define $\mathbf{S} = \bigcup_{\A \subset \mathbb{Z}, |\A| < \infty} \Scal[\A]$
(there is no loss of generality in assuming the alphabet is a subset of $\mathbb{Z}$, since we could always achieve this by renaming.) We endow this universal set of subshifts $\mathbf{S}$ with the Hausdorff metric (see Section~\ref{defs} for more details)
and, as is often done, say that a property $P$ is \textit{generic} (equivalently, that a generic subshift has property $P$) if the set of subshifts with that property is residual (contains a dense $G_{\delta}$ set) in the topological space $\mathbf{S}$. 
Other works that have used this topology include \cite{BLR}, \cite{CyrKraLanguageStable}, \cite{FrischTamuzGenericEntropy},
\cite{HochmanGeneric}, and \cite{Sigmund1971}.


The notion of genericity in dynamical systems has been long studied. Early work in this direction includes that of Oxtoby and Ulam~\cite{OxtobyUlam}, who showed that for a certain class of measures on a compact manifold of dimension two or greater, a generic volume-preserving homeomorphism is ergodic. Not long after, foundational work of Halmos~\cite{Halmos1944,Halmos1960} showed that a generic measure-preserving dynamical system (in the weak topology) is weakly mixing, followed by Rohlin~\cite{Rohlin1948} showing that generically such systems are not strongly mixing.  
In the topological setting, it was shown by Kechris and Rosendal in~\cite{KechrisRosendal} that in the space
$\textnormal{Homeo}(K)$ of homeomorphisms of the Cantor set $K$, the topological conjugacy class of a specific
transformation $T$ is residual. Although 
an explicit description of $T$ was not given in~\cite{KechrisRosendal}, Akin, Glasner, and Weiss in~\cite{AGW2008} later 
gave such a description, which showed that $T$ was somewhat degenerate dynamically. 


Later, in \cite{HochmanGeneric}, Hochman proved many results about more interesting generic properties within 
some distinguished subspaces of $\textnormal{Homeo}(K)$, namely the spaces of transitive and totally transitive systems. 
Interestingly, his proofs largely depend on transferring genericity results from the space $\Scal(Q)$ of `subshifts' with alphabet the Hilbert cube $Q$, i.e. closed shift-invariant subsets of $Q^{\mathbb{Z}}$. This is part of a more general phenomenon first established in the measure-theoretic setting, where often a property turns out to be generic in the space of transformations preserving a prescribed underlying measure if and only if it is generic in the space of Borel probability measures preserved by a prescribed continuous map (see \cite{GlasnerKing1998}, \cite{HochmanGeneric}, \cite{Rudolph1998}). 

A key distinction between our setting $\mathbf{S}$ and the space $\mathbf{S}(Q)$ treated in~\cite{HochmanGeneric} is that all subshifts on finite alphabets are \textit{expansive}.
A result of Sears~\cite{Sears1972} shows that expansive systems are 
meager in $\textnormal{Homeo}(K)$. We also note that $\mathbf{S}$ can be viewed as a subset of $\mathbf{S}(Q)$ (by embedding $\mathbb{Z}$ into $Q$), but it is clearly meager there as well. There's then no immediate reason for a generic system in $\mathbf{S}$ to behave similarly to those in $\textnormal{Homeo}(K)$ and $S(Q)$. 

However, under no restrictions, our results show that again generic systems in $\mathbf{S}$ have quite degenerate properties; see Theorem~\ref{mainall}. Upon restricting to subspaces of infinite transitive/totally transitive systems, we show that there are also some similarities; in both $\mathbf{S}$ and $\textnormal{Homeo}(K)$, generic transitive/totally transitive systems are minimal, uniquely ergodic, and have zero (topological) entropy. There are, however, substantial differences. For example, in the subspace of transitive homeomorphisms in $\textnormal{Homeo}(K)$, Hochman shows~\cite[Thm. 1.1]{HochmanGeneric} that the topological conjugacy class of the universal odometer is residual. 
In addition to the obvious fact that subshifts with finite alphabet cannot be conjugate to odometers, there is a deeper difference; topological conjugacy classes within $\mathbf{S}$ are necessarily countable, and we prove that the space of infinite transitive subshifts is perfect, and so cannot contain a countable residual set. 

Before summarizing our main results, we give a bit more detail about the space $\mathbf{S}$ and the associated Hausdorff metric. We can endow $\mathbb{Z}^\mathbb{Z}$ with the product topology (viewing $\mathbb{Z}$ as a discrete space), and the subshifts in $\Scal$ are then precisely the nonempty shift-invariant compact subsets of $\mathbb{Z}^\mathbb{Z}$. We consider then the space $\Scal$ with the Hausdorff metric, inherited from the space of all nonempty compact subsets of $\mathbb{Z}^{\mathbb{Z}}$. A useful informal description of this metric is that two subshifts $X, Y$ are close if, for some large value of $n$, the sets of $n$-letter words appearing in points of $X$ and those appearing in points of $Y$ coincide.


Our main results all describe various properties of a generic subshift in $\Scal$ and some distinguished subspaces.
A common theme is that generic behavior often turns out to be the `simplest possible' subject to unavoidable restrictions. We begin with the full space $\Scal$, where we give a complete description of the dynamics of a generic subshift.

\begin{theorem}\label{mainall}
A generic subshift $X$ in the space $\Scal$ of all subshifts has the following properties:
\begin{enumerate}
\item $X$ is a countable shift of finite type which is a union of finitely many orbits 
which are bi-asymptotic to periodic orbits (Theorem~\ref{isodense}).
\item The word complexity $c_{X}(n)$ of $X$ grows linearly (Theorem~\ref{lincplx}).
\item The automorphism group of $X$ is virtually free abelian of finite rank (Theorem~\ref{thm:autosinS}).
\end{enumerate}
\end{theorem}

Theorem~\ref{mainall} is a consequence of a technical result (Theorem~\ref{isodense}) which shows that a particular class called NMC subshifts are residual in $\Scal$.

Theorem~\ref{mainall} immediately 
implies all (to the authors' knowledge) existing results in the literature about the structure of generic subshifts: for example, that a generic subshift has zero entropy (proved by~\cite{FrischTamuzGenericEntropy,Sigmund1971}) and is language stable (proved by~\cite{CyrKraLanguageStable}).

It is natural to wonder what happens if instead of the universal space $\Scal$ of subshifts, one works with the space of subshifts contained in some prescribed subshift (as is often done, for example in~\cite{CyrKraLanguageStable,FrischTamuzGenericEntropy,BLR}). For any nonempty shift of finite type $X$, the subshifts contained in $X$ form a clopen subset of $\Scal$ (see Lemma~\ref{SFTclopen}), so our genericity results for $\Scal$ immediately imply analogous results for the space of subshifts contained in $X$.

Our proof in fact shows that the space $\Scal$ has a dense countably infinite set of isolated points whose complement is a Cantor set. 
Such a space is sometimes referred to as Pe\l{}czy\'{n}ski Space, due to a result of Pe\l{}czy\'{n}ski~\cite{Pelczynski1965} showing that a compact zero-dimensional metric space having such structure is unique up to homeomorphism. Then, by the previous paragraph, the topological structure of the space of subshifts contained in any mixing shift of finite type is independent of the shift of finite type; that is, any two such spaces of subshifts are automatically homeomorphic (Corollary~\ref{cor:pelczynski}).

Genericity in $\Scal$ is then completely determined by the dense set of isolated points, in the sense that a property is generic if and only if it holds for every subshift in this set. Since these subshifts have highly degenerate dynamics, a natural step is to consider the derived set of $\Scal$ obtained by removing its isolated points; we denote this space by
$\Scal^{\prime}$. 
We show that $\Scal^{\prime}$ is compact, totally disconnected, and perfect\footnote{This implies that the Cantor-Bendixson rank of $\Scal$ is one.}, and hence homeomorphic to the Cantor set; in particular, the question of genericity in $\Scal^{\prime}$ is not nearly as trivial as in $\Scal$. We also give a complete description of the dynamics of a generic subshift contained in $\Scal^{\prime}$.

\begin{theorem}\label{mainall1}
A generic subshift $X$ in $\Scalprime$, the derived space of $\Scal$ (consisting of all non-isolated points of $\Scal$), has the following properties:
\begin{enumerate}
\item $X$ is a countably infinite subshift in which every point is 
bi-asymptotic to one of finitely many periodic orbits (Corollary~\ref{OMCgen}).
\item $X$ has word complexity $c_{X}(n)$ whose growth is properly superlinear and subquadratic, but there are subsequences along which $c_{X}(n)$ exhibits arbitrarily slow proper superlinear growth (Corollaries~\ref{gencplx} and~\ref{quadres}).
\end{enumerate}
\end{theorem}

Theorem~\ref{mainall1} is a consequence of a technical result (Corollary~\ref{OMCgen}) which shows that a particular class called OMC subshifts are residual in $\Scal$.

Theorem~\ref{mainall1} implies that, similarly to $\Scal$, a generic subshift in $\Scal^{\prime}$ possesses
fairly degenerate dynamics, although less so than in $\Scal$. In particular, even in $\Scal^{\prime}$ a generic subshift is `uninteresting' dynamically (e.g., it is nontransitive, countable, and all points are bi-asymptotic to periodic orbits). 
Just as before, we can restrict further to rule out such degenerate behavior, and this time restrict to the transitive subshifts in
$\Scal$. 
To avoid pathological issues with the space, we again remove isolated points of $\Scal$, and thus we define $\Tcal^{\prime}$ to be the set of infinite transitive subshifts in $\Scal^{\prime}$. (We note that for transitive subshifts, being isolated is equivalent to being finite.)

It turns out that $\Tcal^{\prime}$ is not closed, and hence not complete as a metric space, so we consider its closure
$\overline{\Tcal^{\prime}}$. Within $\Tcalprimebar$, generic subshifts possess much more interesting dynamics.


\begin{theorem}\label{maintrans}
A generic subshift $X$ in the closure $\Tcalprimebar$ of the infinite transitive subshifts has the following properties:
\begin{enumerate}
\item $X$ is a regular Toeplitz subshift (and hence is minimal, uniquely ergodic, and has zero entropy) which factors onto the universal odometer (Theorem~\ref{univ}).
\item $X$ has topological rank two (Corollary~\ref{rk2}).
\item $X$ is strong orbit equivalent to the universal odometer, and in particular, the dimension group of $X$ has rank one and hence no nontrivial infinitesimals (Corollary~\ref{cor:onesoeclass}). 
\item There exist subsequences along which the word complexity function $c_{X}(n)$ has any desired linear/polynomial/stretched exponential growth which is consistent with $X$ being infinite, minimal, and of zero entropy (Proposition~\ref{bigcplx}).
\item The automorphism group of $X$ is generated by the shift map (Corollary~\ref{autgrp}), and the mapping class group is trivial (Theorem~\ref{thm:mcgfortransitivesubshifts}).
\end{enumerate}
In addition, for any increasing unbounded $h: \mathbb{N} \rightarrow \mathbb{R}^+$, a generic subshift in $\Tcalprimebar$ has word complexity satisfying $c_X(n) < n + h(n)$ along a subsequence (Corollary~\ref{sosmall}).
\end{theorem}

In~\cite{HochmanGeneric}, Hochman proves that in the space of transitive systems contained in $\textnormal{Homeo}(K)$, the topological conjugacy class of the universal odometer\footnote{Here, by the universal odometer we mean the unique (up to topological conjugacy) odometer which factors onto every finite transitive subshift.} is residual. There is no hope for any conjugacy class to be residual in
$\Tcalprimebar$ (since every conjugacy class is countable). However, Theorem~\ref{maintrans} gives two natural versions of this fact for our setting:
the first is that the set of subshifts strong orbit equivalent to the universal odometer is in fact residual in $\Tcalprimebar$, and the second is that a generic subshift in $\Tcalprimebar$ is an almost 1-1 extension of the universal odometer (due to being Toeplitz). 



Topological rank is a conjugacy invariant defined for minimal homeomorphisms of the Cantor set (which we can consider for generic subshifts in $\Tcalprimebar$ since they are minimal); see \cite{DDMP2021}.
A minimal Cantor system has topological rank one if and only if it is an odometer, and hence this is not achievable for subshifts. Thus, within the expansive setting, topological rank two is the least rank possible, and Theorem~\ref{maintrans} shows that again this `simplest possible' situation is generic in $\overline{\Tcal^{\prime}}$.




It is a well-known and difficult problem to characterize possible word complexity functions $c_{X}(n)$ (see \cite{FerCplx}), and the question of which growth rates can occur is very much open. It is therefore somewhat interesting that a generic subshift in $\Tcalprimebar$ must achieve all possible linear/quadratic/stretched exponential growth rates along subsequences; to our knowledge, no explicit examples of subshifts with this property are known.



Finally, we consider the subspace of infinite totally transitive subshifts $\TTcalprime$ in $\Scal$. Again this space is not closed, so we consider the closure $\TTcalprimebar$. We prove that a generic system in $\TTcalprimebar$ again exhibits, in many ways, the 'simplest' possible behavior (though some properties, like being Toeplitz, are precluded by definition).

\begin{theorem}\label{maintottrans}
A generic subshift $X$ in the closure $\TTcalprimebar$ of the infinite totally transitive subshifts has the following properties:
\begin{enumerate}
\item $X$ is zero entropy, minimal, topologically mixing, and uniquely ergodic (Theorems~\ref{ttgeneric} and~\ref{ttmixgen}).
\item $X$ has unique invariant measure which is weakly mixing and which has a rigidity sequence (Theorems~\ref{fullWM} and~\ref{rigid}).
\item $X$ has topological rank two (Theorem~\ref{ttgeneric}).
\item The dimension group of $X$ has rank two and no nontrivial infinitesimals (Proposition~\ref{dimrk2}).
\item The automorphism group of $X$ is generated by the shift map (Corollary~\ref{cor:ttbartrivauto}), and the mapping class group of $X$ is isomorphic to a subgroup of the affine group of $\mathbb{Q}$ (Theorem~\ref{thm:genericmcg}).
\end{enumerate}
In addition, for any increasing unbounded $h: \mathbb{N} \rightarrow \mathbb{R}^+$, a generic subshift in $\TTcalprimebar$ has word complexity satisfying $c_X(n) < n + h(n)$ along a subsequence (Corollary~\ref{sosmall2}).
\end{theorem}

We can see that many properties of generic subshifts in $\TTcalprimebar$ are similar to those in $\Tcalprimebar$. One difference is that a generic system in $\TTcalprimebar$ has dimension group with rank two rather than one; however, this is the minimal possible given the restriction (Corollary~\ref{killQ} in the main text) that nontrivial clopen sets for generic systems in $\TTcalprimebar$ must have irrational measure.
Another difference is that no orbit equivalence class is residual in $\TTcalprimebar$ (Corollary~\ref{cor:nogenoeclass} in the main text), whereas the strong orbit equivalence class of the universal odometer is residual in $\Tcalprimebar$. 

The restrictions on the mapping class group reflect the simplicity of a generic subshift in the totally transitive setting, given that~\cite{BoyleChuysurichayMCG} shows that the mapping class groups of mixing shifts of finite type are rather large groups.\\


Our results here concern only the case of subshifts over the group $\mathbb{Z}$, and a natural question is the extent to which analogous results might hold for subshifts over more general groups (see~\cite{FrischTamuzGenericEntropy} for example). Our proofs rely heavily on the use of Rauzy graphs (see Section 2 for definitions), and it is not (at least immediately) obvious how to adapt the techniques used here beyond $\mathbb{Z}$. Second, it would be interesting to expand beyond the zero-dimensional realm. For example, shifts of finite type are precisely the zero-dimensional Smale spaces, and as a starting point one might consider the following question: for a fixed Smale space $(X,f)$, what is the structure of a generic subsystem within the space of compact subsystems of $(X,f)$ with the Hausdorff metric? We note that for a Smale space $(X,f)$, a Markov partition yields a shift of finite type cover $\pi \colon (Y,\sigma) \to (X,f)$ which in turn induces a continuous surjection from the space of subshifts of $(Y,\sigma)$ to the space of subsystems of $(X,f)$.

Finally, we briefly describe the structure of the paper. Section~\ref{defs} presents some useful definitions and preliminary results. In Sections~\ref{main} and \ref{main2}, we prove genericity results about the universal subshift space $\Scal$ and the derived set $\Scalprime$, which lead to Theorems~\ref{mainall} and \ref{mainall1}. The proofs in those sections rely on fairly technical arguments about so-called Rauzy graphs, and the fact that generic properties in those settings are essentially controlled by simple properties of the Rauzy graphs. In Sections~\ref{trans} and \ref{tottrans}, we prove results about the closures $\Tcalprimebar$ and $\TTcalprimebar$ of the spaces of infinite transitive and totally transitive subshifts respectively; these lead to Theorems~\ref{maintrans} and \ref{maintottrans}. The proofs in these sections rely on more dynamical arguments; a key technique in both is showing that for any subshift $X$ satisfying weak hypotheses, there is a dense class of subshifts `similar' to $X$ (see Theorems~\ref{letword} and \ref{conjdense}).


\section{Definitions and preliminaries}\label{defs}

\subsection{General symbolic dynamics definitions}
Let $\mathcal{A}$ be a finite subset of $\mathbb{Z}$. We consider points in $\mathcal{A}^{\mathbb{Z}}$ as biinfinite sequences $x=(x_{i})_{i \in \mathbb{Z}}$ where each $x_{i} \in \mathcal{A}$. Using the metric on $\mathcal{A}^{\mathbb{Z}}$ defined by
\begin{equation}\label{eqn:shiftmetric}
d(x,y) 
= 2^{-\inf\{|k| \colon x_{k} \ne y_{k}\}}
\end{equation}
the space $\mathcal{A}^{\mathbb{Z}}$ becomes a compact metric space. (The set inside the infimum is empty if and only if $x = y$, and in this case we declare the infimum to be $\infty$ and correspondingly $d(x,y) = 0$.) We define the shift map $\sigma \colon \AtoZ \to \AtoZ$ by $\left(\sigma(x)\right)_{i} = x_{i+1}$ and note that $\sigma$ is a self-homeomorphism of $\AtoZ$. By a \emph{subshift}
(on $\mathcal{A}$) we mean a compact $\sigma$-invariant subset of $\AtoZ$. Throughout, the term subshift will always refer to a subshift on a finite alphabet $\mathcal{A} \subset \mathbb{Z}$.

Typically we only refer to a subshift via its corresponding domain $X$; when necessary, we'll write $(X,\sigma_{X})$ when we want to also indicate the respective shift map.


The \textit{language} of a subshift $X$ on $\mathcal{A}$, denoted $L(X)$, is the set of all finite words appearing as subwords of points in $X$. For any $n \in \mathbb{N}$, we denote $L_n(X) = L(X) \cap \mathcal{A}^n$, the set of $n$-letter words in $L(X)$.


For a subshift $X$, the \textit{word complexity function} of $X$ is defined by $c_{X}(n) := |L_{n}(X)|$. 

For a subshift $X$ and word $w \in L(X)$ we denote by $[w]$ the clopen subset in $X$ consisting of all $x \in X$ such that $x_{0}\ldots x_{|w|-1} = w$. By an \textit{$X$-cylinder set} we mean any set of the form $[w]$ for some $w \in L(X)$. Note that we use the term $X$-cylinder set instead of the more commonly used 'cylinder set' to refer to such sets, since we use the term cylinder set to refer to a different class of objects (see Definition~\ref{def:cylinderset}).

A subshift $(X,\sigma_{X})$ is \textit{transitive} if there exists $x \in X$ such that the set $\{\sigma_{X}^{i}(x)\}_{i \in \mathbb{Z}}$ is dense in $X$, and $(X,\sigma_{X})$ is \textit{totally transitive} if for all $n \ge 1$ there exists $x^{(n)} \in X$ such that the set $\{\sigma_{X}^{in}(x^{(n)})\}_{i \in \mathbb{Z}}$ is dense in $X$. A subshift $(X, \sigma_X)$ is \textit{topologically mixing} if, for every $v, w \in L(X)$, there exists $N$ so that for all $n > N$, $[v] \cap \sigma_X^{n} [w] \neq \varnothing$.

A subshift $X$ is \textit{minimal} if for all $w \in L(X)$ and $x \in X$, $w$ appears as a subword of $x$. A subshift $(X,\sigma_{X})$ is \textit{uniquely ergodic} if there exists precisely one $\sigma_{X}$-invariant Borel probability measure on $X$. The \textit{(topological) entropy} of a subshift $(X,\sigma_{X})$ is the quantity $h_{top}(\sigma_{X}) = \lim_{n \to \infty} \frac{1}{n} \log |L_{n}(X)|$.

A subshift $X \subset \mathcal{A}^{\mathbb{Z}}$ is a \textit{shift of finite type} if there exists a finite set of words $F$ over the alphabet $\mathcal{A}$ such that $X$ consists of all points in $\mathcal{A}^{\mathbb{Z}}$ not containing any word from $F$. If $\Gamma$ is a finite labeled directed graph with edge set $\mathcal{E}(\Gamma)$ then the set of points obtained from biinfinite walks on $\Gamma$ is a shift of finite type on the alphabet $\mathcal{E}(\Gamma)$.

A \textit{(topological) conjugacy} between subshifts $(X,\sigma_{X})$ and $(Y,\sigma_{Y})$ is a homeomorphism $\phi \colon X \to Y$ such that $\phi \sigma_{X} = \sigma_{Y} \phi$. An \textit{automorphism} of a subshift $(X,\sigma_{X})$ is a conjugacy from
$(X, \sigma_X)$ to itself. The \textit{automorphism group} of a subshift $(X,\sigma_{X})$ (denoted by $\aut(X,\sigma_{X})$) is the set of all automorphisms of $(X, \sigma_X)$ with the operation of composition.

For a subshift $(X,\sigma_{X})$ we may consider the \textit{suspension space} $\Sigma_{\sigma_{X}}X = (X \times [0,1]) / \sim$ where $(x,t) \sim (\sigma_{X} x, t-1)$ which carries a flow defined by $\phi_{s}(x,t) = (x,t+s)$. The \textit{mapping class group} $\mathcal{M}(\sigma_{X})$ of $(X,\sigma_{X})$ is the group of isotopy classes of orientation-preserving self-homeomorphisms of $\Sigma_{\sigma_{X}}X$ (for background on mapping class groups of subshifts, see~\cite{BoyleChuysurichayMCG,SchmiedingYang2021}).

\subsection{Minimality and notions of rank}
For a subshift $(X,\sigma_{X})$ we let $C(X,\mathbb{Z})$ denote the group of continuous integer-valued functions on $X$. The coboundary map $\partial \colon C(X,\mathbb{Z}) \to C(X,\mathbb{Z})$ is defined by $\partial \colon f \mapsto f - f \circ \sigma_X$, and we define the group of coinvariants associated to $(X,\sigma_{X})$ by
\begin{equation*}
\coinv{\sigma_{X}} = C(X,\mathbb{Z}) / \textnormal{Image}(\partial).
\end{equation*}
Contained in $\coinv{\sigma_{X}}$ is a positive cone $\coinv{\sigma_{X}}^{+} = \{ [f] \mid f \textnormal{ is nonnegative}\}$, and taking the class $[1] \in \coinv{\sigma_{X}}$ of the constant function $1 \colon x \mapsto 1$ as a distinguished order unit, the triple $(\coinv{\sigma_{X}},\coinv{\sigma_{X}}^{+},[1])$ is a unital preordered group (see~\cite{BoyleHandelman1996}). If $(X,\sigma_{X})$ is minimal then $(\coinv{\sigma_{X}},\coinv{\sigma_{X}}^{+},[1])$ is a unital ordered group and $(\coinv{\sigma_{X}},\coinv{\sigma_{X}}^{+},[1])$ is a simple dimension group~\cite{GPS1995}.
Since we will only consider $\coinv{\sigma_{X}}$ when $(X,\sigma_{X})$ is minimal, through an abuse of language we will refer to
$(\coinv{\sigma_{X}},\coinv{\sigma_{X}}^{+},[1])$ as the \emph{dimension group} associated to $(X,\sigma_{X})$ (we will also often abuse language and simply write $\coinv{\sigma_{X}}$ to refer to the triple).

By \cite{GPS1995}, for any minimal Cantor system $(X,T)$ there exists a properly ordered Bratteli diagram such that $(X,T)$ is topologically conjugate to the Vershik map on this Bratteli diagram. We say a minimal system $(X,T)$ has \textit{topological rank $d$} if it has a Bratteli-Vershik presentation for which the number of vertices per level is uniformly bounded by $d$, and there is no Bratteli-Vershik presentation of $(X,T)$ whose number of vertices per level is uniformly bounded by a number less than $d$ (for more details on these definitions, see~\cite{DDMP2021}). If two finite rank minimal subshifts are topologically conjugate, then they have the same rank; see~\cite[Cor. 4.7]{Espinoza2021} for a proof of this.

\subsection{The space of subshifts and some preliminary results}
For a metric space $X$, we let $\mathcal{K}(X)$ denote the space of nonempty compact subsets of $X$ with the Hausdorff metric. When $X$ is compact, $\mathcal{K}(X)$ is compact, and when $X$ is complete, $\mathcal{K}(X)$ is complete. Recall $\mathbb{Z}^{\mathbb{Z}}$ is a complete metric space with the metric~\eqref{eqn:shiftmetric}, and hence $\mathcal{K}(\mathbb{Z}^{\mathbb{Z}})$ is complete. We may consider the space of subshifts as the subspace of $\mathcal{K}(\mathbb{Z}^{\mathbb{Z}})$ consisting of nonempty compact sets $K \subset \mathbb{Z}^{\mathbb{Z}}$ such that $\sigma(K)=K$ (we assume from now on that all subshifts are nonempty). It turns out that in this setting, the Hausdorff metric may be equivalently defined in terms of languages, which we will instead use throughout.
\begin{definition}
We define the space of subshifts $\Scal$ to be the set of nonempty compact subsets $K$ of $\mathbb{Z}^{\mathbb{Z}}$ such that $\sigma(K)=K$, together with the metric defined as follows: for nonempty subshifts $X, Y$,
\begin{equation}\label{eqn:Hausdorff}
d(X,Y) := 2^{-\inf\{n \ \mid \ L_n(X) \neq L_n(Y)\}}.
\end{equation}
\end{definition}

The metric defined on $\Scal$ above is equivalent to the Hausdorff metric inherited from considering $\Scal$ as a subspace of $\mathcal{K}(\mathbb{Z}^{\mathbb{Z}})$. We note that $\Scal$ is closed in $\mathcal{K}(\mathbb{Z}^{\mathbb{Z}})$, and hence complete; indeed, $\Scal$ is precisely the set of fixed points of the homeomorphism $\tilde{\sigma} \colon \mathcal{K}(\mathbb{Z}^{\mathbb{Z}}) \to \mathcal{K}(\mathbb{Z}^{\mathbb{Z}})$ induced by the shift map $\sigma \colon \mathbb{Z}^{\mathbb{Z}} \to \mathbb{Z}^{\mathbb{Z}}$. The space $\mathcal{K}(\mathbb{Z}^{\mathbb{Z}})$ is totally disconnected, separable, and complete, since $\mathbb{Z}^{\mathbb{Z}}$ is complete; thus $\Scal$ is a Polish space.

\begin{definition}\label{def:cylinderset}
For a subshift $X \in \Scal$ and $n \in \mathbb{N}$, the \textit{$n$-cylinder set of $X$} in $\Scal$ is
\[
[X,n] := \{Y \ \mid \ L_n(Y) = L_n(X)\}.
\]
A \textit{cylinder set} is simply any such set $[X,n]$. For any subspace $U$ in $\Scal$, by a cylinder set in $U$ we will always mean the intersection of a cylinder set $[X,n]$ in $\Scal$ with $U$.
\end{definition}

It's clear from the definition of the metric (\ref{eqn:Hausdorff}) that for every $X$ and $n$,
$[X, n] = \{Y \ \mid \ d(X, Y) < 2^{-n}\} = \{Y \ \mid \ d(X, Y) \leq 2^{-(n+1)}\}$, and so every cylinder set is clopen by definition. Also,
we note that any nonempty cylinder $C = [X,n]$ in a subspace $U$ can, without loss of generality, always be assumed to be based at a subshift $Y$ in that subspace $U$.

For a subshift $Y \in \Scal$, we let $\Scal(Y)$ denote the space of subshifts contained in $Y$.

\begin{lemma}\label{SFTclopen}
For any subshift $Y \in \Scal$, the space $\Scal(Y)$ is compact, and when $Y$ is a shift of finite type, $\Scal(Y)$ is also open in $\Scal$.
\end{lemma}

\begin{proof}
Fix any $Y \in \Scal$. Then there is a finite $\mathcal{A} \subset \mathbb{Z}$ for which $\Scal(Y)$ is a subset of the compact set $\Scal[\mathcal{A}]$, meaning that it suffices to prove that $\Scal(Y)$ is closed. But this is immediate; if $Z$ is a limit of subshifts $Z_n \in \Scal(Y)$, then we can assume without loss of generality that $L_n(Z_n) = L_n(Z)$ for all $n$. Since $Z_n \subset Y$, $L_n(Z_n) \subset L_n(Y)$, so $L_n(Z) \subset L_n(Y)$ for all $n$, implying that $Z \subset Y$ so $Z \in \Scal(Y)$.

If $Y$ is a shift of finite type, then there exists $n$ and a finite set $\mathcal{L}$ of words so that $Y$ consists of all sequences in which every $n$-letter subword is in $\mathcal{L}$. Then $\Scal(Y)$ can be written as a union of cylinder sets $[X, n]$ over all $X$ where $L_n(X) \subset \mathcal{L}$, so it is open.
\end{proof}
A useful consequence of Lemma~\ref{SFTclopen} is the following: for any shift of finite type $Y$, any class of subshifts which is generic in $\Scal$ is also generic within the space $\Scal(Y)$.

There is another quite useful way to view cylinders in $\Scal$, which requires the notion of Rauzy graphs.


\begin{definition}
For a subshift $X$ and $n \in \mathbb{N}$, the \textit{$n$th Rauzy graph of $X$} is the directed graph $G_{X, n}$ with vertex
set $L_{n-1}(X)$, and directed edges from $w_1 \ldots w_{n-1}$ to $w_2 \ldots w_{n}$ for all $w_1 \ldots w_{n} \in L_{n}(X)$.
\end{definition}

\begin{example}
If $X$ is the golden mean subshift consisting of biinfinite sequences on $\{0,1\}$ without consecutive $1$s, and $n = 4$, then $G_{X, 4}$ is the following directed graph:\\

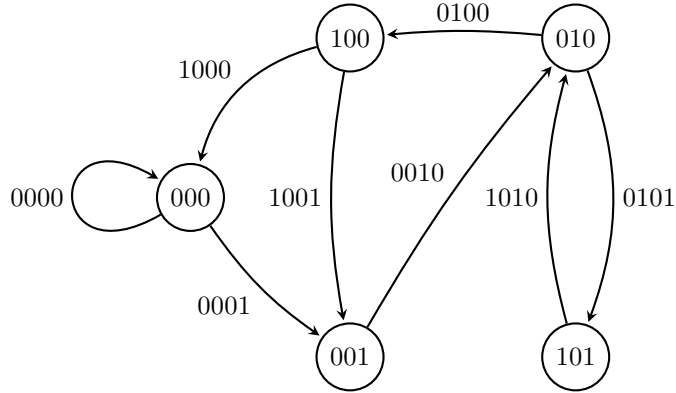
\begin{center}
\tikzset{every loop/.style={min distance=10mm,in=150,out=210,looseness=10}}
\begin{tikzpicture}[
            > = stealth, 
            shorten > = 1pt, 
            auto,
            node distance = 3cm, 
            thick 
        ]

        \tikzstyle{every state}=[
            draw = black,
            thick,
            fill = white,
            minimum size = 4mm
        ]

        \node[state] (000) {$000$};
        \node[state] (100) [above right of=000] {$100$};
        \node[state] (010) [right of=100] {$010$};
        \node[state] (001) [below right of=000] {$001$};
        \node[state] (101) [right of=001] {$101$};

        \path[->] (000) edge [loop left] node {0000} (000);
        \path[->] (000) edge [bend right=10] node[swap] {0001} (001);
        \path[->] (100) edge [bend right=10] node[swap] {1001} (001);
        \path[->] (100) edge [bend right=30] node[swap] {1000} (000);
        \path[->] (010) edge [bend right=5] node[swap] {0100} (100);
        \path[->] (001) edge [bend left=5] node {0010} (010);
        \path[->] (101) edge [bend left=15] node {1010} (010);
        \path[->] (010) edge [bend left=20] node {0101} (101);
    \end{tikzpicture}
    \end{center}
\end{example}

It is immediate from the definition that every Rauzy graph $G_{X,n}$ is \textit{essential}, i.e. that every vertex has at least one incoming edge and at least one outgoing edge. In particular, this means that any finite path can be extended on the left and right to a biinfinite path. All directed graphs we treat are assumed to be essential.

For each $n \in \mathbb{N}$, there is a map from points of $X$ to the Rauzy graph $G_{X,n}$: to any $x \in X$, associate the biinfinite path whose $k$th vertex is $x_k \ldots x_{k+n-2}$ for all $k \in \mathbb{Z}$. Every biinfinite path in $G_{X,n}$ is associated to some biinfinite sequence, which may or may not be in $X$.


The following lemma shows the fundamental connection between Rauzy graphs and cylinders in the space $\Scal$. It is immediate from the definitions, but is so fundamental to our arguments that we state it explicitly.

\begin{lemma}\label{rauzy}
For any subshift $X$ and $n \in \mathbb{N}$, $[X,n]$ is the set of all subshifts $Y$ for which $G_{X,n} = G_{Y,n}$. In particular, for any such $Y$, every $y \in Y$ corresponds to a biinfinite path on $G_{X,n}$ and every edge in $G_{X,n}$ is part of at least one such path.
\end{lemma}

We can now use Rauzy graphs to see that each cylinder set contains a maximal element with respect to inclusion.


\begin{definition}
Suppose $X$ is a subshift and $G_{X,n}$ is a Rauzy graph of $X$ for some $n$. For a subgraph $G$ of $G_{X,n}$, we define $S(G)$ to be the subshift consisting of all biinfinite sequences corresponding to paths on $G$. For a cylinder set $C = [X,n]$, define $S(C)$ to be $S(G_{X,n})$.
\end{definition}
It is clear from the definition that for any cylinder set $C=[X,n]$, the subshift $S(C)$ is always a shift of finite type, since it is defined in terms of allowed $n$-letter words. We note that since $G_{X,n}$ is essential, every edge in $G_{X,n}$ is part of some biinfinite path, so by Lemma~\ref{rauzy}, $S(C) \in C$.

The following lemma is left to the reader.
\begin{lemma}
For every cylinder set $C$, 
every $Y \in C$ is a subset of $S(C)$.
\end{lemma}



We now summarize the decomposition of finite directed graphs into \textit{irreducible components} from \cite[Sec. 4.4]{LM}.

Every directed graph $G$ has an equivalence relation $\sim$ defined as follows: define $v \sim v$ for all vertices $v$, and for any vertices $v \neq w$, set $v \sim w$ if and only if there are (directed) paths from $v$ to $w$ and from $w$ to $v$. Then the vertex set of $G$ can be partitioned into equivalence classes, and the induced subgraphs corresponding to these equivalence classes are the \textit{irreducible components} of $G$, which we denote by $C_1, \ldots, C_k$. We can then define the \textit{tree structure} of $G$ as a directed graph with vertex set $\{C_i\}_{i=1}^{k}$, where there is an edge from $C_i$ to $C_j$ for $i \neq j$ if and only if there is a path in $G$ from some vertex in $C_i$ to some vertex in $C_j$. This directed graph $I(G)$ must have no cycles, since the existence of a cycle would imply equivalence of vertices in distinct irreducible components, a contradiction to their definition. The structural implication is that $G$ must be a union of the components $C_k$, along with some one-way edges between components corresponding to edges of $I(G)$ (a single edge in $I(G)$ can of course correspond to many transitions between the associated pair of components.)

We will change this decomposition very slightly: we allow finite one-way paths between components and will then always assume that all $C_i$ are nontrivial, i.e. contain at least one edge. This can be done since components with no edges must be a single vertex with no self-loop, and so can just be subsumed within a finite simple transition path.

Any irreducible component with an outgoing edge but no incoming edge is called a \textit{source}, and any irreducible component with an incoming edge but no outgoing edges is called a \textit{sink}. A directed graph has a sink and source if and only if it is not irreducible.

Finally, we prove two brief results about word complexity. The first requires a definition; for a subshift $X$, we say that a word $w \in L(X)$ is \textit{right-special} if there exist letters $a \neq b$ so that $wa, wb \in L(X)$. Equivalently, right-special words are those with multiple preimages under the map $f: L_{n+1}(X) \rightarrow L_n(X)$ which removes the final letter. That description leads to the following immediate corollary.

\begin{corollary}\label{RScor}
For any subshift $X$ on $\mathcal{A}$ and any $n$,
\[
c_X(n+1) - c_X(n) = \sum_{\substack{w \in L_n(X)\\w \textrm{ right-special}}} (|\{a \in \mathcal{A} \ | \ wa \in L_{n+1}(X)\}|-1).
\]
\end{corollary}

Finally, we need to show that the set of subshifts satisfying a complexity restriction infinitely often is residual. For any functions $f,g \colon \mathbb{N} \rightarrow \mathbb{R}^+$, let $\mathbf{S}_{f,g}$ denote the collection of subshifts $X$ for which $f(n) \leq c_{X}(n) \leq g(n)$ for infinitely many $n$.

\begin{theorem}\label{cplxthm}
For any functions $f,g \colon \mathbb{N} \rightarrow \mathbb{R}^+$, the set $\mathbf{S}_{f,g}$ is a $G_{\delta}$ in $\Scal$.
\end{theorem}

\begin{proof}

Consider any such functions $f,g$. It suffices to simply note that $\mathbf{S}_{f,g}$ can be written as
\[
\bigcap_{N \in \mathbb{N}} \bigcup_{n > N} \{X \in \Scal \ \mid \ f(n) \leq c_{X}(n) \leq g(n)\} =
\bigcap_{N \in \mathbb{N}} \bigcup_{n > N} \bigcup_{X \in \Scal, f(n) \leq c_{X}(n) \leq g(n)} [X,n].
\]
(Though the final union is technically over uncountably many $X$, $[X,n]$ depends only on $L_n(X)$, for which there are only countably many possibilities. This will be the case for many similar unions for the rest of the paper, which we will not comment each time on.)
\end{proof}

\section{The space $\Scal$ of all subshifts}\label{main}

We begin by working in the space $\Scal$ of all subshifts. Our main results here determine the structure of a generic subshift in this space, which is quite degenerate. In particular, we will show that there is a countable subspace of subshifts in $\Scal$ (which are each, as dynamical systems, degenerate) which completely controls whether subsets of $\Scal$ are residual or not. As a side consequence, the results give drastically simpler proofs of some results in the literature.

\begin{definition}
A directed graph has the \textit{no middle cycles property} (or NMC) if it contains no cycle $K$ which has both an incoming edge (i.e. an edge whose terminal vertex is in $K$ and whose initial vertex is not in $K$) and an outgoing edge (i.e. an edge whose initial vertex is in $K$ and whose terminal vertex is not in $K$). 
\end{definition}


The following alternate representation of graphs having NMC will be helpful, for which need a definition.

\begin{definition}
An \textit{empty barbell} is a directed graph consisting of two vertex-disjoint simple cycles with a single simple directed transition path between them.
\end{definition}


\begin{lemma}\label{NMC}
If a directed graph has NMC, then it can be written as a finite (not necessarily disjoint) union of isolated simple cycles
(i.e. simple cycles with no incoming or outgoing edge) and empty barbells, where all isolated cycles, initial cycles of empty barbells, and
terminal cycles of empty barbells are vertex-disjoint.

If a directed graph can be written as such a union where no cycle is both initial within one empty barbell and terminal within another, then it has NMC.
\end{lemma}

\begin{center}
\tikzset{every loop/.style={min distance=20mm,in=310,out=240,looseness=10}}
\tikzset{->-/.style={decoration={
  markings,
  mark=at position .5 with {\arrow{>}}},postaction={decorate}}}
\begin{minipage}{.5\textwidth}
\begin{tikzpicture}[scale=6,transform shape,
            > = stealth, 
            shorten > = 1pt, 
            auto,
            node distance = 3cm, 
            thick, 
            scale=.1
        ]
        \node[draw,fill,circle,inner sep=0pt,minimum size=4pt] (a) at (-2,3) {};
        \node[draw,fill,circle,inner sep=0pt,minimum size=4pt] (b) at (-0.5,2) {};
        \node[draw,fill,circle,inner sep=0pt,minimum size=4pt] (c) at (1,3) {};
        \node[draw,fill,circle,inner sep=0pt,minimum size=4pt] (d) at (-0.5,4) {};
        \node[draw,fill,circle,inner sep=0pt,minimum size=4pt] (e) at (3.5,3.8) {};
        \node[draw,fill,circle,inner sep=0pt,minimum size=4pt] (f) at (5.2,3.8) {};
        \node[draw,fill,circle,inner sep=0pt,minimum size=4pt] (g) at (4.4,2.6) {};
        \node[draw,fill,circle,inner sep=0pt,minimum size=4pt] (h) at (-1.5,-2.6) {};
        \node[draw,fill,circle,inner sep=0pt,minimum size=4pt] (i) at (-0.5,-1.7) {};
        \node[draw,fill,circle,inner sep=0pt,minimum size=4pt] (j) at (0.5,-2.6) {};
        \node[draw,fill,circle,inner sep=0pt,minimum size=4pt] (k) at (2.5,-1.4) {};
        \node[draw,fill,circle,inner sep=0pt,minimum size=4pt] (l) at (2.5,-3) {};
        \node[draw,fill,circle,inner sep=0pt,minimum size=4pt] (m) at (4.7,-1.8) {};
        \path[->] (a) edge [bend right=30] node {} (b);
        \path[->] (b) edge [bend right=30] node {} (c);
        \path[->] (c) edge [bend right=30] node {} (d);
        \path[->] (d) edge [bend right=30] node {} (a);
        \path[->] (e) edge [bend left=30] node {} (f);
        \path[->] (f) edge [bend left=30] node {} (g);
        \path[->] (g) edge [bend left=30] node {} (e);
        \path[->] (h) edge [bend left=40] node {} (i);
        \path[->] (i) edge [bend left=40] node {} (j);
        \path[->] (j) edge [bend left=40] node {} (h);
        \path[->] (k) edge [bend right=60] node {} (l);
        \path[->] (l) edge [bend right=60] node {} (k);
        \path[->] (m) edge [loop right] node {} (m);
        \path[->] (b) edge [bend right=35] node {} (i);
        \path[->] (b) edge [bend right=30] node {} (k);
        \path[->] (b) edge [bend left=25] node {} (m);
        \path[->] (g) edge [bend right=10] node {} (i);
        \path[->] (g) edge [bend left=20] node {} (k);
        \path[->] (g) edge [bend left=15] node {} (m);
    \end{tikzpicture}
   
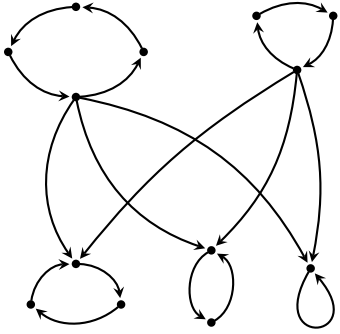
\captionof{figure}{This graph has NMC.}
   \label{fig:fig1}
    \end{minipage}%
\begin{minipage}{0.5\textwidth}
\begin{tikzpicture}[scale=6,transform shape,
            > = stealth, 
            shorten > = 1pt, 
            auto,
            node distance = 3cm, 
            thick, 
            scale=.1
        ]
        \node[draw,fill,circle,inner sep=0pt,minimum size=4pt] (a) at (-2,3) {};
        \node[draw,fill,circle,inner sep=0pt,minimum size=4pt] (b) at (-0.5,2) {};
        \node[draw,fill,circle,inner sep=0pt,minimum size=4pt] (c) at (1,3) {};
        \node[draw,fill,circle,inner sep=0pt,minimum size=4pt] (d) at (-0.5,4) {};
        \node[draw,fill,circle,inner sep=0pt,minimum size=4pt] (e) at (3.5,3.8) {};
        \node[draw,fill,circle,inner sep=0pt,minimum size=4pt] (f) at (5.2,3.8) {};
        \node[draw,fill,circle,inner sep=0pt,minimum size=4pt] (g) at (4.4,2.6) {};
        \node[draw,fill,circle,inner sep=0pt,minimum size=4pt] (h) at (-1.5,-2.6) {};
        \node[draw,fill,circle,inner sep=0pt,minimum size=4pt] (i) at (-0.5,-1.7) {};
        \node[draw,fill,circle,inner sep=0pt,minimum size=4pt] (j) at (0.5,-2.6) {};
        \node[draw,fill,circle,inner sep=0pt,minimum size=4pt] (k) at (2.5,-1.4) {};
        \node[draw,fill,circle,inner sep=0pt,minimum size=4pt] (l) at (2.5,-3) {};
        \node[draw,fill,circle,inner sep=0pt,minimum size=4pt] (m) at (4.7,-1.8) {};
        \path[->] (a) edge [bend right=30] node {} (b);
        \path[->] (b) edge [bend right=30] node {} (c);
        \path[->] (c) edge [bend right=30] node {} (d);
        \path[->] (d) edge [bend right=30] node {} (a);
        \path[->] (e) edge [bend left=30] node {} (f);
        \path[->] (f) edge [bend left=30] node {} (g);
        \path[->] (g) edge [bend left=30] node {} (e);
        \path[->] (h) edge [bend left=40] node {} (i);
        \path[->] (i) edge [bend left=40] node {} (j);
        \path[->] (j) edge [bend left=40] node {} (h);
        \path[->] (k) edge [bend right=60] node {} (l);
        \path[->] (l) edge [bend right=60] node {} (k);
        \path[->] (m) edge [loop right] node {} (m);
        \path[->] (b) edge [bend right=35] node {} (i);
        \path[->] (b) edge [bend right=30] node {} (k);
        \path[->] (b) edge [bend left=25] node {} (m);
        \path[->] (g) edge [bend right=10] node {} (i);
        \path[->] (g) edge [bend left=20] node {} (k);
        \path[->] (g) edge [bend left=5] node {} (m);
        \path[->] (c) edge [bend left=15] node {} (e);
    \end{tikzpicture}
       
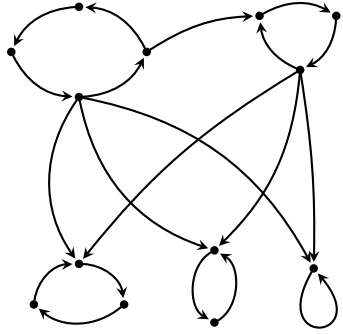
\captionof{figure}{This graph does not have NMC.}
   \label{fig:fig2}
    \end{minipage}
    \end{center}

For instance, the graph shown in Figure 1 has NMC, while the graph shown in Figure 2 does not have NMC, because of the upper-right cycle.

\begin{definition}
A subshift $X$ has the \textit{no middle cycles property} (or NMC) if there exists $n$ for which $G_{X,n}$ has NMC.
\end{definition}

We denote by $\mathbf{NMC}$ the set of NMC subshifts in $\Scal$. The following characterization of NMC subshifts will also be useful.

\begin{lemma}\label{NMCstructure}
A subshift $X$ has NMC if and only if there exist finite sets $P, M, S$ of words so that every sequence in $X$ is either of the form $p^{\infty}$, $s^{\infty}$, or $p^{\infty} m s^{\infty}$ for some $p \in P$, $m \in M$, and $s \in S$, and so that no sequence can be written both as $p^{\infty}$ and $s^{\infty}$ for some $p \in P$ and $s \in S$.
\end{lemma}

\begin{proof}
To see that each NMC subshift has the claimed structure, choose such an $X$ and $n$ for which $G_{X,n}$ has NMC. By Lemma~\ref{NMC}, $G_{X,n}$ can be written as a finite union of isolated simple cycles and empty barbells. Then all sequences in $X$ correspond to biinfinite paths from one of these cycles or empty barbells. Then we can define $P$ to be words corresponding to initial cycles in empty barbells, $S$ to be words corresponding to isolated cycles or terminal cycles in empty barbells, and $M$ to be words corresponding to transition paths in empty barbells. The claimed decomposition is immediate, and no $p^{\infty}$ and $s^{\infty}$ may be equal since no cycle can be initial and terminal in two empty barbells (by Lemma~\ref{NMC}).

Now, assume that a subshift $X$ has the claimed decomposition for sets $P, M, S$. By passing to subsets, we can assume without loss of generality that $p^{\infty}, s^{\infty} \in X$ for all $p \in P$, $s \in S$. Now, choose $N$ which is a common period for all sequences $p^{\infty}$ and $s^{\infty}$ and which is greater than the lengths of all words in $M$. Note that the sets of $N$-letter subwords of sequences of the form $p^{\infty}$ and $s^{\infty}$ must be disjoint, since such a word determines its entire containing sequence.

Then $G_{X,2N}$ is just the finite union of the graphs induced by all sequences $p^{\infty}$, $s^{\infty}$ for $p \in P$ and $s \in S$, and all sequences of the form $p^{\infty} m s^{\infty}$ which are in $X$. It is clear that the first two classes of graphs are simple cycles. We claim that each graph induced by $p^{\infty} m s^{\infty}$ is an empty barbell consisting of a simple cycle induced by $p^{\infty}$ with a single simple directed path to a simple cycle induced by $s^{\infty}$. Then the union must be NMC by definition (the only cycles with incoming edges have vertices corresponding to $(2N-1)$-letter subwords of some $s^{\infty}$, the only cycles with outgoing edges have vertices corresponding to $(2N-1)$-letter subwords of some $p^{\infty}$, and these sets are disjoint by definition of $N$), which will complete the proof.

To justify the claim, we only need to show that no two $5N$-letter subwords of $y := p^{\infty} m s^{\infty}$ can be equal unless they are both part of $p^{\infty}$ or $s^{\infty}$. First, assume without loss of generality that $m$ is minimal, which implies that $p^{\infty} m_1$ and $m_{|m|} s^{\infty}$ are not $N$-periodic. Define the set $S = \{i \ \mid \ y(i) \neq y(i + N)\}$. Clearly $S$ is finite and contained in an interval of length $2N-1$ (for any such $i$, either $y(i)$ or $y(i+N)$ must be part of $m$). To any two $5N$-letter subwords $u = y(j) \ldots y(j+5N-1)$ and $v = y(k) \ldots y(k+5N-1)$ of $y$, we can associate the sets $S(u) := \{0 \leq i < 4N \ \mid \ u(i) \neq u(i+N)\}$ and $S(v) := \{0 \leq i < 4N \ \mid \ v(i) \neq v(i+N)\}$. These are just intersections of $S$ with an interval of length $4N$, and the reader may check that the only way for two such intersections to be equal is if they are empty, which happens only if $u,v$ were subwords of $p^\infty$ or $s^\infty$.

We have verified that each sequence $p^{\infty} m s^{\infty}$ induces an empty barbell consisting of a simple cycle induced by subwords of
$p^{\infty}$ with a single simple directed path to a simple cycle induced by subwords of $s^{\infty}$, which, as described above, shows that $G_{X,2N}$ has NMC, completing the proof.

\end{proof}

We can now prove our main result about $\Scal$.

\begin{theorem}\label{isodense}
Every NMC subshift $X$ is isolated in $\Scal$, and the set $\mathbf{NMC}$ of NMC subshifts is dense in $\Scal$.
Thus, $\mathbf{NMC}$ is residual in $\Scal$.
\end{theorem}

\begin{proof}
We first prove density. Consider any cylinder $C = [X,n]$ in $\Scal$ with associated Rauzy graph $G_{X,n}$. We will find a subshift $X^{\prime}$ contained in $C$ which has the structure described in Lemma~\ref{NMCstructure}. 

Define by $C_1, \ldots, C_k$ the irreducible components of $G_{X,n}$. Recall that we can assume without loss of generality that each $C_i$ is nontrivial, i.e. contains at least one edge, and that components can be connected by finite directed paths.


Now, to each edge $e$ in $G_{X,n}$, we associate a cycle or empty barbell as follows.

\begin{itemize}
\item If $e$ is part of a (nontrivial) irreducible component $C_i$, then $e$ is part of a simple cycle $K_e$.
\item If $e$ is not part of such a component, then it must be on a simple transition path from some $C_i$ to some $C_j$, and by extending maximally forwards and backwards, it is part of a simple directed path $P$ from a source to a sink. By the preceding argument, the initial vertex of $P$ is part of a simple cycle within a source, and the terminal vertex is part of a simple cycle within a sink. Define $B_e$ to be the empty barbell consisting of these two cycles, along with the simple transition path between them, which contains $e$.
\end{itemize}

Since each initial cycle of any $B_e$ is contained in a source and each terminal cycle of any $B_e$ is contained in a sink, there is no cycle which is the initial in some $B_e$ and terminal in some $B_{e'}$. Now we define a subshift $X' \subset X$ to be all sequences corresponding to a path contained entirely within some $K_e$ or $B_e$. It is clear that $X'$ has the structure from Lemma~\ref{NMCstructure}, so
$X'$ is an NMC subshift.
Finally, by Lemma~\ref{rauzy}, $X' \in C = [X,n]$. Since $C$ was arbitrary, we've shown that NMC subshifts are dense.

It remains to show that each NMC subshift $X$ is isolated in $\Scal$. Choose such a subshift $X$, with associated NMC Rauzy graph $G_{X,n}$.

By Lemma~\ref{NMC}, $G_{X,n}$ can be written as a finite union of isolated simple cycles and empty barbells; take $N$ to be larger than the lengths of all transition paths for empty barbells in $G_{X,n}$. We claim that any $Y \in [X, n+N]$ must be equal to $X$, which will clearly imply that $X$ is isolated.

Choose any $Y \in [X, n+N]$. Both $X$ and $Y$ must be subsets of $S([X,n])$, (recall $S([X,n])$ is the set of all sequences corresponding to paths in $G_{X,n}$). By the structure of $G_{X,n}$, all sequences in $S([X,n])$ are either induced by a cycle (and therefore periodic) or induced by an empty barbell (and therefore eventually periodic on both ends with a finite transition word of length less than $N$ between). By Lemma~\ref{rauzy}, every such periodic sequence coming from an isolated simple cycle in $G_{X,n}$ must be in both $X,Y$, since the only biinfinite path containing any edge of such a cycle is a biinfinite traversal of the cycle. If $C$ is a cycle in $G_{X,n}$ which is either initial or terminal within some barbell, then by Lemma~\ref{rauzy}, $X$ and $Y$ each contain some sequence corresponding to a biinfinite path containing an edge of $C$. All such paths contain a one-sided infinite traversal of $C$, and so by compactness, both $X$ and $Y$ contain all periodic sequences induced by $C$.

Since $Y \in [X, n+N]$, for all $i \leq N$, $X$ and $Y$ share the same sets of legal length-$i$ paths within biinfinite paths of $G_{X,n}$ corresponding to sequences in the subshifts $X,Y$. This immediately implies that every sequence in $S([X,n])$ induced by a barbell is either in both $X,Y$ or neither of $X,Y$, since it is the only sequence containing its transitional path, which has length less than or equal to $N$.

This means that every sequence in $S([X,n])$ is in neither/both of $X, Y$, and $S([X,n])$ contains $X$ and $Y$, so $X = Y$. Since $Y \in [X, n+N]$ was arbitrary, $X$ is isolated, and since it was an arbitrary NMC subshift, the proof is complete.

\end{proof}

\begin{corollary}\label{SFT}
Every NMC subshift is a countable shift of finite type.
\end{corollary}

\begin{proof}
Since any NMC subshift $X$ is isolated in $\Scal$, there exists a cylinder $C$ such that $C = \{X\}$. But $S(C)$ is always in $C$, so $X = S(C)$, which is a shift of finite type by definition. By definition, there exists $n$ so that $G_{X,n}$ has NMC. It is clear by Lemma~\ref{NMC} that such a graph has only countably many biinfinite paths, so $X$ is countable.
\end{proof}

By Theorem~\ref{isodense}, the question of whether sets are residual in $\Scal$ is somewhat degenerate; the countable set $\mathbf{NMC}$ of NMC subshifts is open and dense, so residual, and any dense set must contain $\nmc$, since $\nmc$ is precisely the set of isolated points in $\Scal$. So genericity of properties/sets in $\Scal$ is completely controlled by the set $\nmc$ of NMC subshifts.

The following result is somewhat obvious, but we state it for completeness.

\begin{corollary}\label{cor:ctblesftgeneric}
The set of countable shifts of finite type is residual in $\Scal$.
\end{corollary}

\begin{proof}
This is an immediate consequence of Theorem~\ref{isodense} and Corollary~\ref{SFT}.
\end{proof}

We obtain as corollaries two results from the literature. The first, which is immediate from Corollary~\ref{cor:ctblesftgeneric}, regards the genericity of zero entropy subshifts.

\begin{corollary}[\cite{Sigmund1971,FrischTamuzGenericEntropy}]\label{zeroS}
The set $\mathbf{Z}$ of zero entropy subshifts is residual in $\Scal$.
\end{corollary}

For the second, recall in~\cite{CyrKraLanguageStable} the class of language stable subshifts was introduced, motivated by the study of characteristic measures for subshifts. Since countable shifts of finite type have the language stable property, we also obtain as a corollary the genericity of language stable subshifts proved in~\cite{CyrKraLanguageStable}.

\begin{corollary}[\cite{CyrKraLanguageStable}]\label{LS}
The set of language stable subshifts is residual in $\Scal$.
\end{corollary}

We can also characterize word complexity for generic subshifts in $\Scal$.

\begin{corollary}\label{lincplx}
The set of subshifts with linear complexity is residual in $\Scal$.
\end{corollary}

\begin{proof}
By Lemma~\ref{NMCstructure}, every subshift in $\nmc$ is a finite union of periodic orbits (which trivially have bounded word complexity) and orbits of sequences of the form $p^{\infty} m s^{\infty}$. It's easily checked that the subshift of such sequences has linear complexity, since any subword of length $n$ is determined by the `phase shift' of the periodic portions at its beginning and end and location, if any, of the central $m$. So, each NMC subshift has linear complexity, and Theorem~\ref{isodense} completes the proof.
\end{proof}

\subsection{Automorphism groups of generic subshifts in $\Scal$}
In this section, we show that the structure of the automorphism group of a generic subshift in $\Scal$ is virtually free abelian of finite rank.

\begin{theorem}\label{thm:autosinS}
The automorphism group of an NMC subshift $X$ is an extension of a finite group by a finite rank free abelian group whose rank is the number of orbits of $X$. Consequently, the automorphism group of a generic subshift in $S$ is an extension of a finite group by a finite rank free abelian group, and hence is virtually free abelian of finite rank.
\end{theorem}
\begin{proof}
By Lemma~\ref{NMCstructure}, the set $\mathcal{O}(X)$ is finite. Since automorphisms of a system must permute orbits, there is a homomorphism $$\pi_{\mathcal{O}} \colon \aut(X,\sigma_{X}) \to \textnormal{Sym}(\mathcal{O}(X))$$
where $\textnormal{Sym}(\mathcal{O}(X))$ denotes the group of permutations on the set $\mathcal{O}(X)$. If an automorphism $\alpha$ is in the kernel of $\pi_{\mathcal{O}(X)}$ then it maps each orbit to itself, so upon choosing representative points $\{x_{i}\}_{i \in \mathcal{O}(X)}$ for each orbit we have $\alpha(x_{i}) = \sigma^{k_{i}(\alpha)}(x_{i})$ for each $i \in \mathcal{O}(X)$. The automorphism $\alpha$ is determined by this set of integers $\{k_{i}(\alpha)\}_{i \in \mathcal{O}(X)}$, and it follows there is an injective homomorphism $\iota \colon \ker \pi_{\mathcal{O}(X)} \to \mathbb{Z}^{|\mathcal{O}(X)|}$ given by $\iota \colon \alpha \mapsto (k_{1}(\alpha),\ldots,k_{|\mathcal{O}(X)|}(\alpha))$. Thus $\ker \pi_{\mathcal{O}(X)}$ is free abelian.

To prove the claim regarding the rank, fix an orbit $j \in \mathcal{O}(X)$. By Lemma~\ref{NMCstructure}, this orbit is forward asymptotic to some periodic point and backward asymptotic to some periodic point, and we let $p_{j}$ denote the product of the periods of these two periodic points. We may define an automorphism $\beta_{j}$ of $X$ by acting as the identity on all orbits besides $j$, and acting by the $\sigma^{p_{j}}$ on the orbit $j$. Then $\iota(\beta_{j})$ is the vector consisting of all $0$s except for a $1$ in the $j$th component. Since $j \in \mathcal{O}(X)$ was arbitrary, it follows that the image of the injective homomorphism $\ker \pi_{\mathcal{O}(X)} \to \mathbb{Z}^{|\mathcal{O}(X)|}$ is rank $|\mathcal{O}(X)|$.
\end{proof}

\begin{remark}
The mapping class group of a generic subshift in $\Scal$ is also virtually free abelian of finite rank. This follows from genericity of NMC subshifts in $\Scal$ together with the fact (which we don't prove here) that the mapping class group of an NMC subshift is virtually free abelian of finite rank.
\end{remark}

\section{The space $\Scalprime$ of all non-isolated subshifts}\label{main2}


The previous section shows that NMC subshifts, which have quite degenerate dynamics, completely determine which properties are generic in $\Scal$. It is then natural to ask what happens if one removes them from $\Scal$, and so we make the following definition.


\begin{definition}
Denote by $\Scalprime$ the complement of the set of NMC subshifts in $\Scal$.
\end{definition}

Note that clearly the derived set of $\Scal$ is contained in $\Scal^{\prime}$. The following shows that $\Scal^{\prime}$ is in fact precisely the derived set of $\Scal$.

\begin{theorem}\label{perfect}
$\Scalprime$ is a perfect subset of $\Scal$.
\end{theorem}

\begin{proof}
Note that $\Scalprime$ is clearly closed since $\nmc$ was open in $\Scal$, so we need only show that $\Scalprime$ contains no isolated points.

Consider any subshift $X \in \Scalprime$ and any $n \in \mathbb{N}$. We will show that $|[X,n] \cap \Scalprime | \geq 2$, and hence $[X,n] \cap \Scalprime \neq \{X\}$. Consider the Rauzy graph $G_{X,n}$, which since $X \notin \nmc$, must contain a cycle $K$ with an incoming edge $f$ (with initial vertex not in $K$) and an outgoing edge $g$ (with terminal vertex not in $K$); without loss of generality we may assume that $K$ is simple. Denote by $P$ the simple subpath of $K$ from the terminal vertex of $f$ to the initial vertex of $g$. Clearly the subshift $S([X,n])$ consisting of all sequences corresponding to biinfinite paths in $G_{X,n}$ is in $[X,n]$. Define a subshift $Y$ as the subshift of all sequences corresponding to biinfinite paths in $G_{X,n}$ not containing $fPg$ as a subpath. It's clear that $Y \subsetneq S([X,n])$; since $G_{X,n}$ is essential, there is a biinfinite path in $G_{X,n}$ containing $fPg$, and the corresponding sequence is in $S([X,n])$ but not $Y$. We now show that $Y \in [X,n]$.

For every edge $e$ in $G_{X,n}$, since $G_{X,n}$ is essential, there is a biinfinite path $P_e$ in $G_{X,n}$ containing $e$. Since
$f,g$ are not part of $P$, any occurrences of $fPg$ in $P_e$ are nonoverlapping. Given this, we define $P_{e}^{\prime}$ to be the biinfinite path obtained by replacing every occurence of $fPg$ in $P_{e}$ (if any) by $fPKg$. Clearly $P'_e$ still contains $e$, and
does not contain $fPg$, so the corresponding sequence is in $Y$. Since $e$ was arbitrary, by Lemma~\ref{rauzy} we have $Y \in [X,n]$.

It remains only to show that $S([X,n])$ and $Y$ are both in $\Scalprime$, i.e. not NMC subshifts. Note that since $G_{X,n}$ is essential, for every $k$ there exists a biinfinite path $Q_k$ containing $fPK^{k} g$. By the same argument above, we may assume that $Q_k$ does not contain $fPg$, so that each $Q_k$ corresponds to a sequence $y_{k} \in Y \subset S([X,n])$. 
Now, Lemma~\ref{NMCstructure} implies that neither $Y$ nor $S([X,n])$ are NMC subshifts; they each contain $y_k$, which has a subword $fPK^{k} g$ which is not part of a one-sided periodic portion, and for large enough $k$, the length of $fPK^{k}g$ would exceed the maximum length of words in $M$ from Lemma~\ref{NMCstructure}, a contradiction.

We've shown that $|[X,n] \cap \Scalprime| \geq 2$, and since $X,n$ were arbitrary, the proof is complete.

\end{proof}

This leads to an interesting universality result for non-NMC shifts of finite type. In~\cite{Pelczynski1965}, Pe\l{}czy\'{n}ski proved that in the class of compact zero-dimensional metric spaces, there is a unique (up to homeomorphism) space which has a dense set of isolated points whose complement is homeomorphic to the Cantor set.
In fact, we now show that for any non-NMC shift of finite type $Y$, the space of subshifts of $Y$ is homeomorphic to the Pe\l{}czy\'{n}ski Space.

\begin{corollary}\label{cor:pelczynski}
If $Y$ is a shift of finite type which is not NMC, then the space $\Scal(Y)$ of subshifts of $Y$ is homeomorphic to the Pe\l{}czy\'{n}ski Space. In particular, the spaces of subshifts of non-NMC shifts of finite type are all homeomorphic to each other.
\end{corollary}
\begin{proof}
Let $Y \in \Scal$ be a shift of finite type which is not NMC. By Lemma~\ref{SFTclopen}, $\Scal(Y)$ is clopen, and it is obviously a zero-dimensional metric space as a subset of $\Scal$. It follows then from Theorem~\ref{isodense} that $\nmc \cap \Scal(Y)$ is dense in $\Scal(Y)$, and that $\nmc \cap \Scal(Y)$ consists of isolated points in $\Scal(Y)$. Thus by Pe\l{}czy\'{n}ski's uniqueness result from~\cite{Pelczynski1965}, it remains to show that in $\Scal(Y)$, the set
$E = \Scal(Y) \setminus \nmc = \Scalprime \cap \Scal(Y)$ is a Cantor set. It's clear that $E$ is a metric space, and it's nonempty since it contains $Y$. Moreover, $E$ is totally disconnected, and compact, since $\nmc \cap \Scal(Y)$ is open in $\Scal(Y)$. Finally, $E$ is perfect since it's the intersection of the open set $\Scal(Y)$ with $\Scalprime$, which is perfect by Theorem~\ref{perfect}. 
Altogether $E$ is a nonempty compact, totally disconnected, perfect metric space, and hence homeomorphic to the Cantor set.
\end{proof}

While genericity in $\Scalprime$ is not trivially controlled by a countable set as was the case in $\Scal$, there are still cylinders in $\Scalprime$ for which every subshift contained in the cylinder has somewhat degenerate dynamics. Consider, for instance, the subshift $X$ given by all nondecreasing sequences on the alphabet $\{0,1,2\}$ (e.g. $\ldots 000111222\ldots$). The subshift $X$ is not NMC (by Lemma~\ref{NMCstructure}), so $X \in \Scalprime$, but $X$ is countable, and all subshifts in $[X,2]$ share the set of $2$-letter subwords $\{00, 01, 02, 11, 12, 22\}$ with $X$ and so are contained in $X$. All such subshifts are countable, and not topologically transitive; topologically, they are each a countable union of compactifications of $\mathbb{Z}$.

In fact, we will show now that the union of a collection of similar `degenerate' cylinders is actually dense in $\Scalprime$. We first need some definitions.

\begin{definition}
A directed graph has the \textit{one middle cycle} property (or OMC) if there exists a unique simple cycle $K$ with both incoming and outgoing edges.
\end{definition}

We can again give an alternate definition of OMC graphs as unions of basic objects, for which we first need a definition.

\begin{definition}
An \textit{empty double barbell} is a graph consisting of three vertex-disjoint simple cycles and simple directed paths from the first cycle to the second and from the second to the third.
\end{definition}

We leave the proof of the following simple lemma to the reader.

\begin{lemma}\label{OMC}
If a graph has OMC, then it can be written as a finite (not necessarily disjoint) union of isolated simple cycles, empty barbells, and a single empty double barbell, where all isolated cycles and cycles within empty/double barbells are vertex-disjoint.

If a graph can be written as such a union where no cycle is both initial within any barbell and terminal within another, and where the central cycle in the empty double barbell is not part of any other empty barbell, then it has OMC.
\end{lemma}

An example of an OMC graph is given in Figure 3.
\begin{center}
\begin{figure}
\tikzset{every loop/.style={min distance=20mm,in=310,out=240,looseness=10}}
\tikzset{->-/.style={decoration={
  markings,
  mark=at position .5 with {\arrow{>}}},postaction={decorate}}}
\begin{tikzpicture}[scale=6,transform shape,
            > = stealth, 
            shorten > = 1pt, 
            auto,
            node distance = 3cm, 
            thick, 
            scale=.1
        ]
        \node[draw,fill,circle,inner sep=0pt,minimum size=4pt] (a) at (-2,3) {};
        \node[draw,fill,circle,inner sep=0pt,minimum size=4pt] (b) at (-0.5,2) {};
        \node[draw,fill,circle,inner sep=0pt,minimum size=4pt] (c) at (1,3) {};
        \node[draw,fill,circle,inner sep=0pt,minimum size=4pt] (d) at (-0.5,4) {};
        \node[draw,fill,circle,inner sep=0pt,minimum size=4pt] (e) at (3.5,3.8) {};
        \node[draw,fill,circle,inner sep=0pt,minimum size=4pt] (f) at (5.2,3.8) {};
        \node[draw,fill,circle,inner sep=0pt,minimum size=4pt] (g) at (4.4,2.6) {};
        \node[draw,fill,circle,inner sep=0pt,minimum size=4pt] (h) at (-1.5,-2.6) {};
        \node[draw,fill,circle,inner sep=0pt,minimum size=4pt] (i) at (-0.5,-1.7) {};
        \node[draw,fill,circle,inner sep=0pt,minimum size=4pt] (j) at (0.5,-2.6) {};
        \node[draw,fill,circle,inner sep=0pt,minimum size=4pt] (k) at (2.5,-1.4) {};
        \node[draw,fill,circle,inner sep=0pt,minimum size=4pt] (l) at (2.5,-3) {};
        \node[draw,fill,circle,inner sep=0pt,minimum size=4pt] (m) at (4.7,-1.8) {};
        \node[draw,fill,circle,inner sep=0pt,minimum size=4pt] (n) at (6.7,0.8) {};
        \node[draw,fill,circle,inner sep=0pt,minimum size=4pt] (o) at (6.7,-0.8) {};
        \path[->] (a) edge [bend right=30] node {} (b);
        \path[->] (b) edge [bend right=30] node {} (c);
        \path[->] (c) edge [bend right=30] node {} (d);
        \path[->] (d) edge [bend right=30] node {} (a);
        \path[->] (e) edge [bend left=30] node {} (f);
        \path[->] (f) edge [bend left=30] node {} (g);
        \path[->] (g) edge [bend left=30] node {} (e);
        \path[->] (h) edge [bend left=40] node {} (i);
        \path[->] (i) edge [bend left=40] node {} (j);
        \path[->] (j) edge [bend left=40] node {} (h);
        \path[->] (k) edge [bend right=60] node {} (l);
        \path[->] (l) edge [bend right=60] node {} (k);
        \path[->] (m) edge [loop right] node {} (m);
        \path[->] (b) edge [bend right=35] node {} (i);
        \path[->] (b) edge [bend right=30] node {} (k);
        \path[->] (b) edge [bend left=25] node {} (m);
        \path[->] (g) edge [bend right=10] node {} (i);
        \path[->] (g) edge [bend left=20] node {} (k);
        \path[->] (g) edge [bend left=15] node {} (n);
        \path[->] (n) edge [bend left=45] node {} (o);
        \path[->] (o) edge [bend left=45] node {} (n);
        \path[->] (o) edge [bend left=15] node {} (m);
    \end{tikzpicture}
    \caption{A graph having OMC (see the far right cycle).}
    \end{figure}
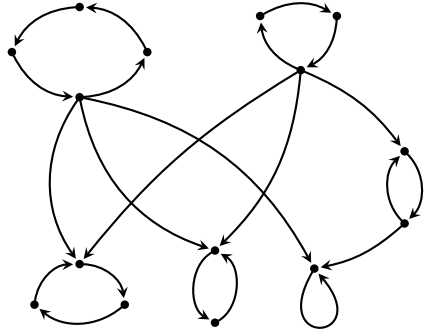
    \end{center}
\begin{definition}
A subshift $X$ has the \textit{one middle cycle} property (or OMC) if there exists $n$ for which $G_{X,n}$ has OMC.
\end{definition}

In some sense, the OMC subshifts are the simplest subshifts in $\Scalprime$ (i.e. the closest to the NMC subshifts removed in the definition of $\Scalprime$). As will often be the case, these simplest cases will in fact turn out to be generic (here in $\Scalprime$).

Let $\omc$ denote the set of OMC subshifts. 

\begin{theorem}\label{badcylinders}
If $C$ is a cylinder set in $\Scalprime$, then $C$ contains a nonempty cylinder set $D$ in $\Scalprime$ such that $D$ is contained in
$\omc$.
\end{theorem}

\begin{proof}
This proof is very similar to the proof of density in Theorem~\ref{isodense}. The only differences are that we know that every cylinder in $\Scalprime$ contains a non-NMC subshift (which therefore has a Rauzy graph with a cycle with both incoming/outgoing edges), and that we need to ensure that we create a subshift whose Rauzy graph has a single simple cycle with incoming/outgoing edges (rather than none).

Consider any cylinder $C = [X,n]$ in $\Scalprime$ with associated Rauzy graph $G_{X,n}$. We may assume without loss of generality that $X \in \Scalprime$, and so $G_{X,n}$ is not NMC. However, by Theorem~\ref{isodense}, there exists an NMC subshift $Y \in [X,n]$. By the proof of Lemma~\ref{NMCstructure}, there exist infinitely many $N$ for which $G_{Y,N}$ is NMC; choose such $N > n$.

The Rauzy graph $G_{X,n}$ contains a simple cycle $K$ with an incoming edge $f$ and outgoing edge $g$. 
Since $G_{X,n}$ is essential, we can extend $K$ backwards and forwards to obtain an empty double barbell subgraph $D$ of $G_{X,n}$ which has $K$ as its central cycle. Define $S = S(D)$; then $S$ consists of shifts of sequences of the form $p^{\infty}$, $m^{\infty}$, $s^{\infty}$, $p^{\infty} q m^{\infty}$, $m^{\infty} r s^{\infty}$, or
$p^{\infty} q m^n r s^{\infty}$ for some words $p, q, m, r, s$, where no two of $p^{\infty}, m^{\infty}$, and $s^{\infty}$ are in the same orbit.

Define $Y' = Y \cup S$. Consider $G_{Y^{\prime},N}$, which is just the union of $G_{Y,N}$ with any new vertices/edges corresponding to
$(N-1)$- and $N$-letter subwords of sequences in $S$. By the described structure of $S$, this is just the union of $G_{Y,N}$ with an empty double barbell $D'$ (which is not equal to $D$ since $N > n$.)

We note that the initial cycle in $D'$ cannot be the terminal cycle of any empty barbell in $G_{Y,N}$; empty barbells used in the construction of $Y$ always ended at sinks of $G_{X,n}$, which had no outgoing edges, and the same is true in $G_{Y,N}$ since $N > n$. Similarly, the terminal cycle of $D'$ cannot be the beginning of any empty barbell in $G_{Y,N}$ and the central cycle of $D'$ can be neither the initial nor terminal cycle of any empty barbell in $G_{Y,N}$. Therefore, $G_{Y^{\prime},N}$ is OMC by Lemma~\ref{OMC}.

But every subshift $Z \in [Y',N] \subset [X,n] = C$ has $G_{Z,N} = G_{Y^{\prime},N}$, so every such $Z$ is OMC. Since $C$ was arbitrary, the proof is complete.

\end{proof}

Since OMC subshifts are in a sense degenerate (they are never transitive, always zero entropy), again most meaningful dynamical properties are not generic in $\Scalprime$.

\begin{corollary}
The sets $\mathbf{M}$ of minimal subshifts and $\mathbf{UE}$ of uniquely ergodic subshifts are nowhere dense in $\Scalprime$.
\end{corollary}

\begin{proof}
We simply note that any subshift in $\omc$ contains two different periodic orbits as proper subsets, and so is neither minimal nor uniquely ergodic. Theorem~\ref{badcylinders} then completes the proof.
\end{proof}

\begin{corollary}\label{OMCgen}
$\omc$ is residual in $\Scalprime$.
\end{corollary}

\begin{proof}
Simply note that $\omc$ contains the union over all cylinders in $\Scalprime$ of the subcylinders guaranteed by Theorem~\ref{badcylinders}; this set is open and dense by definition.
\end{proof}



\begin{theorem}
The uncountable subshifts are nowhere dense in $\Scalprime$, and therefore the countably infinite subshifts are residual in $\Scalprime$.
\end{theorem}

\begin{proof}

The first statement follows immediately from Theorem~\ref{badcylinders} and the fact that all OMC subshifts are countable; the second statement then follows since all subshifts in $\Scalprime$ are infinite.

\end{proof}

\subsection{Complexity for subshifts in $\Scalprime$}

Generic subshifts in $\Scalprime$ no longer have linear complexity, but their complexities still grow rather slowly. By Corollary~\ref{OMCgen}, we can prove that generic subshifts in $\Scalprime$ satisfy various complexity bounds by verifying them for subshifts in $\omc$.

It turns out that every nonempty cylinder in $\Scalprime$ contains subshifts with complexity arbitrarily close to linear.

\begin{lemma}\label{OMCslow}
For any unbounded increasing $f: \mathbb{N} \rightarrow \mathbb{R}$ and any nonempty cylinder $C$ in $\Scalprime$, there exists
$Y \in \Scalprime \cap C$ with $\frac{c_{Y}(m)}{mf(m)} \rightarrow 0$.
\end{lemma}

\begin{proof}
Fix any such function $f$. We assume without loss of generality that $C = [X,n]$ for some OMC subshift $X$. By increasing $n$ if necessary, we may assume that $G_{X,n}$ is OMC. By Lemma~\ref{OMC}, we can write $G_{X,n}$ as a union of isolated simple cycles $K_i$, empty barbells $B_j$, and an empty double barbell $D$, where all isolated cycles and cycles within empty/double barbells are vertex-disjoint, where
no cycle is both initial within any barbell and terminal within another, and where the central cycle in the empty double barbell is not part of any other empty barbell.

For any infinite $R \subseteq \mathbb{N}$, we define a subshift $Y(R)$ consisting of all sequences corresponding to a biinfinite path contained within some $K_i$, $B_j$, or $D$, with the additional restriction that the biinfinite path in $D$ is not allowed to contain a subpath of the form $fPK^n g$ for $P$ a proper subpath in $K$ and $n \notin R$. Note that for all $r \in R$, $Y$ contains a sequence
 $y_r$ corresponding to the unique (up to shifts) biinfinite path in $G_{X,n}$ containing $fPK^r g$. If $Y(R)$ were an NMC subshift, this would contradict Lemma~\ref{NMCstructure}; $y_r \in Y(R)$ contains a subword corresponding to $fPK^{r} g$ which is not part of a one-sided periodic portion, and for large enough $r$, the length of this subword would exceed the maximum length of words in $M$, a contradiction. Therefore, $Y(R) \in \Scalprime$.


We now wish to bound from above the word complexity function of $Y(R)$. Firstly, just as in the proof of Corollary~\ref{lincplx}), the subshift of sequences corresponding to biinfinite paths contained in any $K_i$ or $B_j$ has linear word complexity. Therefore, it suffices to treat only sequences corresponding to biinfinite paths in $D$ which do not contain $fPK^n g$ for $n \notin R$.
By definition of the Rauzy graphs, the number of $m$-letter words in
$Y(R)$ is equal to the number of such paths of length $m-n+1$ in $D$, so it suffices to estimate the number of paths of various lengths in $D$ which do not contain $fPK^n g$ for $n \notin R$.

First we denote the initial/terminal cycles in $D$ by $B$ and $E$ (for beginning/end), the transition paths by $I$ (from $B$ to $K$) and $J$ (from $K$ to $E$), the subpath of $K$ from the end of $I$ to the beginning of $J$ by $P$, and assume without loss of generality that $B$ and $E$ are oriented so that $B$ begins and ends at the initial vertex of $I$ and $E$ begins and ends at the terminal vertex of $J$. Then the paths we wish to consider are of one of the following forms:

\begin{enumerate}
\item subpath of $B^{\infty}$
\item subpath of $E^{\infty}$
\item subpath of $M^{\infty}$
\item subpath of $B^{\infty} I P$ containing at least one edge of $I$
\item subpath of $J E^{\infty}$ containing at least one edge of $J$
\item subpath of $B^{\infty} I P K^{\infty}$ containing all of $I$
\item subpath of $K^{\infty} J E^{\infty}$ containing all of $J$
\item subpath of $B^{\infty} I PK^r J E^{\infty}$ for some $r \in R$ containing all of $P K^r$
\end{enumerate}

It's immediate that the number of paths of length $n$ in any of the first seven categories is either bounded in $n$ or linear in $n$. So, it suffices to treat only category (8). For any length $n$, there is a path in category (8) of length $n$ for $r \in R$ if and only if $r|K| \leq n$, and the number of such paths is $n + 1 - r|K|$. Therefore, the number of total paths in category (8) is
\begin{equation}\label{DBcount}
p_n := \sum_{r \in R, r \leq n|K|^{-1}} (n+1 - r|K|) \leq n \cdot |R \cap \{1, \ldots, n\}|.
\end{equation}
Therefore, as long as we choose our infinite $R$ in such a way that $|R \cap \{1,\ldots,n\}| < \sqrt{f(n)}$ for all $n$, then
$\frac{p_n}{nf(n)} \rightarrow 0$, and as argued earlier, this implies that $\frac{c_{Y(R)}(m)}{mf(m)} \rightarrow 0$.

\end{proof}

\begin{corollary}\label{gencplx}
For any unbounded increasing $f: \mathbb{N} \rightarrow \mathbb{R}$, the set
\[
\mathbf{S}_f := \{X \ \mid \ c_{X}(n) < nf(n) \textrm{ for infinitely many } n\}
\]
is residual in $\Scalprime$.
\end{corollary}

\begin{proof}
That $\mathbf{S}_f$ is a $G_{\delta}$ follows from Lemma~\ref{cplxthm}, and its density follows from Corollary~\ref{OMCgen} and Lemma~\ref{OMCslow}.
\end{proof}

\begin{remark}
We note that the choice of the much weaker complexity condition in Corollary~\ref{gencplx} is due to the fact that
it is not clear whether the set of subshifts satisfying $\frac{c_{X}(n)}{nf(n)} \rightarrow 0$ is a $G_{\delta}$.
\end{remark}

The strongest upper bound that we can give which holds over all lengths is quadratic.

\begin{lemma}\label{OMCquad}
For every OMC subshift $X$, there exist constants $C,D,E$ so that $c_{X}(n) \leq Cn^2 + Dn + E$ for all $n$.
\end{lemma}

\begin{proof}
Choose any OMC subshift $X$ and $n$ for which $G_{X,n}$ is OMC. It's easily checked that every biinfinite path on $G(X,n)$ is either periodic (i.e. repeated traversal of a cycle), the unique path (up to shifts) on an empty barbell, or contained within an empty double barbell. As before, the complexity of the subshift of points corresponding to paths of the first two types is easily checked to be linear. So it suffices to bound the complexity of the subshift of points corresponding to paths on an empty double barbell $D$.

Fix any double barbell $D$, and let $P(D)$ denote the set of sequences corresponding to all paths contained in $D$ (whether or not these sequences are in $X$). Then the complexity of $P(D)$ is clearly an upper bound for that of $P(D) \cap X$, so it suffices to bound the complexity function for $P(D)$.

For this, the same estimates from the proof of Lemma~\ref{OMCslow} can be used. In particular, $P(D)$ is equal to $Y(\mathbb{N})$ defined in that proof, and so for all $m$ (here we use the fact that paths of length $m$ are in bijective correspondence to words of length $m + n$ in the language of $P(D)$),
\[
c_{P(D)}(m + n - 1) = p_{m} \leq m |\mathbb{Z} \cap \{1, \ldots, m\}| = m^2.
\]

Since there are only finitely many empty double barbells $D$ which are subgraphs of $G_{X,n}$, the proof is complete.

\end{proof}

\begin{corollary}\label{quadres}
For a generic subshift in $\Scalprime$, there exist constants $C,D,E$ so that $c_{X}(n) \leq Cn^2 + Dn + E$ for all $n$.
\end{corollary}

\begin{proof}
This follows immediately from Corollary~\ref{OMCgen} and Lemma~\ref{OMCquad}.
\end{proof}

The following corollary is now immediate.

\begin{corollary}
The set $\mathbf{Z}$ of zero entropy subshifts is residual in $\Scalprime$.
\end{corollary}

\begin{remark}
We can theoretically find the automorphism group for a generic subshift in $\Scalprime$ by finding the structure of such groups for OMC subshifts. This turns out to be quite technical, and so we omit a formal proof here. However, roughly speaking, such an automorphism group will come from the extension of a finite group by some $\mathbb{Z}^N$ (just as in the NMC case), along with additional components coming from semidirect products of sums of countably many copies of $\mathbb{Z}$ with permutations of $\mathbb{N}$ with finite support. These both come from how an automorphism can act on orbits corresponding to paths on a double barbell; first note that any such orbit is determined by a number of traversals of the central cycle. Any automorphism is a sliding block code with fixed window size, and so can only freely permute/shift orbits corresponding to small enough numbers of cycle traversals; these permutations/shifts give rise to the direct sums of countably many copies of $\mathbb{Z}$ (coming from shifts) and permutations of $\mathbb{N}$ with finite support
(coming from permutations) mentioned above. We leave an exact computation of these automorphism groups to the interested reader.

\end{remark}

\section{The space $\Tcalprime$ of infinite transitive subshifts}\label{trans}
The last sections showed that in both $\Scal$ and $\Scalprime$, while the dynamics of a generic subshift are rather simple, they are also rather degenerate. In particular, in both $\Scal$ and $\Scalprime$ a generic subshift consists of sequences biasymptotic to periodic sequences. It is natural then to restrict ourselves to classes of subshifts possessing properties often of dynamical interest. Thus, in the remaining two sections we will focus on the space of transitive subshifts, and the space of totally transitive subshifts, respectively.

In these settings, there will no longer be degenerate subshifts which control genericity, and in fact a generic subshift in these settings is dynamically much more interesting. We will still show however that, in some sense, generic subshifts are as `simple as possible.'

We note that our consideration of the transitive and totally transitive classes is analogous to the path taken by Hochman in~\cite{HochmanGeneric}, where the subspaces of transitive and totally transitive systems within the space of homeomorphisms of the Cantor set are studied.

We begin now with $\Tcal$, the set of transitive subshifts in $\Scal$. It turns out (though we don't supply proofs here) that periodic orbits in $\Tcal$ are dense and isolated just as the set $\nmc$ was in $\Scal$, so in order to get any meaningful results, we need to remove these isolated points as before. Thus the main object of our study in this section will be the space $\Tcalprime = \Tcal \cap \Scalprime$. 

\begin{definition}
Denote by $\Tcalprime$ the intersection of $\Tcal$ with the complement of $\nmc$ in $\Scal$.
\end{definition}

As noted earlier, for transitive subshifts, being isolated is equivalent to being finite, and hence $\Tcalprime$ is precisely the set of infinite transitive subshifts in $\Scal$.

We note that $\Tcalprime$ is not closed in $\Scal$. To see this, for every $n$, define $X_n$ to be the subshift
of all sequences created from biinfinite concatenations of the words $0^n 1^n$ and $0^{n+1} 1^n$. Then each $X_n$ is in
$\Tcalprime$, but it is not hard to check that their limit is the union of the orbits of $0^{\infty}, 1^{\infty}, 0^{\infty} 1^{\infty}$, and $1^{\infty} 0^{\infty}$, which is not transitive. Since we make use of the Baire Category Theorem (to guarantee that the intersection of residual sets is residual), from now on we will work in the closure $\Tcalprimebar$. However, our results regarding genericity in $\Tcalprimebar$ apply equally well to $\Tcalprime$, due to the following.

\begin{lemma}\label{TGdelta}
$\Tcalprime$ is a dense $G_{\delta}$ in $\Tcalprimebar$.
\end{lemma}

\begin{proof}

We first note that $\Tcalprime$ is dense in $\Tcalprimebar$ by definition. Also, $\Tcalprime = \Tcal \cap (\nmc)^c = \Tcal \cap \bigcap_{X \in \nmc} \{X\}^c$. Since each singleton $\{X\}$ is closed in the metric space $\mathbf{S}$, it then suffices to show that $\Tcal$ is a $G_{\delta}$ in $\Scal$.

We note that $X$ is transitive if and only if for all $k$, there exists $n$ so that for every pair $u,v$ of $k$-letter words occurring within words in $L_n(X)$, there exists $w \in L_n(X)$ containing an occurrence of $u$ to the left of an occurrence of $v$.
For any $k, n, \A$, denote by $T_{k,n,\A}$ the set of subsets of $\A^n$ with this property. Then, the set $\Tcal$ of transitive subshifts can be written as
\[
\bigcap_{k \in \mathbb{N}} \bigcup_{\substack{n \in \mathbb{N}\\ \A \subset \mathbb{Z}, |\A| < \infty}} \{X \in \Scal \mid L_n(X) \in T_{k,n,\A}\} =
\bigcap_{k \in \mathbb{N}} \bigcup_{\substack{n \in \mathbb{N}\\ \A \subset \mathbb{Z}, |\A| < \infty}} \bigcup_{\substack{X \in \Scal\\ L_n(X) \in T_{k,n,\A}}} [X,n],
\]
which is clearly a $G_{\delta}$.

\end{proof}


Thus, a property $P$ is generic in $\Tcalprimebar$ if and only if it is generic in $\Tcalprime$, and henceforth we will simply state our results for $\Tcalprimebar$. We first show that $\Tcalprimebar$ has no isolated points.


\begin{lemma}\label{perfect2}
$\Tcalprimebar$ is a perfect subset of $\Scal$.
\end{lemma}

\begin{proof}

The proof of this is nearly identical to that of Theorem~\ref{perfect}. We simply need to show that if the shift $X$ in that proof is assumed transitive, then the subshift $Y$ constructed to belong to a cylinder $[X,n]$ for an arbitrary $n \ge 1$ is also transitive. Recall that $Y$ was constructed by forbidding a single path $fPg$ on a Rauzy graph, but there existed a cycle $K$ so that the path $fPKg$, with the same initial/terminal vertices as $fPg$, was still allowed.

Consider any two words $u, v \in L(Y)$; they correspond to paths $Q, R$ in $G_{X,n}$ which do not contain $fPg$. By irreducibility of $G_{X,n}$, there exists a biinfinite path $P'$ in $G_{X,n}$ containing an occurrence of $Q$ before an occurrence of $R$. By replacing $fPg$ by $fPKg$ in the path $P'$, we arrive at a path $P''$ in $G_{X,n}$ containing no $fPg$. Also, $P''$ still contains an occurrence of $Q$ before an occurrence of $R$ (the introduction of $K$ into $fPg$ does not change $Q$ or $R$). Finally, $P''$ corresponds to a point of $Y$ containing an occurrence of $u$ before an occurrence of $v$; since $u,v$ were arbitrary, $Y$ is transitive.

The remainder of the proof of Theorem~\ref{perfect} demonstrates that all nonempty cylinders in $\Tcalprimebar$ have at least two subshifts, so no subshift in $\Tcalprimebar$ is isolated.
\end{proof}

We will now prove a theorem that serves as our main tool for showing that various sets are dense in $\Tcalprimebar$.
Before the theorem, we set some notation. Suppose $\mathcal{A}$ is an alphabet and $\tau \colon \{0,1\} \to \mathcal{A}^{\ell}$ is an injective map. For any point $y = \ldots y_{-1}.y_{0}y_{1} \ldots \in \{0,1\}^{\mathbb{Z}}$ we define $\tau(y) \in \mathcal{A}^{\mathbb{Z}}$ by $\tau(y) = \ldots \tau(y_{-1}).\tau(y_{0})\tau(y_{1})\ldots$. For any subshift $Y \subset \{0,1\}^{\mathbb{Z}}$, we define $\tau^{*}(Y)$ to be the subshift $\{\sigma^{i}\tau(y) \mid y \in Y, 0 \le i < \ell \}$ (note that, as sets, $\tau^{*}(Y)$ is generally larger than $\tau(Y)$, since $\tau(Y)$ is generally not shift-invariant). In this way, $\tau$ induces a map from subshifts $Y \subset \{0,1\}^{\mathbb{Z}}$ to subshifts $\tau(Y) \subset \mathcal{A}^{\mathbb{Z}}$.

\begin{theorem}\label{letword}
For any nonempty cylinder $C = [X,n]$ in $\Tcalprimebar$ (with associated alphabet $\A_C$), there exists $\ell \geq n$ and an injective map $\tau: \{0,1\} \rightarrow \A_C^\ell$ so that for any subshift $Y \in \Tcalprime$ contained in $\{0,1\}^{\mathbb{Z}}$, 
we have $\tau^{*}(Y) \in C$. In addition, $\tau^{*}(Y)$ is uniquely decipherable in the sense that every word in $L(\tau^{*}(Y))$ of length at least $3\ell-1$ can be written as $p \tau(v) s$ in exactly one way for some $v \in L(Y)$, $p$ a proper prefix of some $\tau(a)$, and $s$ a proper suffix of some $\tau(a)$.
\end{theorem}

\begin{proof}
Choose any cylinder $C = [X,n] \cap \Tcalprimebar$ as in the theorem. Since $C$ is nonempty and has nontrivial intersection with $\Tcalprime$, we can without loss of generality assume that $X$ is transitive and infinite. Then the graph $G_{X,n+1}$ is irreducible and not a cycle. First, by irreducibility of $G_{X,n+1}$, there exists a cycle $K$ which visits all vertices of $G_{X,n+1}$. Without loss of generality, we assume that $K$ is minimal with respect to containment. This means that some vertex appears only once in $K$. To see this, note that if all vertices appeared twice in $K$, then removal of any minimal subcycle would yield a new $K$ still containing all vertices (since a minimal cycle cannot contain any vertex twice), contradicting minimality of $K$. Now, we break into two cases. We let $|K|$ denote the number of edges in $K$ and let $V$ denote the vertex set of $G_{X,n+1}$.

\begin{itemize}
\item If $|K| > |V|$, then denote by $v$ any vertex appearing exactly once in $K$, and reorder $K$ so that it begins and ends with $v$
(and does not pass through $v$ at any other time). Then, choose any proper subcycle $K'$ of $K$ (which must exist by the Pigeonhole Principle). It cannot contain $v$ (since $v$ appeared only at the beginning and end of $K$). Finally, reorder $K$ so that it begins and ends with whichever vertex $w$ is at the start/end of $K'$.

\item If instead $|K| = |V|$, then $K$ contains every vertex exactly once. As noted earlier, $G_{X,n+1}$ is not a cycle, so
$G_{X,n}$ contains some edge $e$ not in $K$, with initial/terminal vertices $v', v''$. By the definition of $G_{X, n+1}$, there cannot be two edges from $v'$ to $v''$, so the subpath $P$ of $K$ from $v'$ to $v''$ has length at least $2$. So, we can replace $P$ in $K$ by $e$, yielding a cycle $K'$ in $G_{X,n+1}$ which does not contain all vertices. Denote by $v$ a vertex missing from $K'$, and reorder $K$ so that it begins and ends with whichever vertex $w$ is at the start/end of $K'$.
\end{itemize}

In each case, we know that there are vertices $v \neq w$ so that $K$ and $K'$ both start/end with $w$, $v$ appears only once in $K$, and $v$ does not appear in $K'$. Now, we define two cycles $L = KKK'K'K'$, $L' = KK'KK'K'$, and note that $|L| = |L'|$; denote their common value by $\ell$. We note that each contains $v$ exactly twice, separated by distance $|K|$ in $L$ and by $|K| + |K'|$ in $L'$.

By the usual Rauzy graph correspondence, $L$ and $L'$ correspond to words $t,t'$ of length $\ell + n$ in $L(X)$ which each contain every $n$-letter word in $L(X)$ (since $K, K'$ contained all vertices in $G_{X, n+1}$) and which begin and end with the same $n$-letter word $w$. Denote by $u,u'$ the truncations of
$t,t'$ obtained by removing their terminal $w$s. Then define $\tau: 0 \mapsto u, 1 \mapsto u'$. The reader may check that for any $x \in \{0,1\}^{\mathbb{Z}}$, $\tau(x)$ is a sequence corresponding to a biinfinite concatenation of $L, L'$, which is a biinfinite path in $G_{X,n+1}$, and so $\tau(x) \in S(C)$. Moreover, since $\tau(x)$ contains some concatenation in $\{u,u'\}^2$, it contains either $u$ or $u'$, and so contains all $n$-letter words in $L(X)$.


Then, for any subshift $Y$ on $\{0,1\}$, $\tau^{*}(Y)$ is in $[X,n]$ by Lemma~\ref{rauzy}. Finally, it is routine to check that since $Y \in \Tcalprime$, $\tau^{*}(Y) \in \Tcalprime$ as well, so $\tau^{*}(Y) \in C$. It remains only to prove the claimed unique decipherability. For this, we consider distances between occurrences of $v$ (viewed as a word in $L_n(X)$) in points of $\tau^{*}(Y)$. Such occurrences appear exactly once in each of $t$ and in $t'$, at distances $|K|$ and $|K| + |K'|$ respectively. These distances appear nowhere else in points of $\tau^{*}(Y)$; all other separations between nearest $v$s are at least $|K| + 2|K'|$.

Suppose that $x$ is a word in $L(\tau^{*}(Y))$ of length at least $3\ell-1$. It must contain either $u$ or $u'$, and therefore contains a pair of $v$s separated by distance either $|K|$ or $|K| + |K'|$. This pair uniquely determines its containing $t$ or $t'$ and the location, and so the representation $x = p \tau(v) s$ is unique as claimed.

\end{proof}

We now have a strategy for showing that properties are generic in $\Tcalprimebar$. We can show that a property corresponds to a $G_{\delta}$ set directly, and as long as the property holds for a single $\{0,1\}$ subshift and is preserved under any $\tau$ as in Theorem~\ref{letword}, it is automatically dense in $\Tcalprimebar$. 

In fact the unique decipherability from Theorem~\ref{letword} is quite powerful; in particular, it implies via the following that any cylinder in $\Tcalprimebar$ contains an entire subcylinder each of whose subshifts are of the form $\tau^{*}(Y)$ (for $\tau, Y$ as in Theorem~\ref{letword}).

\begin{theorem}\label{letwordcyl}
For any nonempty cylinder $C = [X,n]$ in $\Tcalprimebar$ (with associated alphabet $\A_C$) and $\tau: \{0,1\} \rightarrow \A_C^\ell$ defined as in Theorem~\ref{letword}, the collection $\{\tau^{*}(Y) \mid Y \subset \{0,1\}^{\mathbb{Z}}\}$ contains a nonempty subcylinder of $C$.
\end{theorem}

\begin{proof}
Choose $X, n, C$, and $\tau$ as in Theorem~\ref{letword}. We claim that
\begin{equation}\label{mainletwordcyl}
[\tau^{*}(\{0,1\}^{\mathbb{Z}}), 3\ell] \subset \{\tau^{*}(Y) \mid  Y \subset \{0,1\}^{\mathbb{Z}}\},
\end{equation}
which will complete the proof, since the right-most set is a subset of $C$ by Theorem~\ref{letword}.

To see this, choose any $Y \in [\tau^{*}(\{0,1\})^{\mathbb{Z}}, 3\ell]$, and any $y \in Y$. By definition, every $(3\ell-1)-$ and $(3\ell)$-letter subword of $y$ is a subword of some $\tau(x)$, and so can be written in a unique way as $p \tau(w) s$ for some $w \in \{0,1\}^*$, $p$ proper suffix of some $\tau(a)$, and $s$ proper prefix of some $\tau(b)$. For the rest of this proof, we refer to such a representation of a word as a $\tau$-decomposition.

We will prove by induction that all subwords of $y$ of any length $M \geq 3\ell - 1$ have a unique $\tau$-decomposition. Our base case will be $M = 3\ell - 1$ and $M = 3\ell$, which has already been shown by the above. Now, assume that the claim is true for subwords of length $M$ and $M+1$ for $M \geq 3\ell-1$.

Now, choose any subword $v$ of $y$ of length $M+2$, and write $v = azb$. By the inductive hypothesis for $M+1$, $az$ and $zb$ have unique $\tau$-decompositions $az = p \tau(w) s$ and $zb = p' \tau(w') s'$. Clearly $v = azb$ cannot have two $\tau$-decompositions, since this would contradict uniqueness of the $\tau$-decompositions of $az, zb$. So it suffices to show that $v$ has at least one $\tau$-decomposition.

Each of $az = p \tau(w) s$ and $zb = p' \tau(w') s'$ yields a decomposition of $z$ (by removing the first letter of $p \tau(w) s$ and the last of $p' \tau(w') s'$, and these must be equal by the inductive hypothesis for $M$. 

If $p$ and $s'$ are nonempty, then denote by $\overline{p}$ and $\overline{s'}$ the words obtained by deleting their first/last letters respectively. Then the decomposition of $z$ must be $\overline{p} \tau(w) s = p' \tau(w') \overline{s'}$, meaning that $p' = \overline{p}$, $w = w'$, $s = \overline{s'}$. But this means that $azb = p \tau(w) s'$, yielding the desired $\tau$-decomposition.
If $p$ is empty but $s$ is not, then by similar reasoning, $p'$ must be the word obtained by removing the first letter of $\tau(w_1)$, $w'$ is the word obtained by removing the first letter $w_1$ of $w$, and $azb = \tau(w) s'$. Similarly, if $s'$ is empty but $p$ is not, then $s$ must be the word obtained by removing the last letter of $\tau(w'_{|w'|})$, $w$ is the word obtained by removing the last letter $w'_{|w'|}$ of $w'$, and $azb = \tau(w') s'$. Finally, if both $p, s'$ are empty, then $p'$ must be the word obtained by removing the first letter of $\tau(w_1)$, $s$ must be the word obtained by removing the last letter of $\tau(w_{|w|})$,
$w_2 \ldots w_{|w|} = w'_1 \ldots w'_{|w'| - 1}$, and $azb = \tau(\overline{w})$, where $\overline{w} = w'_1 w = w' w_{|w|}$. In every case, $azb$ has a $\tau$-decomposition, completing the proof of the claim by induction.

Now, by taking limits, it's clear that $y$ itself can be written as $\sigma^i \tau(x)$ for some $0 \leq i < \ell$ and $x \in \{0,1\}^{\mathbb{Z}}$. Since $y$ was arbitrary and $Y$ is a subshift, $Y = \tau^{*}(X)$ for some subshift $X \subseteq \{0,1\}^{\mathbb{Z}}$.

\end{proof}

Though the next two results will be subsumed later by the fact that a generic subshift in $\Tcalprimebar$ is regular Toeplitz
(Theorem~\ref{toegen}), we provide short proofs to illustrate this technique. 

\begin{theorem}\label{zeroT}
The set $\mathbf{Z}$ of zero entropy subshifts is residual in $\Tcalprimebar$.
\end{theorem}

\begin{proof}
It's clear that zero entropy is preserved under any letter-to-word substitution $\tau$, so by Theorem~\ref{letword}, it suffices to prove that the set of zero entropy subshifts is a $G_{\delta}$. For this, it clearly suffices to show that for any $\epsilon$, the set $\mathbf{S}_{\epsilon}$ of subshifts with entropy at least $\epsilon$ is closed. A proof of this is already given in~\cite{Sigmund1971}, but we rewrite it here for completeness.

Fix any $\epsilon$. By the definition of topological entropy, a subshift $X$ has $h(X) \geq \epsilon$ if $c_{X}(n) \geq e^{n\epsilon}$ for all $n$, and by Fekete's subadditivity lemma, the converse is also true. Therefore, $\mathbf{S}_{\epsilon}$ can be written as
\[
\bigcap_{n \in \mathbb{N}} \{X \in \Scal \mid c_{X}(n) \geq e^{n\epsilon}\},
\]
and each set $\{X \in \Scal \mid c_{X}(n) \geq e^{n\epsilon}\}$ is clearly closed.

\end{proof}

\begin{theorem}\label{minimal}
The set $\mathbf{M}$ of minimal subshifts is residual in $\Tcalprimebar$.
\end{theorem}

\begin{proof}
It's clear that minimality is preserved under any letter-to-word substitution $\tau$, so again it suffices to prove that the set of minimal subshifts is a $G_{\delta}$. 
From the definition, we know that a subshift $X$ is minimal if and only if for every $k$,
there exists $n$ so that all words in $L_n(X)$ contain all words in $L_k(X)$ as subwords. Since languages are factorial and extendable, we can rephrase this slightly: $X$ is minimal if and only if for all $k$, there exists $n$ so that every word in $L_n(X)$ contains all $k$-letter words which appear as subwords of some word in $L_n(X)$. For every $k, n, \A$, define $M_{k,n,\A}$ to be the set of subsets of $\A^n$ with the property above.

Then, the set of minimal subshifts can be written as
\[
\bigcap_{k \in \mathbb{N}} \bigcup_{\substack{m,n \in \mathbb{N},\\ \A \subset \mathbb{Z}, |\A| < \infty}} \{X \in \Scal \mid L_n(X) \in M_{k,n,\A}\} =
\bigcap_{k \in \mathbb{N}} \bigcup_{\substack{m,n \in \mathbb{N},\\ \A \subset \mathbb{Z}, |\A| < \infty}} \bigcup_{\substack{X \in \Scal \\ L_n(X) \in M_{k,n,\A}}} [X,n],
\]
which is clearly a $G_{\delta}$.

\end{proof}

\subsection{Regular Toeplitz subshifts in $\Tcalprimebar$}
Our first main goal is to show that a generic subshift in $\Tcalprimebar$ is regular Toeplitz. While this is in contrast to the more degenerate systems that appeared generically in $\Scal$ and $\Scalprime$, it continues the theme of generic systems being in some the sense the ``simplest possible.''

\begin{definition}
A sequence $x \in \A^\zz$ is \textit{regular Toeplitz} if for every $\epsilon > 0$, there exists $n \in \zz$ and a set
$S \subset [0,n)$ of integers for which $|S| > (1-\epsilon)n$ and, for each $s \in S$, $x(s + in)$ is independent of $i \in \zz$, i.e. takes the same value for every $i$. A \textit{regular Toeplitz subshift} is the orbit closure of any regular Toeplitz sequence.
\end{definition}

The difficulty in showing that regular Toeplitz subshifts are generic is that the definition above involves an infinite amount of information (infinite arithmetic progression), and so a priori cannot be represented as a $G_{\delta}$ condition.
However, with a little effort, it turns out that the definition can be rephrased in a finite way, which we present as an auxiliary lemma. We first need some definitions.

\begin{definition}
For any word $w = avb$ of length at least $2$ ($a, b \in \mathcal{A}$), write $p(w) = av$ and $s(w) = vb$.
For $\epsilon > 0$ and $n,N,k \in \mathbb{N}$, define a set $W \subset \A^N$ to be \textit{$(\epsilon, n, N, k)$-regular Toeplitz compatible} if, for every $w \in W$ and each $u \in \{w, p(w), s(w)\}$, there exists a set $S \subset [0,n)$ with $|S| = k$ so that
$u(s + in)$ is independent of $i \in [0, \lfloor \frac{N - s}{n} \rfloor)$ if and only if $s \in S$. In other words, each $u$ which is either in $W$ or equal to $p(w)$ or $s(w)$ for some $w \in W$ is constant on exactly $k$ arithmetic progressions of step length $n$.
\end{definition}

\begin{lemma}\label{toe}
A sequence $x$ is regular Toeplitz if and only if for every $\epsilon > 0$, there exist $n, N, k$ for which the set $L_N(x)$ of $N$-letter subwords of $x$ is $(\epsilon, n, N, k)$-regular Toeplitz compatible.
\end{lemma}

\begin{proof}
For the forward direction, assume that $x$ is a regular Toeplitz sequence, and fix any $\epsilon > 0$. By definition, there exists $n$ and a set $S \subset [0,n)$ for which $|S|/n > 1 - \epsilon$ and where $x(s + in)$ is constant if and only if $s \in S$. (The reverse direction wasn't part of the definition, but clearly can be assumed by simply taking $S$ to be maximal.) Define $k = |S|$ and note that $k/n > 1 - \epsilon$. There then must exist $M$ so that $w = x([-M, M])$ is constant on exactly $k$ of the arithmetic progressions of step length $n$ (and contains unequal letters in each of the other $n-k$ arithmetic progressions). By minimality of $X$, there exists $N$ so that every $(N-1)$-letter subword of $x$ contains $w$, and so fails to be constant on at least $n-k$ arithmetic progressions of step length $n$. However, as subwords of $x$, all such words are constant on at least $k$ arithmetic progressions of step length $n$, and so are constant on exactly $k$ such arithmetic progressions. Exactly the same argument works for $N$-letter subwords of $x$, and so by definition, $L_N(x)$ is $(\epsilon, n, N, k)$-regular Toeplitz compatible.

For the reverse direction, assume that for all $\epsilon$, there exist $n, N, k$ so that the set $L_N(x)$ of $N$-letter subwords of $x$ is $(\epsilon, n, N, k)$-regular Toeplitz compatible. We claim that for all $M \geq N$, the set $L_M(x)$ is $(\epsilon, n, M, k)$-regular Toeplitz compatible, which implies upon taking limits that $x$ is constant on $k$ arithmetic progressions of step length $n$. This will complete the proof, since having this property for all $\epsilon$ implies by definition that $x$ is regular Toeplitz.

We verify the claim by induction on $M$. The base case $M = N$ holds by assumption. Assume that $L_M(x)$ is $(\epsilon, n, M, k)$-regular Toeplitz compatible, and consider any $w \in L_{M+1}(x)$. We can write $w = avb$, and by assumption, each of $v$, $av$, and $vb$ (as either words in $L_M(X)$ or such words with first or last letter removed) is constant on exactly $k$ of its arithmetic progressions of step length $n$. Denote by $S$ the set of residue classes $\pmod n$ on which $av$ is constant, and $T$ the corresponding set for $vb$; then $|S| = |T| = k$. Clearly $v$ is constant on all residue classes in $S \cup T$, so $|S \cup T| \leq k$. Then it must be the case that $S = T = S \cup T$. It's easily checked that $w = avb$ is constant precisely on the residue classes in $S \cap T = S = T$, and so it is constant on exactly $k$ of these. Since $w \in L_{M+1}(x)$ was arbitrary and the same conclusion holds for $u \in L_M(x)$ by the assumption that $L_M(x)$ is $(\epsilon, n, M, k)$-regular Toeplitz compatible, we've shown that $L_{M+1}(x)$ is $(\epsilon, n, M+1, k)$-regular Toeplitz compatible, completing the proof of the claim by induction. As mentioned above, this completes the proof of the lemma as well.
\end{proof}

\begin{theorem}\label{toegen}
The set of regular Toeplitz subshifts is residual in $\Tcalprimebar$.
\end{theorem}

\begin{proof}
We first claim that the property of being regular Toeplitz is preserved under any letter-to-word substitution $\tau$ as in Theorem~\ref{letword}. Choose any such $\tau$ (with length $\ell$), regular Toeplitz $X = \overline{\mathcal{O}(x)}$ with alphabet $\{0,1\}$, and any $\epsilon > 0$. By definition, there exists $n$ and a set $S \subset [0,n)$ with $|S| > (1-\epsilon)n$ so that $x$ is constant on all arithmetic progressions $\{s + in\}$ with $s \in S$. But then clearly we can define $S' \subset [0,n\ell)$ by $S' = \bigcup_{s \in S} [s\ell, (s+1)\ell)$, and then $\tau(x)$ is constant on all arithmetic progressions $\{s' + in\ell\}$ with $s' \in S'$. Since $|S'| = \ell |S| > (1-\epsilon)n\ell$, we've shown that $\tau(x)$ is regular Toeplitz, so $\tau^{*}(X) = \overline{\mathcal{O}(\tau(x))}$ is a regular Toeplitz subshift.

So using Theorem~\ref{letword}, it again suffices to prove that the set of regular Toeplitz subshifts is a $G_{\delta}$.

For any finite $\A \subset \mathbb{Z}$, denote by $C(\epsilon, n, N, k, \A)$ the set of $(\epsilon, n, N, k)$-regular Toeplitz compatible subsets of $\A^N$. Then by Lemma~\ref{toe}, the set of regular Toeplitz subshifts can be written as
\begin{gather}
\bigcap_{t \in \mathbb{N}} \bigcup_{\substack{m,n,N,k \in \mathbb{N}\\ \A \subset \mathbb{Z}, |\A| < \infty}} \{X \in \Scal \mid L_N(X) \in C(t^{-1}, m, n, N, k)\} \\
= \bigcap_{t \in \mathbb{N}} \bigcup_{\substack{m,n,N,k \in \mathbb{N}\\ \A \subset \mathbb{Z}, |\A| < \infty}} \bigcup_{\substack{X \in \Scal \\L_N(X) \in C(t^{-1}, m, n, N, k)}} [X,N],
\end{gather}
which is clearly a $G_{\delta}$.
\end{proof}

The following is immediate, since regular Toeplitz subshifts are known to be zero entropy, minimal, and uniquely ergodic~\cite{DownarowiczSurvey}.

\begin{corollary}\label{cor:zemuegeneric}
The set of zero entropy, minimal, uniquely ergodic subshifts is residual in $\Tcalprimebar$.
\end{corollary}

\begin{remark}
In~\cite{DownarowiczKasjan2015} it was shown that Sarnak's conjecture holds for all regular Toeplitz subshifts. It follows from Theorem~\ref{toegen} then that Sarnak's conjecture holds for a generic system in $\Tcalprimebar$.
\end{remark}

In fact we can give slightly more information about a generic regular Toeplitz subshift in $\Tcalprimebar$: it must factor onto the universal odometer. For this, we will use the following simple lemma.

\begin{lemma}\label{eigen}
For any subshift $X$ and $n \in \mathbb{N}$, if there exists a clopen set $E$ so that $X = \bigsqcup_{i=0}^{n-1} \sigma^i E$, then
all $n$th roots of unity are topological eigenvalues of $(X, \sigma)$.
\end{lemma}

\begin{proof}
Letting $\zeta_n$ be a primitive $n$th root of unity, simply define the eigenfunction $f = \sum_{i = 0}^{n-1} \zeta_n^i \chi_{\sigma^i E}$.
\end{proof}

\begin{theorem}\label{univ}
The set of regular Toeplitz subshifts factoring onto the universal odometer is residual in $\Tcalprimebar$.
\end{theorem}

\begin{proof}

It suffices to show that the set of Toeplitz subshifts which contain $\mathbb{Q}$ in their set of topological eigenvalues is generic in $\Tcalprimebar$, and we will do this using Lemma~\ref{eigen}. Given any $n$, we will show that the set $P_n$ of transitive subshifts $X$ with a clopen partition $\bigsqcup_{i=0}^{n-1} \sigma^i E$ is generic in $\Tcalprimebar$.

Fix $n \in \mathbb{N}$ and let $C$ be a nonempty cylinder in $\Tcalprimebar$. Let $X$ be the subshift on the alphabet $\{0,1\}$ defined as the closure of the unions of orbits of all sequences of the form
\[
\ldots (0^{\epsilon_0} 1)(0^{\epsilon_1} 1) \ldots
\]
for all sequences $(\epsilon_k) \in \{n-1,2n-1\}^\mathbb{Z}$. Then $X$ is infinite and transitive so $X \in \Tcalprime$, and using the clopen subset $E = [1] \cup \sigma^n [1]$ it is straightforward to check that $X \in P_n$ as well. Then by Lemma~\ref{letword} we may find $\ell$ and $\tau$ so that $\tau^{*}(X)$ is contained in $C$ and has unique decipherability as in Lemma~\ref{letword}. We claim that $\tau^{*}(X) \in P_{n\ell}$, which implies that $\tau^{*}(X)$ is also in $P_n$. (If $X = \bigsqcup_{i=0}^{n\ell-1} \sigma^i E$, then $X = \bigsqcup_{i=0}^{n-1} \sigma^i F$ for $F = \bigsqcup_{j = 0}^{\ell - 1} \sigma^{jn} E$.)

To see this, define the clopen set $D = \tau(E)$, i.e. biinfinite concatenations coming from sequences in $E$ with `dividing line' at the origin. This set is well-defined only because of the unique decipherability of $\tau$. Choose any $0 \leq i < n\ell$, and consider the sets $D$ and $\sigma^i D$. If $i$ is not a multiple of $\ell$, then $\sigma^i D$ and $D$ are disjoint due to unique decipherability of $\tau$. If $i = j\ell$ for some $j < n$, then $\sigma^i D = \tau(\sigma^j E)$, which is disjoint from $D$ since $\tau$ is uniquely decipherable and $E$ and $\sigma^j E$ are disjoint. Finally, by definition of $\tau^{*}(X)$, every $y \in \tau^{*}(X)$ can be written as $\sigma^k \tau(x)$ for some $0 \leq k < \ell$ and $x \in X$, which can in turn be written as $\sigma^k \tau(\sigma^m e) = \sigma^{k + m\ell} \tau(e)$ for some $0 \leq m < n$ and $e \in E$. Since $0 \leq k + m\ell < n\ell$, we've shown that $\tau^{*}(X) = \bigsqcup_{r=0}^{n\ell-1} \sigma^r D$, so $X \in P_{n\ell}$. Since $C$ was arbitrary and $\tau^{*}(X) \in C \cap P_{n\ell} \subset C \cap P_n$, we've shown that $P_n$ is dense.

It remains only to show that each $P_n$ is a $G_{\delta}$. We note that by compactness, for any finite $\A \subset \mathbb{Z}$,
subshift $X \in \Scal[\A]$, $m \in \mathbb{N}$, $S \subset \A^m$, and associated clopen set $E = \bigcup_{w \in S} [w]$,
$X = \bigsqcup_{i = 0}^{n-1} \sigma^i E$
if and only if there exists $N>m$ so that for every $w \in L_N(X)$, there exists $i$ so that every $m$-letter subword
$w(j) \ldots w(j+m-1)$ has cylinder set contained in $E$ if and only if $j \equiv i \pmod n$.
Denote by $D(m,S,N,\A)$ the set of all $S \subset \A^N$ with this property. Then, we can write $P_n$ as 
\[
\bigcup_{\substack{m,N \in \mathbb{N}, \A \subset \mathbb{Z},\\ |\A| < \infty, S \subset \A^m}} \{X \in \Tcalprimebar \mid L_N(X) \in D(m,S,N,\A)\} \]
\[ =
\bigcup_{\substack{m,N \in \mathbb{N}, \A \subset \mathbb{Z},\\ |\A| < \infty, S \subset \A^m}} \bigcup_{\substack{X \in \Tcalprimebar \\L_N(X) \in D(m,S,N,\A)}} [X,N],
\]
which is clearly open. Each $P_n$ is then open and dense, so the intersection $\bigcap_n P_n$ is a dense $G_{\delta}$, and
the set of regular Toeplitz subshifts in $\bigcap_n P_n$ is a dense $G_{\delta}$. By Lemma~\ref{eigen}, each such Toeplitz factors onto the universal odometer, completing the proof.

\end{proof}








\begin{remark}
In~\cite{HochmanGeneric}, Hochman showed that in fact a single conjugacy class (that of the universal odometer) was generic in the space of transitive homeomorphisms of the Cantor set. We remind the reader that such a result cannot hold here: by the Curtis-Hedlund-Lyndon theorem, the conjugacy class within $\Scal$ of any subshift is countable, and a countable set cannot contain a dense $G_{\delta}$ in $\Tcalprimebar$ since $\Tcalprimebar$ has no isolated points.
\end{remark}

\subsection{Complexity for a generic subshift in $\Tcalprimebar$}\label{cplxt}

For increasing $f,g: \mathbb{N} \rightarrow \mathbb{R}^+$, the set $\mathbf{S}_{f,g}$ of subshifts in $\Tcalprimebar$ whose complexity is between $f,g$ infinitely often is a $G_{\delta}$ set by Theorem~\ref{cplxthm}. Our main result in this section is that this set is also dense (and therefore generic) in $\Tcalprimebar$ if $f$ and $g$ are `far enough apart' in the set of a certain order on increasing functions.

\begin{definition}
For two increasing functions $f,g: \mathbb{N} \rightarrow \mathbb{R}^+$, we say that $f \prec g$ if for every $s,t \in \mathbb{N}$, there exists a constant $N$ so that $tf(n + s) < g(tn)$ and $f(tn) < tg(n - s)$ for all $n > N$. 
\end{definition}



\begin{lemma}\label{cplxbd}
If $X$ is a subshift with alphabet $\A_X$, $\ell \in \mathbb{N}$, and $\tau: \A_X \rightarrow \A_Y^\ell$ is an injective map with the unique decipherability property from Theorem~\ref{letword}, then
for all $n \geq 3\ell$,
\[
\ell c_{X}(\lfloor n/\ell \rfloor - 2) \leq c_{\tau^*(X)}(n) \leq \ell c_{X}(\lceil n/\ell \rceil + 2).
\]
\end{lemma}

\begin{proof}
We note that any word in $L_n(\tau^*(X))$ is a subword of a concatenation of the words $\{\tau(a)\} \subset \A^\ell$, and so is determined completely by a residue class modulo $\ell$ and a concatenation of less than or equal to $1 + \frac{n+1}{\ell} \leq 2 + \lceil n/\ell \rceil$ words of the form $\tau(a)$. In turn, by definition of $\tau(x)$, these are determined by some word in $L(X)$ with length $m_i \leq 2 + \lceil n/\ell \rceil$ depending only on the residue class $i$. Therefore,
\[
c_{\tau^{*}(X)}(n) \leq \sum_{i = 0}^{\ell-1} c_{X}(m_i) \leq \ell c_{X}(\lceil n/\ell \rceil + 2).
\]
For the other inequality, we note that since $n \geq 3\ell$, for every word $u \in L_n(Y)$, there is a unique decomposition of the form $p \tau(v) s$ for $v \in L(X)$ whose length $m_{|p|}$ is at least $n/\ell - 2$ and depends only on $|p|$. Clearly for fixed $|p|$, different $v \in L(X)$ yield different $u \in L(\tau^*(X))$, so
\[
c_{\tau^{*}(X)}(n) \geq \sum_{i = 0}^{\ell-1} c_{X}(m_i) \geq \ell c_{X}(\lfloor n/\ell \rfloor + 2),
\]
completing the proof.

\end{proof}



\begin{theorem}\label{mainfg}
If $f,g: \mathbb{N} \rightarrow \mathbb{R}^+$ are increasing, $f \prec g$, and there exists any infinite transitive subshift $X$ and sequence $(k_n)$ where $f(k_n) \prec c_{X}(k_n) \prec g(k_n)$, then the set $\mathbf{S}_{f,g}$ is residual in $\Tcalprimebar$.
\end{theorem}

\begin{proof}
That $\mathbf{S}_{f,g}$ is a $G_{\delta}$ is implied by Theorem~\ref{cplxthm}, so we need only show that it is dense.

We claim that this follows from Theorem~\ref{letword} and Lemma~\ref{cplxbd}. Indeed, choose such a subshift $X$, (we assume without loss of generality that $\A_X = \{1, \ldots, |\A_X|\}$),
sequence $(k_n)$ where $f(k_n) \prec c_{X}(k_n) \prec g(k_n)$, and any cylinder $C$ in $\Tcalprimebar$ (with associated alphabet $\A_C$). First, we need to find $M$ so that there is an injective map $\rho: \A_X \rightarrow \{0,1\}^M$ with the unique decipherability property from Theorem~\ref{letword}. For this, simply choose $K$ so that $2^K > |\A_X|$, define $M = 2K+2$, and define $\rho(i)$ to be $K$ $0$s, followed by a $1$, then the $K$-bit binary expansion of $i$, then another $1$. It's easy to see that $\rho$ has the desired decipherability, since every $\rho(i)$ begins with $0^K$, and $0^K$ appears nowhere else in concatenations of the words $\rho(a)$ for $a \in \A_X$. Then $\rho^{*}(X)$ is a subshift on $\{0,1\}$.

By Theorem~\ref{letword}, there exists $\ell$ and $\tau: \{0,1\} \rightarrow \A_C^\ell$ so that $\tau^{*}(\rho^*(X)) \in C$.
(Note that $\tau^*(\rho^*(X)) \in \Tcalprime$ since $X$ is.) We can then define an injective map $\tau \circ \rho: \A_X \rightarrow \A_C^{M\ell}$, which has the unique decipherability of Theorem~\ref{letword} since $\tau$ and $\rho$ do; it is clear that $(\tau \circ \rho)^*(X) = \tau^{*}(\rho^*(X))$. Then, Lemma~\ref{cplxbd} implies that for every
$n \geq 3M\ell$,
\begin{equation}\label{fgcplx0}
M\ell c_{X}(k_n - 2) \leq c_{(\tau \circ \rho)^*(X)}(M\ell k_n) \leq M\ell c_{X}(k_n + 2).
\end{equation}
Since $f(k_n) \prec c_{X}(kn) \prec g(k_n)$, by definition of $\prec$ (for $s = 2$ and $t = M\ell$), for large enough $n$ 
\begin{equation}\label{fgcplx1}
M\ell c_{X}(k_n + 2) < g(M\ell k_n) \textrm{ and } f(M\ell k_n) < M\ell c_{X}(k_n - 2).
\end{equation}
Combining (\ref{fgcplx0}) and (\ref{fgcplx1}) yields
\[
f(M\ell k_n) < c_{(\tau \circ \rho)^*(X)}(M \ell k_n) < g(M\ell k_n)
\]
for sufficiently large $n$, implying that $(\tau \circ \rho)^*(X) \in \mathbf{S}_{f,g} \cap C$ and completing the proof.
\end{proof}

In particular, this already implies that for any countable collection of pairs $(f_k, g_k)$ satisfying Theorem~\ref{mainfg}, a generic subshift has complexity function entering each one of these ranges infinitely often! In order to give some more explicit statements, we need to collect some examples which are (nearly) in the literature. 

\begin{proposition}\label{linearex}
For every $1 < \alpha < \beta$, there exists a transitive subshift $X$ with alphabet $\{0,1\}$ and sequence $(k_n)$ where
$\alpha k_n < c_{X}(k_n) < \beta k_n$ for all $n$. 
\end{proposition}

\begin{proof}
We will make use of a class of examples from \cite{OP}, each defined via an increasing sequence $(n_k)$ of integers. For any such $(n_k)$, define a sequence
\[
x = 0^{\infty}.1 0^{n_0} 1 0^{n_1} 1 0^{n_0} 1 0^{n_2} 1 0^{n_0} 1 0^{n_1} 1 0^{n_0} 1 0^{n_3} \ldots,
\]
where the subscripts in the exponents follow the sequence $01020103\ldots$ whose $i$th letter is the largest $j$ s.t. $2^j$ divides
$i$. Then, define $X$ to be the closure of the orbit of $x$. For each such subshift, \cite{OP} gives a formula for the word complexity in terms an auxiliary sequence $m_k$ (there written $|w(k)|$) defined by the recursion
\[
m_k = 2^k + 2n_k + \sum_{j = 0}^{k-1} 2^{k-j-1} n_j.
\]
Whenever $n_k \geq m_{k-1}$ for all $k$ (which will be true of all examples we consider here), we can define the disjoint union of intervals
$R = R(n_k) = \bigcup (n_k, m_k]$. Under that assumption, the complexity of $X$ is proved in \cite{OP} to be
\begin{equation}\label{ormesbd}
c_{X}(n) = n + 1 + |R \cap [1, n)|.
\end{equation}
We claim that $(n_k)$ can be chosen so that $\liminf_j \frac{c_{X}(j)}{j} = 1$ and $\limsup_j \frac{c_{X}(j)}{j} = 2$. Specifically, we claim that $(n_k)$ can be chosen to satisfy $\frac{c_{X}(n_{k^2+1})}{n_{k^2+1}} < 1 + k^{-1}$ and $\frac{c_{X}(n_{(k+1)^2})}{n_{(k+1)^2}} > 2 - k^{-1}$ for all $k$. Suppose that $n_1, \ldots, n_{k^2}$ have been chosen satisfying the hypothesis. We define any $n_{k^2+1} > k(1+m_{k^2})$, and note that by (\ref{ormesbd}),
\[
c_{X}(n_{k^2+1}) \leq n_{k^2+1} + 1 + m_{k^2} < n_{k^2+1}(1 + k^{-1}),
\]
verifying $\frac{c_{X}(n_{k^2+1})}{n_{k^2+1}} < 1 + k^{-1}$. Then, define $n_{k^2 + j} = m_{k^2 + j - 1}$ for $1 < j \leq 2k + 1$.
This means that $R \supseteq (n_{k^2+1}, m_{k^2 + 2k}]$. By definition, $m_i > 2n_{i-1}$ for all $i$, and since
$n_{k^2 + j} = m_{k^2 + j - 1}$ for $1 < j \leq 2k + 1$, $n_{(k+1)^2} = m_{k^2 + 2k} > 2^{2k-1} m_{k^2 + 1}$. By (\ref{ormesbd}),
\[
c_{X}(n_{(k+1)^2}) = n_{(k+1)^2} + 1 + m_{k^2 + 2k} - n_{k^2+1} > 2n_{(k+1)^2} - m_{k^2+1} > n_{(k+1)^2}(2 - k^{-1}),
\]
verifying $\frac{c_{X}(n_{(k+1)^2})}{n_{(k+1)^2}} < 2 - k^{-1}$ and completing the induction.

Denote by $X_0$ the subshift just defined satisfying $\liminf_j \frac{c_{X_{0}}(j)}{j} = 1$ and $\limsup_j \frac{c_{X_{0}}(j)}{j} = 2$. We now prove Theorem~\ref{linearex} for $1 < \alpha < \beta < 2$. By a result of Cassaigne (\cite{cass}), $c_{X_{0}}(n+1) - c_{X_{0}}(n)$ is bounded, i.e. there exists $C$ so that $c_{X_{0}}(n+1) - c_{X_{0}}(n) < C$ for all $n$. Choose any $N > C(\beta - \alpha)^{-1}$. By definition of $\liminf$ and $\limsup$, there exist $K, L > N$ so that $\frac{c_{X_{0}}(K)}{K} < \alpha$ and $\frac{c_{X_{0}}(L)}{L} > \beta$. 
Define $J \in [K, L]$ maximal so that $\frac{c_{X_{0}}(J)}{J} \leq \alpha$ (meaning that $\frac{c_{X_{0}}(J+1)}{J+1} > \alpha$). Then
\[
c_{X_{0}}(J+1) < c_{X_{0}}(J) + C \leq J\alpha + C < (J+1)\alpha + C < (J+1)\alpha + N(\beta - \alpha) < (J+1)\beta.
\]
Therefore, $\frac{c_{X_{0}}(J+1)}{J+1} < \beta$, and we already knew that $\frac{c_{X_{0}}(J+1)}{J+1} > \alpha$. Since $N$ was arbitrary and $J > N$, this completes the proof when $1 < \alpha < \beta < 2$.

For the remaining cases, we wish to take the product of $X_0$ with a periodic orbit with period of the form $2^j$ to achieve larger linear complexities. However, we need to ensure that such products retain the transitive property. Recall that $X_0$ is the closure of the orbit of the sequence
\[
x = 0^{\infty}.1 0^{n_0} 1 0^{n_1} 1 0^{n_0} 1 0^{n_2} 1 0^{n_0} 1 0^{n_1} 1 0^{n_0} 1 0^{n_3} \ldots
\]
Define a sequence of words $v(k)$ as follows: $v(0) = 1$, and for all $k$, $v(k+1) = v(k) 0^{n_k} v(k)$. By definition of $X_0$, every word in $L(X_0)$ is a subword of some $v(k)$. In addition, for every $m > k$, the word $v(k) 0^{n_m} v(k)$ is a subword of $x$, and so $|v(k)| + n_m$ is a return time of $[v(k)]$ to itself, i.e. $[v(k)] \cap \sigma^{|v(k)| + n_m} [v(k)] \neq \varnothing$.

We now wish to show that in the construction of $X_0$, one can control the sequence $(n_m)$ so that for all $j$, every residue class modulo $2^j$ is achieved by infinitely many $n_m$. This is fairly clear, since there are infinitely many $n_m$ which are defined only in terms of a lower bound (for $m = k^2 + 1$ as written above). We from now on assume without loss of generality that $X_0$ has this property. Then, fix any $j$ and any $u,v \in L(X_0)$. There exists $k$ so that $u,v$ are subwords of $v(k)$; say that $u,v$ appear at locations $i,j$ respectively (i.e. $[v(k)] \subset \sigma^i[u] \cap \sigma^j[v])$. By the above, every $|v(k)| + n_m$ for $m > k$ is a return time of $[v(k)]$ to itself, and therefore for every $m$, $[u] \cap \sigma^{|v(k)| + j-i + n_m}[v] \neq \varnothing$. Since $n_m$ with $m > k$ achieve every residue class modulo $2^j$, the same is true for the positive integers $|v(k)| + j-i + n_m$. Since $u,v$ were arbitrary, this implies that $(X_0, \sigma^{2^j})$ is transitive, and therefore that $X_j := X_0 \times P_j$ is transitive (with respect to $\sigma$), where $P_j$ is the orbit of the periodic sequence $(10^{2^j-1})^{\infty}$. It is clear that $c_{P_{j}}(n) = 2^j c_{X_{0}}(N)$ for $n \geq 2^j$, and so $\liminf \frac{c_{X_{j}}(n)}{n} = 2^j$ and $\limsup \frac{c_{X_{j}}(n)}{n} = 2^{j+1}$.

Now, the proof of Theorem~\ref{linearex} proceeds exactly as above for the case $2^j < \alpha < \beta < 2^{j+1}$ for any $j \geq 0$. The only remaining case is where $\alpha \leq 2^j \leq \beta$ for some $j$, and this case is obvious, since we can just define $\beta' = \frac{\alpha + 2^j}{2}$. Then the theorem holds for $\alpha, \beta'$, which implies that it trivially holds for $\alpha, \beta$ since $\beta' < \beta$.

\end{proof}

\begin{proposition}\label{polyex}
For every $1 < \alpha < \beta$, there exists a transitive subshift $X$ with alphabet $\{0,1\}$ and sequence $(k_n)$ where
$k_n^{\alpha} < c_{X}(k_n) < k_n^{\beta}$ for all $n$.
\end{proposition}

\begin{proof}
This is a consequence of Theorem A from \cite{koskas}, which states that for every rational $p/q > 1$, there exists
a Toeplitz subshift $X$ and constants $c_1, c_2 > 0$ so that
\begin{equation}\label{polyexeqn}
c_1 n^{p/q} < c_{X}(n) < c_2 n^{p/q}
\end{equation}
for all $n$. To derive Theorem~\ref{polyex}, just take, for any $1 < \alpha < \beta$, some $p/q \in (\alpha, \beta)$,
and then (\ref{polyexeqn}) implies that the associated subshift $X$ must satisfy $k_n^{\alpha} < c_{X}(k_n) < k_n^{\beta}$ for sufficiently large $n$.
\end{proof}

\begin{theorem}\label{expex}
For every $1 < \alpha < \beta$, there exists a transitive subshift $X$ with alphabet $\{0,1\}$ and sequence $(k_n)$ where
$e^{k_n^{\alpha}} < c_{X}(k_n) < e^{k_n^{\beta}}$ for all $n$.
\end{theorem}

\begin{proof}
This is a consequence of Theorem 3 from \cite{Cass03}, which proves that for any $\phi: \mathbb{R}^+ \rightarrow \mathbb{R}^+$ which is differentiable (except possibly at $0$), satisfies $\frac{\phi(t)}{\log t} \rightarrow \infty$, and has $\phi'(t)$ decreasing and bounded from above by $t^{-\beta}$ for some positive $\beta$, there exists a minimal subshift $X$ with $\log c_X(n)/\phi(t) \rightarrow 1$. For any $\alpha, \beta$ as in the theorem, clearly $\phi(t)= t^{(\alpha+\beta)/2}$ satisfies the above conditions, and
then $X$ guaranteed by \cite{Cass03} has the desired properties.

\end{proof}

Using these examples, we can give some surprising properties of generic subshifts in $\Tcalprimebar$.

\begin{proposition}\label{bigcplx}
For a generic subshift $X$ in $\Tcalprimebar$, each of the following holds:
\begin{enumerate}
\item For all $\gamma > 1$, there is a subsequence $k_n$ so that $\frac{c_{X}(k_n)}{k_n} \rightarrow \gamma$
\item For all $\gamma \geq 1$, there is a subsequence $k_n$ so that $\frac{\log c_{X}(k_n)}{\log k_n} \rightarrow \gamma$
\item For all $0 < \gamma < 1$, there is a subsequence $k_n$ so that $\frac{\log \log c_{X}(k_n)}{\log k_n} \rightarrow \gamma$.
\end{enumerate}
\end{proposition}

Informally, this theorem states that a generic subshift in $\Tcalprimebar$ has complexity function for which there are subsequences where it behaves like any possible linear function $\gamma n$, like any possible polynomial function $n^{\gamma}$, and like any possible `stretched exponential' function $e^{n^{\gamma}}$ (within the restriction of zero entropy, which we know to be generic by Theorem~\ref{zeroT}).

\begin{proof}
The proofs for all three items in the list are similar; we begin with the first. It's easily checked by definition that for any $1 < \gamma < \delta$, $\gamma n \prec \frac{2\gamma + \delta}{3} n \prec \frac{\gamma + 2\delta}{3} n \prec \delta n$. By Proposition~\ref{linearex} (with $\alpha = \frac{2\gamma + \delta}{3}$ and $\beta = \frac{\gamma + 2\delta}{3}$), there exists a subshift $X$ with alphabet $\{0,1\}$ and sequence $(k_n)$ so that $c_{X}(k_n) \in [\alpha k_n, \beta k_n]$ for all $n$, implying that
\[
\gamma k_n \prec c_{X}(k_n) \prec \delta k_n.
\]

Therefore, by Theorem~\ref{mainfg},
$\mathbf{S}_{\gamma n, \delta n}$ is generic in $\Tcalprimebar$. Now, fix any $\alpha > 0$ and any $k$; then
$\mathbf{S}_{\gamma n, (\gamma + k^{-1})n}$ is generic in $\Tcalprimebar$, meaning that $\bigcap_k \mathbf{S}_{\gamma n, (\gamma + k^{-1})n}$ is generic in $\Tcalprimebar$.
But by definition, for all $X \in \bigcap_k \mathbf{S}_{\gamma n, (\gamma + k^{-1})n}$ and for all $k$, there exists $n_k$ so that
$\gamma n_k < c_{X}(n_k) < (\gamma + k^{-1})n_k$, implying that $\frac{c_{X}(k_n)}{k_n} \rightarrow \gamma$.

So, for each $\gamma > 1$, the set of subshifts for which there exists $(k^{(\gamma)}_n)$ with $\frac{c_{X}(k^{(\gamma)}_n)}{k^{(\gamma)}_n} \rightarrow \gamma$ is residual in $\Tcalprimebar$. Since the intersection of countably many residual sets is residual, a generic subshift in $\Tcalprimebar$ in fact has such sequences for all rational $\gamma > 1$. But this implies the existence of such sequences for all $\gamma > 1$ by a diagonal argument.

The other two statements are proved similarly 
by using Propositions~\ref{polyex} and \ref{expex}.
\end{proof}

Near minimal complexity, even more can be said. 

\begin{lemma}\label{hlem}
For any unbounded increasing $h: \mathbb{N} \rightarrow \mathbb{R}^+$, $n \prec n + h(n)$. (Here, $n$ refers to the function $f(n) = n$). 
\end{lemma}

\begin{proof}
If we define $f(n) = n$ and $g(n) = n + h(n)$, for any $s,t$,
$tf(n + s) = tn + ts$ and $g(tn) = tn + h(tn)$, and clearly for large enough $n$, $h(tn) > ts$, implying
that $tf(n+s) < g(tn)$.
Similarly, $f(tn) = tn$ and $tg(n-s) = tn - ts + th(n-s)$, and clearly for large enough $n$, $h(n-s) > s$, implying that $f(tn) < tg(n-s)$.
The claim now follows from the definition of $\prec$.
\end{proof}

The following corollary is nearly immediate.

\begin{corollary}\label{sosmall}
For any unbounded increasing $h: \mathbb{N} \rightarrow \mathbb{R}^+$, $\mathbf{S}_{n, n+h(n)}$ is residual in $\Tcalprimebar$.
\end{corollary}

\begin{proof}
There is a subtlety here; the subshifts that we wish to use for the hypothesis of Theorem~\ref{mainfg} are Sturmians, where $c_{X}(n) = n+1$ for all $n$, and it is not the case that $n \prec n+1$. However, for any $h$ as in the statement, we just apply Theorem~\ref{mainfg} with $f = 0.5n$ and $g = n + h(n)$. Clearly $0.5n \prec n+1 \prec n + h(n)$ (using Lemma~\ref{hlem}), and so $\mathbf{S}_{0.5n, n+h(n)}$ is generic in $\Tcalprimebar$.

We now just note that no periodic (finite) subshift can be in $\mathbf{S}_{0.5n, n+h(n)}$ by the Morse-Hedlund theorem, and that all infinite subshifts have $c_{X}(n) > n$ for all $n$. Therefore, $\mathbf{S}_{0.5n, n+h(n)} = \mathbf{S}_{n, n+h(n)}$, and the proof is complete.
\end{proof}

We note that Corollary~\ref{sosmall} cannot be improved, i.e. $\mathbf{S}_{n, n+C}$ is not residual for any constant $C$. This is because for any $X \in \mathbf{S}_{n, n+C}$, there is a subsequence along which $n \leq c_{X}(n) \leq n + C$, which implies that in fact $c_{X}(n) = n + D$ for all large enough $n$ (note that $X$ is not periodic, so by the Morse-Hedlund theorem, $c_{X}(n+1) - c_{X}(n) \geq 1$ for all $n$). However, this implies that (for instance) $\mathbf{S}_{1.1n, 2n}$ and $\mathbf{S}_{n, n+C}$ are disjoint, and so since the former is residual in $\Tcalprimebar$ by Theorem~\ref{bigcplx}, the latter cannot be.

One more corollary here is worth noting.

\begin{corollary}\label{1rs}
The collection $\mathbf{U}_{1RS}$ of subshifts for which there are infinitely many $n$ with $c_{X}(n+1) = c_{X}(n) + 1$ is residual in $\Tcalprimebar$.
\end{corollary}

\begin{proof}
We claim that $\mathbf{S}_{1.1n, 1.2n} \subset \mathbf{U}_{1RS}$, and then this follows immediately from Theorem~\ref{bigcplx}. To see this, choose any
$X \in \mathbf{U}_{1RS}^c$. By definition, there exists $N$ so that $c_{X}(n+1) \geq c_{X}(n) + 2$ for $n > N$, which implies that
$c_{X}(n) \geq 2n - N$ for all $n$. This clearly implies $X \notin \mathbf{S}_{1.1n, 1.2n}$, completing the proof.
\end{proof}

In fact we can use Corollary~\ref{1rs} to derive a useful substitutional structure for generic subshifts in $\Tcalprimebar$, for which we need a definition.

\begin{definition}
A subshift $X$ has \textit{alphabet rank $k$} if there exist alphabets $\A_n$ (for $n \geq 0$) and substitutions
$\rho_n: \A_n \rightarrow \A_{n-1}^*$ (for $n \geq 1$) so that $\liminf |\A_n| = k$ and every word in $L(X)$ is a subword of $(\rho_1 \circ \cdots \circ \rho_n)(a)$ for some $n$ and $a \in \A_n$. The sequence $(\rho_n)$ is \textit{right proper} if all $\rho_n(a)$ ($a \in \A_n$) have the same terminal letter for all $n$.
\end{definition}

\begin{corollary}\label{ark2}
A generic subshift in $\Tcalprimebar$ has alphabet rank two for a right proper sequence $(\rho_n)$.
\end{corollary}

\begin{proof}
Due to Theorem~\ref{minimal} and Corollary~\ref{1rs}, it suffices to find a right proper sequence $(\rho_n)$ inducing $X$ for every infinite minimal subshift $X$ where $c_{X}(n+1) - c_{X}(n) = 1$. Suppose that $X$ is such a subshift, and define the infinite set
$S = \{n \mid c_{X}(n+1) - c_{X}(n) = 1\}$.

Choose any $n \in S$. Then by Corollary~\ref{RScor}, there is a unique right-special word $w_n \in L_{n}(X)$ and exactly two letters $a_n \neq b_n$ which can follow $w_n$. Since every word in $L_n(X)$ except $w_n$ forces the following letter, we can repeatedly extend $w_n a_n$ to the right as long as its terminal $n$ letters force the next letter, which will continue until the final $n$ letters are exactly $w_n$. Put another way, there exists $u_n$ beginning with $a_n$ where $w_n u_n$ ends with $w_n$, $w_n u_n$ contains no occurrences of $w_n$ except at the beginning and end, and any occurrence of $w_n a_n$ in any $x \in X$ is a prefix of $w_n u_n$. 
Define $w_n v_n$ similarly using $w_n b_n$. Assume without loss of generality that $|u_n| \leq |v_n|$. Then by the above, every $x \in X$ is some biinfinite concatenation of the words $u_n, v_n$ (simply mark locations of $w_n$ within $x$, and then every end of a $w_n$ is an end of an $u_n$ or $v_n$). In addition, every occurrence of $w_n$ in such an $x$ must share its last letter with the last letter of some concatenated $u_n$ or $v_n$, which means that the decomposition of any $x$ into a concatenation of $u_n$ and $v_n$ is unique.

We make the following claim: for every $k$, there exists $N$ so that for every $n \in S$ with $n > N$, the associated words
$u_n, v_n$ have lengths greater than $k$. Assume for a contradiction that this is not the case, i.e. that there exists $k$ so that for infinitely many $n \in S$, $|u_n| \leq k$. Recall that $w_n u_n$ ends with $w_n$ by definition, so in particular, $w_n(i) = w_n(i+|u_n|)$ for all $i$ with $1 \leq i, i+|u_n| \leq |w_n|$. But then passing to a subsequence with $|u_n|$ constant and taking a limit of the associated subsequence of $(w_n)$ yields a periodic point, a contradiction to $X$ being infinite and minimal. Therefore, the original claim holds.

Now, we inductively describe the sequence $(\rho_n)$. First, choose any $n_1 \in S$, and define
$\rho_1$ sending $0$ to $u_{n_1}$ and $1$ to $v_{n_1}$. Note that $u_{n_1}$ and $v_{n_1}$ both have final letter equal to the final letter of $w_{n_1}$, so $\rho_1(0)$ and $\rho_1(1)$ have the same final letter. Now, choose $n_2 \in S$ so that $n_2$ and $|u_{n_2}|$ are greater than $|v_{n_1}|$.
Since $w_{n_2} u_{n_2}$ and $w_{n_2} v_{n_2}$ have $w_{n_2}$ as a suffix, we know that $u_{n_2}$, $v_{n_2}$, and $w_2$ share a common suffix of length $\min(|u_{n_2}|, n_2) \geq |v_{n_1}|$. 
Finally, we note that $w_1$ is a suffix of $w_2$ (the suffix of a right-special word is right-special), so either $u_{n_1}$ or $v_{n_1}$ is a suffix of $w_2$.

Therefore, both $u_{n_2}$ and $v_{n_2}$ have a common suffix $c \in \{u_{n_1}, v_{n_1}\}$. Since all $x \in X$ are concatenations of
$u_{n_2}$ and $v_{n_2}$, we know that $cu_{n_2}$ and $cv_{n_2}$ are in $L(X)$. This means that $u_{n_2}$ and $v_{n_2}$ are concatenations of $a_{n_1}$ and $b_{n_1}$. We can therefore define $\rho_2: \{0,1\} \rightarrow \{0,1\}^*$ so that
$u_{n_2} = (\rho_1 \circ \rho_2)(0)$ and $v_{n_2} = (\rho_1 \circ \rho_2)(1)$. Since $u_{n_2}$ and $v_{n_2}$ have common suffix $c$,
$\rho_2(0)$ and $\rho_2(1)$ share the same final letter.

Continue in this way to define a sequence $(\rho_k)$ so that for all $k$,
$u_{n_k} = (\rho_1 \circ \rho_2 \circ \cdots \circ \rho_k)(0)$ and $\rho_k(0)$ and $\rho_k(1)$ share the same final letter. Since $X$ is minimal and $u_{n_k} = (\rho_1 \circ \rho_2 \circ \cdots \circ \rho_k)(0)$ are words in $L(X)$ of increasing length, every word in $L(X)$ is a subword of some $(\rho_1 \circ \rho_2 \circ \cdots \circ \rho_k)(0)$. Finally, since $\A_n = \{0,1\}$ for all $n > 0$,
$\liminf |\A_n| = 2$.

\end{proof}

An immediate corollary is that these subshifts have the minimal possible topological rank among nontrivial subshifts.

\begin{corollary}\label{rk2}
A generic subshift in $\Tcalprimebar$ has topological rank two.
\end{corollary}

\begin{proof}
By Corollary 2.5 of \cite{durandleroy}, any minimal subshift $X$ with alphabet rank two for a right proper sequence $(\rho_n)$ has topological rank two. (Corollary 2.5 requires the extra assumption that each $\rho_n$ acts injectively, but this is implied by
Theorem 3.1 of \cite{recog} since each $\rho_n$ acts on an alphabet of size $2$ and $X$ contains no periodic points by minimality.)
\end{proof}

Corollary~\ref{rk2} immediately implies that the automorphism group of a generic subshift in $\Tcalprimebar$ is generated by the shift.

\begin{corollary}\label{autgrp}
A generic subshift in $\Tcalprimebar$ has automorphism group consisting of only powers of the shift.
\end{corollary}

\begin{proof}

\cite[Thm 3.1]{DDMP1} and \cite[Sec. 7]{DDMP2021} imply that topological rank two and minimality imply that the automorphism group is generated by the shift, so this is an immediate consequence of Theorem~\ref{minimal} and Corollary~\ref{rk2}.

\end{proof}

\subsection{Orbit equivalence, dimension groups, and mapping class groups for generic subshifts in $\Tcalprimebar$}
We now move on to characterizing dimension groups for generic subshifts in $\Tcalprimebar$, which will in turn yield results about orbit equivalence and mapping class groups.


\subsubsection{Orbit equivalence in $\Tcalprimebar$}

The goal of this subsection is to prove the following theorem.

\begin{theorem}\label{thm:dimgrpQ}
\sloppy The set of minimal subshifts $(X,\sigma)$ whose dimension group $(\coinv{\sigma},\coinv{\sigma}^{+},[1])$ is isomorphic (as ordered unital groups) to $(\mathbb{Q},\mathbb{Q}_{+},1)$ is residual in $\Tcalprimebar$.
\end{theorem}

Recall systems $(X,T)$ and $(Y,S)$ are orbit equivalent if there is a homeomorphism $\phi \colon X \to Y$ such that $\phi$ takes $T$-orbits onto $S$-orbits; in other words, for all $x \in X$, $\{\phi(T^{n}(x))\}_{n \in \mathbb{Z}} = \{S^{n}(\phi(x))\}_{n \in \mathbb{Z}}$. If two infinite minimal systems $(X,T)$ and $(Y,S)$ are orbit equivalent then there are well-defined maps $m \colon X \to \mathbb{Z}, n \colon Y \to \mathbb{Z}$ such that $\phi T(x) = S^{n(x)}\phi(x)$ and $\phi T^{m(x)}(x) = S \phi(x)$ for all $x \in X$, and if both $m,n$ have at most one point of discontinuity then we say $(X,T)$ and $(Y,S)$ are strong orbit equivalent. For more background on these notions we refer the reader to~\cite{GPS1995}.

In~\cite{GPS1995} it is proved that two Cantor minimal systems are strong orbit equivalent if and only if their associated ordered unital dimension groups are isomorphic (as unital ordered groups). Since the dimension group of the universal odometer is
$(\mathbb{Q},\mathbb{Q}_{+},1)$, Theorem~\ref{thm:dimgrpQ} then implies the following.

\begin{corollary}\label{cor:onesoeclass}
The strong orbit equivalence class of the universal odometer is residual in $\Tcalprimebar$.
\end{corollary}
Note that the universal odometer is not expansive and so does not actually belong to $\Scal$.

Before beginning, we briefly outline the proof of Theorem~\ref{thm:dimgrpQ}. The main step is to show that the set of uniquely ergodic subshifts in $\Tcalprimebar$ whose invariant measure takes only rational values on clopen subsets is generic. This implies that for a generic subshift, the state on the dimension group induced by the unique invariant measure has its image contained in the rationals. For a uniquely ergodic system, the kernel of the unique state is precisely the subgroup of infinitesimals. Thus to finish the proof, we show that generically, the image of the state map is all of $\mathbb{Q}$, and that the subgroup of infinitesimals is trivial; Theorem~\ref{thm:dimgrpQ} then follows.

Throughout this section, we will use the following convention: if $X$ is a uniquely ergodic subshift, we will denote by $\mu_{X}$ its unique invariant probability measure.
We begin by proving a somewhat technical condition for a generic class of subshifts in $\Tcalprimebar$.

\begin{definition}
For any uniquely ergodic subshift $X$ and any $v,w \in L(X)$, the
\textit{discrepancy of $w$ in $v$} is $D(w,v) := |v|_w - |v| \mu_X([w])$, where $|v|_w$ denotes the number of occurrences of $w$ in $v$.
\end{definition}


\begin{definition}\label{baldef}
For any uniquely ergodic subshift $X$ and $w \in L(X)$, we say \textit{$X$ is balanced for $w$} if there exists a constant $C_w$ so that for all $v \in L(X)$, $|D(w,v)| < C_w$. We say that \textit{$X$ is balanced for factors} (see \cite{berthebalanced}) if for all $w \in L(X)$, $X$ is balanced for $w$.
\end{definition}

We recall that $\mathbf{UE}$ denotes the set of uniquely ergodic subshifts. Define $\mathbf{UERB}$ to be the set of subshifts $X \in \mathbf{UE}$ such that $\mu_X([w])$ is rational for all words $w \in L(X)$ and $X$ is balanced for factors. 

\begin{theorem}\label{thm:UERBgeneric}
The set $\mathbf{UERB}$ is residual in $\Tcalprimebar$.
\end{theorem}

\begin{proof}

For all $m$, we define the set $\mathbf{UERB}_{m}$ of subshifts $X$ such that $X$ is uniquely ergodic, $\mu_{X}([w])$ is rational for all $m$-letter words $w$, and $X$ is balanced for all $m$-letter words. We will prove that for every $m$, $\mathbf{UERB}_m$ is the intersection of an open dense set with $\mathbf{UE}$; since $\mathbf{UERB} = \bigcap \mathbf{UERB}_m$ and since $\mathbf{UE}$ is generic in $\Tcalprimebar$ by Corollary~\ref{cor:zemuegeneric}, this completes the proof.

Choose any cylinder $C$ in $\Tcalprimebar$ and any $m \in \mathbb{N}$. Clearly there exists $m' \geq m$ so that $C$ has a subcylinder
$C'$ (possibly equal to $C$) of the form $[X,m']$ for some $X$. Define $\tau$ as in Theorem~\ref{letword}. By Theorem~\ref{letwordcyl}, there is a subcylinder $D$ so that
$D \subset \tau^{*}(\{0,1\}^{\mathbb{Z}}) \subset C' \subset C$. We claim that $D \cap \mathbf{UE} \subset \mathbf{UERB}_{m'} \subset \mathbf{UERB}_m$.

To see this, we recall that in the definition of $\tau$, $\tau(0), \tau(1) \in \A^\ell$ are defined via the words corresponding to paths $KKK'K'K'$ and $KK'KK'K'$ for cycles $K$, $K'$ in $G_{X,m'+1}$ which begin and end at the same vertex $w$. Define $u,u'$ to be the
$(\ell+m')$-letter words corresponding to these paths; then $\tau(0), \tau(1)$ are obtained by removing the $m'$-letter suffix $w$ from $u,u'$ respectively.

We now choose any $X \in UE \cap D$ and any $v \in L_{m'}(X)$ with $v \neq w$. Clearly 
$KKK'K'K'$ and $KK'KK'K'$ visit $v$ the same number of times, so $|u|_v = |u'|_v$; denote their common value by $N_v$. Since $u, u'$ begin and end with $w$, for any $t \in \{0,1\}^k$ we have $|\tau(t)w|_v = |t| N_v$.

By definition, since $X \in D$ we have $X = \tau^{*}(Y)$ for some subshift $Y$ on $\{0,1\}$. So, for any $x \in \tau(Y) \subset X$ and any $n \in\mathbb{N}$, $x([-n\ell,n\ell+m'))$ is of the form $\tau(t_n)w$ for some word $t_n \in \{0,1\}^{2n}$, and hence
\[
\frac{|x([-n\ell, n\ell+m'))|_v}{2n\ell+m'} = \frac{2n N_v}{2n \ell + m'} \stackrel{n \to \infty}\longrightarrow \frac{N_v}{\ell}.
\]
By the pointwise ergodic theorem, $\mu_X([v]) = N_v / \ell \in \mathbb{Q}$, and $\mu_X([w]) = 1 - \sum_{v \ne w} \mu_X([v]) \in
\mathbb{Q}$, so indeed the measures of all $X$-cylinder sets of $m^{\prime}$-letter words are rational.

We note that any word $q \in L(\tau^{*}(X))$ can be written as $q = p \tau(t) s$ for some $u \in L(X)$, $p$ a proper suffix of some $\tau(a)$, and $s$ a proper prefix of some $\tau(b)$ (for $a, b \in \{0,1\}$). In particular, $|p|, |s| < \ell$, so $|q| \in [|t|\ell, (|t|+2)\ell)$. Again, choose any $v \in L_{m'}(X)$ with $v \neq w$. The number of occurrences of $v$ in $\tau(t) w$ is exactly $|t| N_v$, so
\[
|t| N_v - m' \leq |q|_v \leq |t| N_v + 2\ell.
\]
Since $\mu_X([v]) = N_v/\ell$ and $|q| \in [|t|\ell, (|t|+2)\ell)$, $D(v,q) = |q|_v - \frac{|q|N_v}{\ell}$ satisfies
\[
-2N_v - m' < D(v,q) \leq 2\ell.
\]
Since this interval is independent of $q$, $X$ is balanced for $v$. Finally, we note that $|q|_w = |q| - m' + 1 - \sum_v |q|_v$ and
$\mu_X([w]) = 1 - \sum_v \mu_X([v])$. Therefore,
\[
D(w, q) = |q|_w - |q|\mu_X([w]) = |q| - m' + 1 - \sum_v |q|_v - |q| + \sum_v |q| \mu_X([v]) \]
\[ = -m' + 1 - \sum_v D(v,q),
\]
and since all $D(v,q)$ have bounds independent of $q$, $D(w,q)$ does as well, implying that $X$ is balanced for $w$ and so for all $m'$-letter words.

Combining all of this yields that $X \in \mathbf{UERB}_{m'}$, and since $X \in \mathbf{UE} \cap D$ was arbitrary, we have $\mathbf{UE} \cap D \subset \mathbf{UERB}_{m'} \subset \mathbf{UERB}_m$. Since $D$ is a subcylinder of $C$ and $C$ was arbitrary, this completes the proof that $\mathbf{UERB}_m$ is the intersection of $\mathbf{UE}$ with an open dense set, and therefore the entire proof.

\end{proof}

For a minimal Cantor system $(X,T)$, we define the infinitesimal subgroup $\infin \subset \mathcal{G}_{T}$ by
$$\infin = \{[f] \in \mathcal{G}_{T} \mid \int f \, d\mu = 0 \textnormal{ for all } T-\textnormal{invariant Borel probability measures } \mu\}.$$

Any $T$-invariant Borel probability measure $\mu$ induces a state, i.e. an order-preserving homomorphism $\tau_{\mu} \colon (\coinv{T},\coinv{T}^{+},[1]) \to (\mathbb{R}, \mathbb{R}_{+},1)$ taking $[1]$ to 1, defined by $\tau_{\mu}([f]) = \int_{X} f\, d\mu$.

If $(X,T)$ is a uniquely ergodic minimal Cantor system with $T$-invariant probability measure $\mu_X$ then there is an exact sequence
$$0 \to \infin \longrightarrow \mathcal{G}_{T} \stackrel{\tau_{\mu_X}}\longrightarrow \mathbb{R} \to 0.$$

It is straightforward to check that the image of $\tau_{\mu_X}$ in $\mathbb{R}$ is the subgroup generated by $\{\mu_X(W) \mid W \textnormal{ is clopen in } X\}$.

\begin{theorem}\label{thm:genericQimage}
The set of uniquely ergodic subshifts for which $\textnormal{Image}(\tau_{\mu_X}) = \mathbb{Q}$ is residual in $\Tcalprimebar$.
\end{theorem}
\begin{proof}
By Theorem~\ref{thm:UERBgeneric}, the set of uniquely ergodic subshifts $X$ for which the image of $\tau_{\mu_{X}}$ is contained in $\mathbb{Q}$ is generic in $\Tcalprimebar$. To obtain the equality, first note that if $(X,T)$ is any uniquely ergodic minimal Cantor system, then the image of $\tau_{\mu_X}$ contains the (additive) group of continuous eigenvalues of $(X,T)$ (see~\cite[Prop. 11]{CDP2016} for a proof of this). It follows from Theorem~\ref{univ} that the set of subshifts whose group of continuous eigenvalues contain $\mathbb{Q}$ is generic in $\Tcalprimebar$, so this completes the proof.
\end{proof}

\begin{theorem}\label{thm:trivinfgeneric}
The set of minimal subshifts whose infinitesimal subgroup is trivial is residual in $\Tcalprimebar$.
\end{theorem}

\begin{proof}
We claim that any minimal subshift in $\mathbf{UERB}$ has no nontrivial infinitesimals; together with Theorem~\ref{thm:UERBgeneric}, this will complete the proof. Consider any $X \in \mathbf{UERB}$ with unique invariant measure $\mu_{X}$. Since $X$ is balanced for factors, for each $w \in L_n(X)$, there exists $C_w$ so that for every $v \in L(X)$, $|D(w,v)| \leq C_w$.

Choose any function $f \in C(X, \mathbb{Z})$ with $\int f \ d\mu_{X} = 0$. Clearly there exists $n$ so that $f$ can be written as $\sum_{w \in L_n(X)} \alpha_w \chi_{[w]}$ where $\sum \alpha_w \mu_{X}([w]) = 0$. For any $x \in X$ and $N$, denote
$v = x([0,N+n))$. Then,
\begin{equation}\label{infeqn1}
\sum_{i=0}^{N-1} f(\sigma^i x) = \sum_{w \in L_n(X)} \alpha_w |v|_w.
\end{equation}
Therefore,
\begin{equation}\label{infeqn2}
\left| \sum_{w \in L_n(X)} \alpha_w |v|_w - \sum_{w \in L_n(X)} \alpha_w |v|\mu_{X}([w]) \right| \leq \sum_{w \in L_n(X)} |\alpha_w D(w,v)| \leq \sum_w C_w |\alpha_w|.
\end{equation}
Finally, since $\int f \ d\mu_{X} = 0$,
\begin{equation}\label{infeqn3}
\sum_{w \in L_n(X)} \alpha_w |v|\mu_{X}([w]) = |v| \int f \ d\mu_{X} = 0.
\end{equation}
Combining (\ref{infeqn1})-(\ref{infeqn3}) yields
\[
\left|\sum_{i=0}^{N-1} f(\sigma^i x)\right| \leq \sum_w C_w |\alpha_w|.
\]
Since the right-hand side is independent of $x$ and $N$, by Gottschalk-Hedlund (\cite{GottschalkHedlund}), $f$ is a coboundary. Since $f \in C(X,\mathbb{Z})$ was arbitrary with integral $0$, the proof is complete.
\end{proof}

We can now prove Theorem~\ref{thm:dimgrpQ}.

\begin{proof}[Proof of Theorem~\ref{thm:dimgrpQ}]
By Corollary~\ref{cor:zemuegeneric}, the set of uniquely ergodic minimal subshifts is residual in $\Tcalprimebar$. Theorem~\ref{thm:trivinfgeneric} implies the set of such systems which have trivial infinitesimal subgroup is residual in $\Tcalprimebar$ as well. Then the intersection of these two residual sets is residual, and for any system in this intersection, the triple $(\coinv{\sigma},\coinv{\sigma}^{+},[1])$ is isomorphic (as a unital ordered group) to its image under $\tau_{\mu_X}$, which by Theorem~\ref{thm:genericQimage} is generically $(\mathbb{Q},\mathbb{Q}_{+},1)$.
\end{proof}

\subsubsection{Mapping class groups in $\Tcalprimebar$}
We finish this section with an analysis of the mapping class group $\mathcal{M}(\sigma_X)$ of a generic subshift $X$ in $\Tcalprimebar$. Given two systems $(X,T)$ and $(Y,S)$, recall that a flow equivalence is an orientation preserving homeomorphism between their suspensions $\phi \colon \Sigma_{T}X \to \Sigma_{S}Y$. If a flow equivalence exists between two systems, then their mapping class groups are isomorphic (see~\cite{SchmiedingYang2021} for details).

By Corollary~\ref{autgrp}, the automorphism group of a generic subshift in $\Tcalprimebar$ is as small as possible, i.e. is generated by the shift map. We'll prove that the analogous result holds for mapping class groups: namely, that the mapping class group of generic subshift in $\Tcalprimebar$ is trivial.

\begin{remark}
There is no subshift $(X,\sigma)$ whose flow equivalence class is generic in $\Tcalprimebar$ since, as a consequence of the Parry-Sullivan Theorem (see~\cite[Sec. 4]{BCE2017}), the flow equivalence class within $\Scal$ of a subshift is always countable.
\end{remark}


\begin{theorem}\label{thm:mcgfortransitivesubshifts}
The collection of subshifts whose mapping class group is trivial is residual in $\Tcalprimebar$.
\end{theorem}
\begin{proof}
Let $\mathbf{C}$ denote set of uniquely ergodic minimal subshifts $(X,\sigma_{X})$ in $\Tcalprimebar$ whose dimension group is isomorphic (as a group) to $\mathbb{Q}$. By Theorem~\ref{thm:dimgrpQ} this set is generic in $\Tcalprimebar$, so it suffices to show that any subshift $(X,\sigma_{X})$ in $\mathbf{C}$ has trivial mapping class group.
In~\cite[Cor. 4.23]{SchmiedingYang2021} it is shown that if $(X,\sigma_{X})$ is a uniquely ergodic minimal subshift such that $\textnormal{Inf}(\coinv{\sigma_{X}}) = \{\textnormal{id}\}$, then either $(X,\sigma_{X})$ is flow equivalent to a subshift arising from a primitive substitution, or $\mathcal{M}(\sigma_{X})$ is isomorphic to $\textnormal{Aut}(X,\sigma_{X}) / \langle \sigma_{X} \rangle$. By Theorem~\ref{thm:trivinfgeneric} we know that the set of subshifts having trivial infinitesimal subgroup is generic in $\Tcalprimebar$. Moreover, by Corollary~\ref{autgrp} the set of subshifts $(X,\sigma_{X})$ for which $\aut(X,\sigma_{X}) / \langle \sigma_{X} \rangle$ is trivial is also generic in $\Tcalprimebar$, so it suffices to show that any subshift $(X,\sigma_{X})$ in the class $\mathbf{C}$ is not flow equivalent to a substitution.

If two systems $(X,T)$ and $(Y,S)$ are flow equivalent then there is an isomorphism between their coinvariant groups; that is, $\coinv{T}$ is isomorphic (as an abelian group) to $\coinv{S}$. Thus given Theorem~\ref{thm:dimgrpQ}, it is enough to show that if $(X,\sigma_{X})$ is a subshift coming from a primitive aperiodic substitution, then $\coinv{\sigma_{X}}$ is not isomorphic to $\mathbb{Q}$. One can see this for example using Bratteli diagrams: by~\cite[Prop. 20]{DHS1999}, if $(X,\sigma_{X})$ is a subshift associated to a substitution then $(X,\sigma_{X})$ is conjugate to the Vershik map on some stationary Bratteli diagram\footnote{Alternatively one could note that $\coinv{\sigma_{X}}$ is isomorphic to the Cech cohomology group $\check{H}^{1}(\Sigma_{\sigma_{X}}X,\mathbb{Z})$ which, in the case $(X,\sigma_{X})$ comes from a primitive substitution, can be computed using theory from tiling spaces; see~\cite{BargeDiamond}.}. In particular, if $(X,\sigma_{X})$ is a subshift defined by a primitive substitution, then there exists an $r \times r$ integral matrix $A$ such that $\coinv{\sigma_{X}}$ is isomorphic to the direct limit of the stationary system $\mathbb{Z}^{r} \stackrel{x \mapsto Ax}\longrightarrow \mathbb{Z}^{r} \stackrel{x \mapsto Ax}\longrightarrow \mathbb{Z}^{r} \longrightarrow \cdots$. This direct limit group is isomorphic to the direct limit of a stationary system $\mathbb{Z}^{r} \stackrel{x \mapsto A^{\prime}x}\longrightarrow \mathbb{Z}^{r} \stackrel{x \mapsto A^{\prime}x}\longrightarrow \mathbb{Z}^{r} \longrightarrow \cdots$ where $A^{\prime}$ is a nonsingular integral matrix. But $A^{\prime}$ is invertible over $\mathbb{Z}[\frac{1}{\textnormal{det}(A^{\prime})}]$, so the latter direct limit is isomorphic to a subgroup of $\mathbb{Z}[\frac{1}{\textnormal{det}(A^{\prime})}]^{r}$, and such a subgroup can not be isomorphic to $\mathbb{Q}$.
\end{proof}



\section{The space $\TTcalprime$ of infinite totally transitive systems}\label{tottrans}
Continuing with our study of genericity in more dynamically interesting subspaces, we consider in this last section the space of totally transitive subshifts contained in $\Scal$. As was done previously, we need to remove all isolated points in $\Scal$, but here that is particularly simple; the only subshifts in $\TTcal \cap \nmc$ are systems consisting of a single point (and hence defined by a single constant sequence). We then make the following definition.


\begin{definition}
Denote by $\TTcal'$ the subspace of infinite totally transitive subshifts in $\Scal$.
\end{definition}
It turns out that each $X_n$ in the sequence used to show that $\Tcalprime$ was not closed was in fact also totally transitive, so the same sequence of systems shows that $\TTcal'$ is not closed in $\Scal$ either.

For every $k$, a proof virtually identical to that of Lemma~\ref{TGdelta} shows that the set of subshifts $(X, \sigma)$ for which $(X, \sigma^k)$ is transitive is a $G_{\delta}$. Therefore, the space $\TTcalprime$ is a $G_{\delta}$ in $\Scal$, and hence also in $\TTcalprimebar$. Thus again, any results on genericity in $\TTcalprime$ may be determined in $\TTcalprimebar$, and we phrase all our results here in terms of $\TTcalprimebar$.


\begin{lemma}
$\TTcalprimebar$ is a perfect subset of $S$.
\end{lemma}

\begin{proof}
By definition, $\TTcalprimebar$ is closed, so we need only show that it has no isolated points.

The proof of this is again similar to that of Theorem~\ref{perfect}. We claim that if the shift $X$ in that proof is assumed totally transitive, then the shift $Y$ constructed to share an arbitrary cylinder $[X,n]$ is also totally transitive. Recall that $Y$ was constructed by forbidding a single path $fPg$, but allowing all $fPK^n g$ for a fixed cycle $K$, and that it was shown in Lemma~\ref{perfect2} that $Y$ is a transitive shift of finite type. Since $[X,n]$ contains a totally transitive infinite subshift, its Rauzy graph $G_{X,n}$ is primitive and nontrivial, so we can find cycles $K', K''$ which both contain $fPg$ and have relatively prime lengths. We can assume without loss of generality that $K', K''$ end with $fPg$.\\ 

Suppose the numbers of occurrences of $fPg$ in $K', K''$ are $m', m'' > 0$ respectively. Then, for every $i > 0$, $G_{X,n}$ contains cycles $K'_i, K''_i$ obtained by replacing each $fPg$ in $K', K''$ respectively by $fPK^i g$. Then $|K'_i| = |K'| + im'|K|$ and $|K''_i| = |K''| + im''|K|$. Since $|K'|$ and $|K''|$ were relatively prime, for large enough $j$, $|K'_{jm''}|$ and $|K''_{jm'}|$ are also relatively prime (since they come from adding the same large integer $jm'm''|K|$ to both $|K'|$ and $|K''|$.)

Therefore, we have two cycles $L = K'_{jm''}$ and $L' = K''_{jm'}$ of relatively prime length in $G_{X,n}$ which do not contain $fPg$ and which end with $g$. Since $L$ and $L'$ end with $g$, each yields a biinfinite path (under repeated traversal) which does not contain $fPg$ and so corresponds to a periodic point of $Y$. Then $Y$ is a transitive shift of finite type with periodic points of relatively prime least periods, and so it is mixing and thereby totally transitive.

So, again all nonempty cylinders in $\TTcalprimebar$ have at least two subshifts, so no subshift in $\TTcalprimebar$ is isolated.
\end{proof}



One useful observation is that the topologically mixing subshifts are dense in $\TTcalprimebar$.

\begin{lemma}\label{ttmix}
For any nonempty cylinder $C$ in $\TTcalprimebar$, the subshift $S(C) \in C$ is an infinite mixing shift of finite type.
\end{lemma}

\begin{proof}
This is fairly clear; by definition, $C$ must be equal to the intersection of $[X, n]$ with $\TTcalprimebar$ for some subshift $X \in \TTcalprimebar$. Since this cylinder must intersect $\TTcalprime$, we may assume that $X$ itself is infinite and totally transitive. The Rauzy graph $G_{X,n}$ must then be nontrivial, irreducible, and aperiodic, therefore primitive. Then, by definition, $G_{X,n}$ is the graph defining the $n$th higher block presentation (see~\cite[Sec. 1.4]{LM}) of $S(C)$, so this presentation is infinite and mixing. Since the higher block presentation is topologically conjugate to $S(C)$, it follows $S(C)$ must be infinite and mixing as well.
\end{proof}

Our main tool for proving that various sets/properties are dense in $\TTcalprimebar$ is the following theorem, which plays the role that Theorem~\ref{letword} did in $\Tcalbar$.  The result closely mimics that of~\cite[Thm. 6.4]{HochmanGeneric}, and
Hochman alludes to (but does not prove) the version here.

\begin{theorem}\label{conjdense}
For any subshift $X \in \TTcalprimebar$ which has zero entropy and no periodic points, the conjugacy class of $X$ is dense in $\TTcalprimebar$.
\end{theorem}

\begin{proof}
Choose any such $X$ and any nonempty cylinder $C$ in $\TTcalprimebar$. By Lemma~\ref{ttmix}, there is an infinite mixing shift of finite type $Y = S(C)$ and $n$ so that
$C = [Y, n]$. We note that Krieger's embedding theorem would immediately allow us to construct an embedding of $X$ into $Y$, which yields a subshift of $Y$ conjugate to $X$. However, this shift might not contain all words in $L_n(Y)$ in its language, and so might not be in $C$. To obviate this issue, we will construct a mixing shift of finite type subsystem of $Y$ where every point contains all words in $L_n(Y)$.

By primitivity of $G_{Y,n}$, there exist cycles $K, K'$ with relatively prime lengths which each contain all edges of $G_{Y,n}$. Without loss of generality, we can assume that $K, K'$ start and end at the same vertex $v$. Define a labeled directed graph $\mathcal{G}$ consisting of copies of the cycles $K$, $K'$ which share the vertex $v$ and no others, and where each edge is labeled by the initial letter of the word in $L_{n}(Y)$ it corresponds to. Define a subshift $Z$ consisting of all labels of biinfinite walks on $\mathcal{G}$; by definition, $Z$ is irreducible and sofic (see \cite{LM} for a definition of sofic subshift). It is clear by definition that $Z$ consists of all sequences in $Y$ corresponding to biinfinite concatenations (in any order) of $K, K'$, so $Z$ has periodic points of relatively prime periods and is therefore mixing.

Since $K, K'$ each contained all edges of $G_{Y,n}$, every $z \in Z$ contains all words in $L_n(Y)$ as subwords. Moreover, since $Z$ is mixing sofic, it has a synchronizing word $w$ and there exist words $u, v$ with lengths differing by $1$ so that $wuw, wvw \in L(Z)$. Then, the set of all configurations of concatenations of $wu, wv$ (in any order) is a subshift $Z'$ of $Z$ (this uses the synchronizing property of $w$; again, see \cite{LM} for a definition), and it's easy to see that it's a shift of finite type. It is also mixing, since it contains periodic sequences $(wu)^{\infty}$ and $(wv)^{\infty}$ with periods differing by $1$ (and therefore relatively prime). 

By Krieger's embedding theorem~\cite[Cor. 10.1.9]{LM}, there is an embedding from $X$ to $Z'$ (as a mixing shift of finite type, $Z'$ automatically has positive entropy, so $h(Z') > h(X)$). Now, $Z'$ contains a shift conjugate to $X$, which is in $C$ by Lemma~\ref{rauzy}, completing the proof.
\end{proof}

This theorem is quite powerful, showing that any set defined by a conjugacy-invariant dynamical property which is possessed by some zero entropy aperiodic totally transitive subshift is dense in $\TTcalprimebar$. In particular, any such property/set that we have already shown to be a $G_{\delta}$ is automatically generic, so we have the following.

\begin{theorem}\label{ttgeneric}
The set of zero entropy, minimal, uniquely ergodic subshifts with topological rank two is residual in $\TTcalprimebar$.
\end{theorem}
(Sturmian subshifts are examples showing that these properties satisfy the hypothesis of Theorem~\ref{conjdense}.)

We emphasize that Toeplitz subshifts are no longer generic in $\TTcalprimebar$ (as was the case for $\Tcalbar$), since no Toeplitz subshift in $\Scalprime$ is totally transitive.

Unlike the transitive case, we can show that topologically mixing is generic in $\TTcalprimebar$.

\begin{theorem}\label{ttmixgen}
The set $\mathbf{TM}$ of topologically mixing subshifts is residual in $\TTcalprimebar$.
\end{theorem}

\begin{proof}
Density is an immediate corollary of Theorem~\ref{conjdense} (\cite{HK} gives examples of topologically mixing subshifts with zero entropy), so we need only show that the mixing subshifts form a $G_{\delta}$ in $\Scal$ (and thereby $\TTcalprimebar$). To see this, we note that for any finite $\A \subset \mathbb{Z}$, a subshift $X \in \Scal[\A]$ is topologically mixing if and only if, for all $n$, there exists $k$ so that, for all $v,w \in L_n(X)$, there exists $u \in L_k(X)$ for which $vuw \in L_{2n+k}(X)$. An equivalent statement (since languages are factorial) is: for all $n$, there exists $k$ so that, for all $v, w \in \A^n$ which are subwords of some words in $L_{2n+k}(X)$, there exists $u \in \A^k$ for which $vuw \in L_{2n+k}(X)$.

For any $\A,n,k$, denote by $M(\A,n,k)$ the set of all $S \subset \A^{2n+k}$ with this property. Then, we can write
\[
\mathbf{TM} = \bigcap_{n \in \mathbb{N}} \bigcup_{\substack{k \in \mathbb{N},\\ \A \subset \mathbb{Z}, |\A| < \infty}} \{X \in \Scal \mid L_{2n+k}(X) \in M(\A,n,k)\} \]
\[ =
\bigcap_{n \in \mathbb{N}} \bigcup_{\substack{k \in \mathbb{N},\\ \A \subset \mathbb{Z}, |\A| < \infty}} \bigcup_{\substack{X \in \Scal \\L_{2n+k}(X) \in M(\A,n,k)}} [X,2n+k],
\]
which is clearly a $G_{\delta}$, completing the proof. 
\end{proof}

\subsection{Measure-theoretic properties for generic subshifts in $\TTcalprimebar$}
Theorem~\ref{ttgeneric} shows that generic subshifts in $\TTcalprimebar$ are uniquely ergodic, and we can prove that properties of the unique measure for a generic subshift mimic those in the simplex of invariant measures. There is one subtlety; since the space of (shift-invariant) measures on $\mathbb{Z}^{\mathbb{Z}}$ is not complete in the weak topology (and the issue persists even for measures supported on finite-alphabet full shifts), we must restrict to a fixed ambient finite-alphabet full shift in our statements. 
We first need a fact about genericity in the measure-theoretic setting.

\begin{lemma}\label{ZWM}
For any finite $\A \subset \mathbb{Z}$ with associated space $\mathcal{M}(\A)$ of shift-invariant probability measures on $\A^{\mathbb{Z}}$ endowed with the weak topology, the set of weak mixing measures with zero entropy is residual in $\mathcal{M}(\A)$.
\end{lemma}

\begin{proof}
The fact that weak mixing measures are generic in $\mathcal{M}(\A)$ comes from \cite{partha1} and \cite{partha2}; though the proof in
\cite{partha2} is stated for measures on $\mathbb{R}^\mathbb{Z}$, it rests on results from \cite{partha1} which were proved for a finite alphabet.

To see that zero entropy measures are generic, we first note that \cite{Sigmund1971} implies that the set of measures supported on a finite subshift (i.e. a finite union of periodic orbits) is dense in $\mathcal{M}(\A)$. Finally, since the shift on $\A^\mathbb{Z}$ is expansive, the entropy map $\mu \mapsto h(\mu)$ is upper semi-continuous, meaning that for all $n$, the set of shift-invariant measures on $\A^\mathbb{Z}$ with measure-theoretic entropy less than $n^{-1}$ is open. Taking the intersection shows that the zero entropy measures form a $G_{\delta}$ set, and since all measures with finite support are zero entropy, they are dense as well, completing the proof.

\end{proof}


We can now state our main measure-theoretic result. Note that our proof is similar to that in~\cite{HochmanGeneric}, but the result does not directly imply ours due to the differences in settings.

\begin{theorem}\label{measure}
For any finite $\A \subset \mathbb{Z}$ with $|\A| > 1$ and any nonempty set $G$ of ergodic shift-invariant probability measures which is closed under measure-theoretic isomorphisms on $\A^{\mathbb{Z}}$ and which is a dense $G_{\delta}$ in the space $\mathcal{M}(\A)$ of invariant measures on $\A^{\mathbb{Z}}$, a generic subshift in $\Scal[\A] \cap \TTcalprimebar$ is uniquely ergodic with unique measure in $G$.
\end{theorem}

\begin{proof}
Fix any $\A, G$ as in the theorem. If $\mathbf{UE}(\A)$ is the set of uniquely ergodic subshifts contained in $\A^{\mathbb{Z}}$ and
$\mathcal{E}(\A)$ is the set of ergodic shift-invariant probability measures on
$\A^\mathbb{Z}$, then we claim that the function $f: \mathbf{UE}(\A) \rightarrow \mathcal{E}(\A)$ defined by $X \mapsto \mu_X$ is continuous.

To see this, consider any $X \in \mathbf{UE}(\A)$ and any $k, \epsilon > 0$. By definition of unique ergodicity, for every $u \in L_k(X)$, there exists $N_u$ so that for every $M > N_u$, every $v \in L_M(X)$ contains between $M(\mu_X([u]) - \epsilon/2)$ and $M(\mu_X([u]) + \epsilon/2)$ occurrences of $u$. Then, if we define $N = \max_u N_u$ and take any $Y \in \mathbf{UE}(\A) \cap [X, N]$, then for every $u \in L_k(X)$, every word in $L_N(Y) = L_N(X)$ contains between $N(\mu_X([u]) - \epsilon/2)$ and $N(\mu_X([u]) + \epsilon/2)$ occurrences of $u$. If $N$ was assumed large enough, this implies that for each such $u$ and every $y \in Y$,
the frequency of occurrences of $u$ in $y$ is between $\mu_X([u]) - \epsilon$ and $\mu_X([u]) + \epsilon$. If we define $\nu = f(Y)$, then
by the ergodic theorem, $|\mu_X([u]) - \nu([u])| < \epsilon$ for all $u \in L_k(X)$. Since this is true for $f(Y)$ for all $Y \in \mathbf{UE}(\A) \cap [X, N]$ and since $k, \epsilon$ were arbitrary, we've proven the desired continuity.

Therefore, $f^{-1}(G)$ is a $G_{\delta}$ subset of $\mathbf{UE}(\A)$ in the induced topology, i.e. the intersection of $\mathbf{UE}(\A)$ with a $G_\delta$ set in $\Scal$. We now wish to show that $f^{-1}(G)$ is dense in $\Scal[\A] \cap \TTcalprimebar$.

Since $G$ is a dense $G_\delta$ and the set of weak mixing measures with zero entropy is generic in $\mathcal{M}(\A)$ by Lemma~\ref{ZWM},
$G$ contains a weak mixing measure $\mu$ with zero entropy. By the Jewett-Krieger Theorem, there is a uniquely ergodic subshift
$X \in \Scal[\A]$ whose unique ergodic measure $\mu_X$ is measure-theoretically isomorphic to $\mu$, so $\mu_X \in G$. Since $\mu$ was  weak mixing with zero entropy, $\mu_X$ is also, and so $X$ has zero (topological) entropy, no periodic points, and is totally transitive. So, by Theorem~\ref{conjdense}, the conjugacy class $C(X)$ of $X$ is dense in $\TTcalprimebar$. Therefore, $C(X) \cap \Scal[\A]$ is dense in $\Scal[\A] \cap \TTcalprimebar$. Since $f(X) = \mu_X \in G$ and every topological conjugacy between uniquely ergodic subshifts induces a measure-theoretic isomorphism, $f(C(X) \cap \Scal[\A]) \subset G$, and hence $f^{-1}(G)$ is dense in $\Scal[\A] \cap \TTcalprimebar$.

Therefore, $f^{-1}(G)$ is the intersection of $\mathbf{UE}(\A)$ with a dense $G_\delta$ set in $\TTcalprimebar$. Finally, we note that
$\mathbf{UE}$ is generic in $\TTcalprimebar$ by Theorem~\ref{ttgeneric}, $\mathbf{UE}(\A) = \mathbf{UE} \cap \Scal[\A]$, and $\Scal[\A]$ is clopen in $\Scal$.
This implies that $f^{-1}(G)$ is generic in $\Scal[\A] \cap \TTcalprimebar$, completing the proof.

\end{proof}



\begin{theorem}\label{fullWM}
The set $\mathbf{WM}$ of uniquely ergodic subshifts with weak mixing unique measure is residual in $\TTcalprimebar$.
\end{theorem}

\begin{proof}
For any alphabet $\A$ with $|\A| > 1$, the set $G(\A)$ of weakly mixing measures on $\A^{\mathbb{Z}}$ satisfies the hypotheses of Theorem~\ref{measure} (by Lemma~\ref{ZWM}), so by Theorem~\ref{measure}, $\mathbf{WM} \cap \Scal[\A]$ is residual in
$\TTcalprimebar \cap \Scal[\A]$. (We need not consider singleton $\A$ since the only such subshifts are isolated in $\Scal$ and so not
in $\TTcalprimebar$.) But then by unioning over all finite $\A \subset \mathbb{Z}$, we see that
\[
\mathbf{WM} = \bigcup_{\A} \mathbf{WM} \cap \Scal[\A] \textrm{ is residual in } \bigcup_{\A} \TTcalprimebar \cap \Scal[\A] = \TTcalprimebar.
\]
\end{proof}

We now show that the existence of a rigidity sequence (a sequence $\{n_k\}$ is a rigidity sequence for a shift-invariant measure
$\mu$ if $\mu(A \triangle \sigma^{n_k} A) \rightarrow 0$ for all measurable $A$) is also generic. We are not aware of a reference for this property being generic in the space of shift-invariant measures on a finite alphabet, so we give a self-contained proof of this fact.

\begin{theorem}\label{rigid}
The set of uniquely ergodic subshifts in $\Scal$ with a rigidity sequence is residual in
$\TTcalprimebar$. 
\end{theorem}

\begin{proof}
Define $\mathbf{R}$ to be the set of uniquely ergodic subshifts in $\Scal$ with a rigidity sequence. Clearly $\mathbf{R}$ is conjugacy-invariant, and there exists a totally transitive uniquely ergodic subshift with zero entropy, no periodic points, and a rigidity sequence (e.g. any Sturmian shift), so by Theorem~\ref{conjdense}, $\mathbf{R}$ is dense in $\TTcalprimebar$. 
It remains to show that $\mathbf{R}$ is a $G_{\delta}$ in $\Scal$ (and therefore in $\TTcalprimebar$ for the induced topology).

We first claim that a uniquely ergodic subshift $X$ has a rigidity sequence if and only if the following holds: for all $n$ and $\epsilon$, there exist $M, N$ so that for every word $v \in L_N(X)$, the proportion of locations $i$ where the same $n$-letter word appears at locations $i$ and $i+M$ is greater than $1 - \epsilon$.

For the forward direction, first note that if $X \in \mathbf{UE} \cap \Scal[\A]$ and $(n_k)$ is a rigidity sequence for $\mu_X$, then for any $n, \epsilon$, there exists $k$ so that for each $u \in \A^n$, $\mu_X([u] \triangle \sigma^{n_k} [u]) < \frac{\epsilon}{2|\A|^n}$. Define $M = n_k$. Then $\mu_X \left( \bigcup_{u \in \A^n} ([u] \cap \sigma^M [u]) \right) > 1 - |\A|^n \frac{\epsilon}{2|\A|^n} = 1 - \epsilon/2$. Then, by unique ergodicity, there exists $N$ so that every $v \in L_N(X)$ has proportion at least $1 - \epsilon$ of locations in $\bigcup_{u \in \A^n} ([u] \cap \sigma^M [u])$.

For the reverse direction, assume that for all $n, \epsilon$, there exist $M, N$ as described. For each $k$, define $\epsilon = \frac{1}{k|\A|^k}$, define $M, N$ associated to $k, \epsilon$, and take $n_k = M$. Then every word $v \in L_N(X)$ has proportion at least $1 - \epsilon$ of locations $i$ where $i$ and $i+n_k$ contain the same $n$-letter word. For every $x \in X$, $x$ is a concatenation of words in $L_N(X)$, and so the proportion of locations in $x$ with the same property is at least $1 - \epsilon$. Then by the ergodic theorem,
$\mu_X \left( \bigcup_{u \in \A^n} ([u] \cap \sigma^{n_k} [u]) \right) > 1 - \epsilon$. This implies that for every $u \in \A^n$,
$\mu_X([u] \cap \sigma^{n_k} [u]) > \mu_X([u]) - \epsilon = \mu_X([u]) - \frac{1}{k|\A|^k}$. Now, define $P(u, k')$ to be the set of all $u' \in \A^{k'}$ with $u$ as a prefix. Then for any $k' > k$,
\[
\mu_X([u] \cap \sigma^{n_{k'}} [u]) \geq \sum_{u' \in P(u, k')} \mu_X([u'] \cap \sigma^{n_{k'}} [u']) \]
\[ > \sum_{u' \in P(u, k')}
\mu_X([u']) - \frac{1}{k'|\A|^{k'}} = \mu_X([u]) - \frac{1}{k' |\A|^k},
\]
where the last equation uses the definition of $P(u, k')$ and the fact that $|P(u, k')| = |\A|^{k'-k}$. So,
$\displaystyle \mu_X([u] \cap \sigma^{n_k} [u]) \rightarrow \mu_X([u])$ as $k \rightarrow \infty$. Since $u$ was arbitrary and since cylinder sets generate the Borel $\sigma$-algebra, $(n_k)$ is a rigidity sequence.

Finally, for any $k, \epsilon, N, M$ and finite $\A \subset \mathbb{Z}$, we define $R(\A, n, \epsilon, M, N)$ to be the set of subsets $S \subset \A^N$ where for each
$v \in S$, the proportion of locations $i$ where the same $n$-letter word appears at locations $i$ and $i+M$ is greater than $1 - \epsilon$. Then by the above argument, we can write
\begin{multline*}
\mathbf{R} = \bigcap_{m,n \in \mathbb{N}} \bigcup_{\substack{M, N \in \mathbb{N}\\ \A \subset \mathbb{Z}, |\A| < \infty}} \{X \in
\mathbf{UE} \mid L_{N}(X) \in R(\A, n, m^{-1}, M, N)\} = \\
\mathbf{UE} \cap \bigcap_{m,n \in \mathbb{N}} \bigcup_{\substack{M, N \in \mathbb{N}\\ \A \subset \mathbb{Z}, |\A| < \infty}} \bigcup_{\substack{X \in \Scal \\L_N(X) \in R(\A, n, m^{-1}, M, N)}} [X,N].
\end{multline*}
Since $\mathbf{UE}$ was already known to be a $G_{\delta}$, this shows that $\mathbf{R}$ is a $G_{\delta}$, completing the proof.
\end{proof}

\subsection{Complexity for generic subshifts in $\TTcalprimebar$}

Much as in the transitive case, we will be able to show that in some sense all achievable complexity growth rates are realized along subsequences for a generic subshift in $\TTcalprimebar$. The proofs will be simpler, as instead of the somewhat complicated Theorem~\ref{letword}, we can just use Theorem~\ref{conjdense} (along with Theorem~\ref{cplxthm}).

\begin{definition}
For two increasing functions $f,g: \mathbb{N} \rightarrow \mathbb{R}^+$, we say that $f <_s g$ if for every $t \in \mathbb{N}$, there exists a constant $N$ so that $f(n + t) < g(n)$ for all $n > N$. 
\end{definition}

We note that this a weaker notion of inequality than $\prec$ defined in Section~\ref{cplxt}. The following is an easy consequence of the Curtis-Hedlund-Lyndon theorem.

\begin{lemma}\label{cplxbd2}
If $X$ and $Y$ are conjugate subshifts, then there exists $t \in \mathbb{N}$ so that $c_{X}(n-t) < c_{Y}(n) < c_{X}(n+t)$ for all $n$.
\end{lemma}

Recall for functions $f,g \colon \mathbb{N} \rightarrow \mathbb{R}^+$ we let $\mathbf{S}_{f,g}$ denote the collection of subshifts $X$ for which $f(n) \leq c_{X}(n) \leq g(n)$ for infinitely many $n$.

\begin{theorem}\label{mainfg2}
Suppose $f,g: \mathbb{N} \rightarrow \mathbb{N}$ are increasing and $f <_s g$. Suppose further that there exists a totally transitive zero entropy subshift $X$ without periodic points and a sequence
$(k_n)$ where $f(k_n) <_s c_{k_n}(X) <_s g(k_n)$. Then the set $\mathbf{S}_{f,g}$ defined in Theorem~\ref{cplxthm} is residual in $\TTcalprimebar$.
\end{theorem}

\begin{proof}
That $\mathbf{S}_{f,g}$ is a $G_{\delta}$ is implied by Theorem~\ref{cplxthm}, so we need only show that it is dense.

We claim that this follows from Theorem~\ref{conjdense} and Lemma~\ref{cplxbd2}. Indeed, choose a subshift $X$ and sequence $(k_n)$ where
$f(k_n) <_s c_{X}(k_n) <_s g(k_n)$, and any cylinder $C$ in $\TTcalprimebar$. By Theorem~\ref{conjdense}, there exists
$Y \in C$ conjugate to $X$, and by Lemma~\ref{cplxbd2}, there exists $t \in \mathbb{N}$ so that for all $n$,
\begin{equation}\label{fgcplx0b}
c_{X}(k_n - t) \leq c_{Y}(k_n) \leq c_{X}(k_n + t).
\end{equation}
We note that by the definition of $<_s$, for large enough $n$,
\begin{equation}\label{fgcplx1b}
c_{X}(k_n + t) < g(k_n).
\end{equation}
Similarly, for large enough $n$,
\begin{equation}\label{fgcplx2b}
f(k_n) < c_{X}(k_n - t).
\end{equation}
Combining (\ref{fgcplx0b})-(\ref{fgcplx2b}) yields
\[
f(k_n) < c_{Y}(k_n) < g(k_n)
\]
for sufficiently large $n$, implying that $Y \in \mathbf{S}_{f,g}$ and completing the proof.
\end{proof}

We do not have a proof of a version of Theorem~\ref{bigcplx} for totally transitive subshifts, since most examples of subshifts with complexity in certain regimes are Toeplitz or come from block concatenation constructions and are therefore not necessarily totally transitive. We strongly suspect that such a theorem does hold, but do not have a proof.





The following, however, is proven exactly as in the transitive case (since Sturmian subshifts are totally transitive).

\begin{corollary}\label{sosmall2}
For any unbounded increasing $h: \mathbb{N} \rightarrow \mathbb{R}^+$, $\mathbf{S}_{n, n+h(n)}$ is residual in $\TTcalprimebar$.
\end{corollary}

Just as for Corollary~\ref{sosmall}, this result cannot be improved: as before, $\mathbf{S}_{n, n+C}$ and $\mathbf{S}_{1.1n, 2n}$ are disjoint, and the Chacon shift is an example of a zero entropy aperiodic totally transitive shift in $\mathbf{S}_{1.1n, 2n}$. So, by Theorem~\ref{mainfg2}, $\mathbf{S}_{1.1n, 2n}$ is generic in $\TTcalprimebar$, so $\mathbf{S}_{n, n+C}$ cannot be.

Corollary~\ref{sosmall2} has an interesting consequence. Ferenczi in \cite{ferenczi95} asked whether some Chacon-type examples he constructed were the subshifts of minimal complexity which support weak mixing measures. In the sense of complexity along a subsequence, our results show that the only restriction is that $c_{X}(n) - n \rightarrow \infty$.

\begin{corollary}\label{WMlowcplxex}
For any unbounded increasing $f: \mathbb{N} \rightarrow \mathbb{R}^+$, the set of subshifts
$X$ which are uniquely ergodic with weak mixing invariant measure and for which there exist infinitely many $n$ with $c_{X}(n) < n + f(n)$ is residual in $\TTcalprimebar$.
\end{corollary}

\begin{proof}
This is an immediate corollary of Theorem~\ref{fullWM} and Corollary~\ref{sosmall2}.
\end{proof}


Finally, the following is proved exactly as in the transitive case.

\begin{corollary}\label{1rsTT}
The collection $\mathbf{U}_{1RS}$ of subshifts for which there are infinitely many $n$ with $c_{X}(n+1) = c_{X}(n) + 1$ is residual in $\TTcalprimebar$.
\end{corollary}

This yields the following corollaries, again exactly as in the transitive case.

\begin{corollary}\label{TTadic}
A generic subshift in $\TTcalprimebar$ has alphabet rank two for a right proper sequence of substitutions, and therefore has topological rank two.
\end{corollary}

\subsection{Dimension groups, infinitesimals, and orbit equivalence in $\TTcalprimebar$}

The biggest difference in the (generically unique) invariant measure for subshifts in $\TTcalprimebar$ versus $\Tcalprimebar$ is that it is no longer the case that measures of clopen sets are generically in $\mathbb{Q}$.

\begin{theorem}\label{killonenum}
For every $n \in \mathbb{N}$ and $x \in (0,1)$, the set $\mathbf{M}(n,x)$ of uniquely ergodic subshifts $X$ for which there is $S \subset L_n(X)$ with $\mu_X\left( \bigcup_{w \in S} [w]\right) = x$ is nowhere dense in $\TTcalprimebar$.
\end{theorem}

\begin{proof}
Consider any nonempty cylinder $C$ in $\TTcalprimebar$, $n \in \mathbb{N}$, and $x \in (0,1)$. We will show that 
$\mathbf{M}(n,x)$ is not dense in $C$.

There exists $N$ and infinite totally transitive $X$ so that $C = [X,N]$. We will assume without loss of generality that
$N = n$ (by setting each equal to the larger). Completing the proof in this case completes the general proof, since
$[X, N] \subset [X, n]$ if $n > N$, and since $\mathbf{M}(n,x) \subset \mathbf{M}(N, x)$ if $n < N$. Since $X$ is totally transitive, $G_{X,n}$ is nontrivial and primitive.

Therefore, we can find cycles $K, K'$ in $G_{X,n}$ which contain all edges of $G_{X,n}$, and which have lengths $L, L'$ which are relatively prime, where $L' < L$, and where $Lx \notin \mathbb{N}$. Then for any $S \subset \A^n$, the proportion of visits to edges corresponding to $S$ along $K$ is of the form $\frac{i}{L}$ for some $i$. Choose $M$ relatively prime to $L'$ so that $|x - \frac{i}{L}| > M^{-1}$ for every $i \in \mathbb{N}$.

Define $K'' = K^M$ (where $K^{M}$ denotes an $M$-fold traversal of $K$), and define the subshift $Z$ consisting of all sequences in $Y$ corresponding to biinfinite concatenations (in any order) of $K', K''$ where $K'$ does not appear consecutively. Since the lengths $L^{\prime}, ML$ of $K^{\prime}, K^{\prime \prime}$ are relatively prime, just as in the proof of Theorem~\ref{conjdense}, $Z$ is a mixing sofic shift. Choose any $S \subset L_n(X)$, and denote the proportions of visits of edges corresponding to $S$ along $K^{\prime}$ and $K^{\prime \prime}$ by $t$ and $\frac{i}{L}$ respectively.

Then for any measure $\mu$ on $Z$ and $\mu$-generic $z \in Z$, the limiting proportion of visits to words in $S$ along $z$ is between
$\frac{i}{L}$ and $\frac{t + (Mi/L)}{M+1}$. This implies that it is in the interval $(\frac{i}{L} - M^{-1}, \frac{i}{L} + M^{-1})$, and so is not equal to $x$. Therefore, any uniquely ergodic subshift contained in $Z$ is in $\mathbf{M}(n,x)^c$.

By definition, every $z \in Z$ contains all words in $L_n(X)$ as subwords. Exactly as shown in the proof of Theorem~\ref{conjdense}, $Z$ contains a mixing shift of finite type $Z'$. Define $m \geq n$ so that $Z'$ is defined by a set of $m$-letter forbidden words; note that then $Z' = S([Z',m])$. Then the cylinder $[Z', m]$ is contained in $[Z, n]$, which is in turn contained in $C$, and any uniquely ergodic subshift in $[Z', m]$ is a subset of $Z'$, so also in $\mathbf{M}(n,x)^c$. Thus $C$ contains a subcylinder disjoint from $\mathbf{M}(n,x)$, completing the proof.

\end{proof}

Let $\mathbf{UEI}$ denote the set of uniquely ergodic subshifts for which the only rational measures of clopen sets are $0$ and $1$.

\begin{corollary}\label{killQ}
The set $\mathbf{UEI}$ is residual in $\TTcalprimebar$.
\end{corollary}

\begin{proof}
We simply note that any $X$ in $\left(\bigcup_{q \in \mathbb{Q}, n \in \mathbb{N}} \mathbf{M}(n,q)\right)^c$ has the desired property, and by Baire category this set is residual in $\TTcalprimebar$.
\end{proof}

We can now show that the dimension group has rank two for generic subshifts in $\TTcalprimebar$.


\begin{proposition}\label{dimrk2}
The set of uniquely ergodic minimal subshifts whose dimension group is rank two and contains no nontrivial infinitesimals is residual in $\TTcalprimebar$.
\end{proposition}

\begin{proof}
We already know by Theorem~\ref{ttgeneric} that the set of uniquely ergodic minimal subshifts with topological rank two is residual and that the set $\mathbf{UEI}$ of uniquely ergodic subshifts where all clopen sets have measure $0$, $1$, or irrational is residual. Then the set of subshifts in $\mathbf{UEI}$ which have topological rank two is also residual. We will show that any $X \in \mathbf{UEI}$ which has topological rank two has a rank two dimension group and no nontrivial infinitesimals.

It is well-known (see e.g.~\cite{GPS1995}) that topological rank is an upper bound for the rank of the dimension group, and $X$ has topological rank two by assumption, so $\mathcal{G}_{\sigma_{X}}$ has rank at most two. Since $X \in \mathbf{UEI}$, the range of the state
$\tau_{\mu_X}$ contains both $1$ and some irrational, so the image of $\tau_{\mu_X}$ has rank at least two, implying that the rank of $\mathcal{G}_{\sigma_{X}}$ is exactly two. Since $\mathcal{G}_{\sigma_{X}}$ is torsion-free, this implies there are no nontrivial infinitesimals.
\end{proof}

Recall from Corollary~\ref{cor:onesoeclass} that in $\Tcalprimebar$ there is a single strong orbit equivalence class which is generic. In stark contrast, we now show that in $\TTcalprimebar$, no orbit equivalence class is generic. We will make use of the following fact: if two uniquely ergodic minimal Cantor systems $(X_{1},\sigma_{X_{1}})$ and $(X_{2},\sigma_{X_{2}})$ with invariant probability measures $\mu_{X_1}, \mu_{X_2}$ are orbit equivalent then the sets $\{\mu_{X_1}(E) \mid E \subset X_{1} \textnormal{ is clopen}\}$ and $\{\mu_{X_2}(F) \mid F \subset X_{2} \textnormal{ F is clopen}\}$ are equal (in fact, by~\cite[Cor. 1]{GPS1995} this condition is also sufficient in the uniquely ergodic minimal case).
\begin{corollary}\label{cor:nogenoeclass}
For any uniquely ergodic minimal subshift $X \in \TTcalprimebar$, the orbit equivalence class of $X$ in $\TTcalprimebar$ is meager. Consequently, no orbit equivalence class is generic in $\TTcalprimebar$.
\end{corollary}
\begin{proof}
\sloppy Let $(X,\sigma_{X})$ be a uniquely ergodic minimal subshift in $\TTcalprimebar$, define any nonempty clopen $C \subsetneq X$, and define $\alpha = \mu_X(C)$ (clearly $\alpha \notin \{0,1\}$). Suppose $(Y,\sigma_{Y})$ is orbit equivalent to $(X,\sigma_{X})$ via an orbit equivalence $\phi \colon Y \to X$; then $(Y, \sigma_Y)$ is also uniquely ergodic. By the fact above, there exists clopen $D$ so that $\mu_Y(D) = \alpha$. Therefore, $Y \in \bigcup_{n \in \mathbb{N}} \mathbf{M}(n,\alpha)$. We have shown that the orbit equivalence class of $(X,\sigma_{X})$ is contained in the set $\bigcup_{n \in \mathbb{N}} \mathbf{M}(n,\alpha)$, which is meager by Theorem~\ref{killonenum}.


The second statement follows from the fact that uniquely ergodic minimal subshifts are generic in $\TTcalprimebar$ by Theorem~\ref{ttgeneric}.
\end{proof}

Finally, we will show that a generic subshift in $\TTcalprimebar$ is (uniquely ergodic and) not balanced for any letter in the sense of Definition~\ref{baldef}. This is an interesting contrast to $\Tcalprimebar$; there we first proved that a generic subshift was balanced for factors and used that to show that a generic subshift there has no infinitesimals. Here, we show that even though a typical subshift still has no infinitesimals, it is not even balanced for letters.

\begin{proposition}\label{ballet}
The set $\mathbf{NBL}$ of uniquely ergodic subshifts which are not balanced for any letter is residual in $\TTcalprimebar$.
\end{proposition}

\begin{proof}
We first show that $\mathbf{NBL}$ is dense. To see this, recall that in the proof of Theorem~\ref{conjdense}, it was shown that any nonempty cylinder $C = [X, n]$ in $\TTcalprimebar$ contains a (mixing sofic) subshift $Z$ containing all possible concatenations of words $wu$ and $wv$ for some words $w, u, v$ where $u$ and $v$ have lengths differing by $1$, and where $wu$ and $wv$ each contain all letters of the alphabet $\A$ of $C$ (since they contained all words in $L_n(X)$). Consider the words $s = (wu)^{|wv|}$ and $t = (wv)^{|wu|}$ and any letter $a \in \A$. If $s$ and $t$ contained the same number of occurrences of $a$, then that number would be a multiple of $|wu|$ and $|wv|$; since they are relatively prime, this number of occurrences would be a multiple of $|wu| |wv|$. But $s$ and $t$ are of length $|wu||wv|$ and $|wv||wu|$ and contain both $a$ and other letters, and so this is not possible. Therefore, $s$ and $t$ contain different numbers of occurrences of $a$ for every $a \in \A$. Finally, consider any uniquely ergodic subshift $Y$ on $\{0,1\}$ which is not balanced for $0$ and $1$ (for instance, the Chacon substitution) and define $\tau$ sending $0$ to $wu$ and $1$ to $wv$. It's easily checked that $\tau^*(Y)$ is uniquely ergodic and not balanced for any letter, and $\tau^*(Y)$ is in $C$ since it's a subset of $Z \in C$
and since every word in $L_n(X)$ occurs in $wu$ and $wv$, which are subwords of all points of $\tau^*(Y)$. Since
$C$ was an arbitrary nonempty cylinder in $\TTcalprimebar$, $\mathbf{NBL}$ is dense in $\TTcalprimebar$.

Finally, we show that $\mathbf{NBL}$ is a $G_{\delta}$ in $\Scal$ (and therefore also in $\TTcalprimebar$). To see this, we note that for any finite $\A \subset \mathbb{Z}$, a uniquely ergodic subshift $X \in \Scal[\A]$ is not balanced for any letter if and only if, for all $n$, there exists $k$ so that for all $a \in \A$, there are words $v,w$ of the same length which are subwords of some
words in $L_k(X)$ where the number of $a$ in $v$ and $w$ differs by at least $n$.

For any $\A,n,k$, denote by $U(\A,n,k)$ the set of all $S \subset \A^{k}$ with this property. Then, we can write
\[
\mathbf{NBL} = \bigcap_{n \in \mathbb{N}} \bigcup_{\substack{k \in \mathbb{N},\\ \A \subset \mathbb{Z}, |\A| < \infty}} \{X \in \mathbf{UE} \mid L_{k}(X) \in U(\A,n,k)\} \]
\[ =
\mathbf{UE} \cap \bigcap_{n \in \mathbb{N}} \bigcup_{\substack{k \in \mathbb{N},\\ \A \subset \mathbb{Z}, |\A| < \infty}} \bigcup_{\substack{X \in \Scal \\L_{k}(X) \in U(\A,n,k)}} [X,k].
\]
Since $\mathbf{UE}$ is known to be a $G_{\delta}$, this completes the proof.
\end{proof}

\begin{remark}
That a generic subshift in $\TTcalprimebar$ is not balanced may be deduced from Theorem~\ref{ttmixgen} together with Proposition 5.4 of~\cite{BCBDLPP}. Proposition~\ref{ballet} proves something stronger however, namely that a generic subshift in $\TTcalprimebar$ is not balanced for any letter.
\end{remark}

\subsection{Automorphism groups and mapping class groups in $\TTcalprimebar$}

Analogous to $\Tcalbar$, the automorphism group of a generic subshift in $\TTcalprimebar$ is generated by the shift map.

\begin{corollary}\label{cor:ttbartrivauto}
The set of subshifts whose automorphism group is generated by the shift is residual in $\TTcalprimebar$.
\end{corollary}
\begin{proof}
By Theorem~\ref{ttgeneric}, a generic subshift in $\TTcalprimebar$ is infinite, minimal and has topological rank two, so again by~\cite[Thm 3.1]{DDMP1} and \cite[Sec. 7]{DDMP2021}, its automorphism group is generated by the shift map.
\end{proof}

We now turn to mapping class groups again. Recall the affine group $\textnormal{Aff}(\mathbb{Q})$ is the group of affine transformations of $\mathbb{Q}$ of the form $x \mapsto ax+b$, and is isomorphic to the subgroup of matrices in $GL_{2}(\mathbb{Q})$ of the form $\begin{pmatrix} s & 0 \\ r & 1 \end{pmatrix}$. Our goal in what remains is to show the following.

\begin{theorem}\label{thm:genericmcg}
The set of subshifts whose mapping class group $\mathcal{M}(\sigma_{X})$ is isomorphic to a subgroup of the affine group $\textnormal{Aff}(\mathbb{Q})$ is residual in $\TTcalprimebar$.
\end{theorem}

Before beginning the proof, we briefly describe some tools from~\cite{SchmiedingYang2021} we will make use of. Recall for a subshift $(X,\sigma_{X})$ the group $\coinv{\sigma_{X}}$ of coinvariants (defined in Section 2). There is a coinvariants representation for $\mathcal{M}(\sigma_{X})$ in the form of a homomorphism
$$\Phi_{X} \colon \mathcal{M}(\sigma_{X}) \to \textnormal{Aut}(\coinv{\sigma_{X}})$$
where $\textnormal{Aut}(\coinv{\sigma_{X}})$ denotes the group of automorphisms of the abelian group $\coinv{\sigma_{X}}$. Note that automorphisms in the image of $\Phi_{X}$ need not preserve the order unit $[1] \in \coinv{\sigma_{X}}$. For a uniquely ergodic subshift $(X,\sigma_{X})$ there is a homomorphism
\begin{equation*}
\begin{gathered}
R_{\mu_{X}} \colon \mathcal{M}(\sigma_{X}) \to \mathbb{R}^{*}_{>0}\\
R_{\mu_{X}} \colon [f] \mapsto \tau_{\mu_{X}}(\Phi_{X}(f)([\boldsymbol{1}])).
\end{gathered}
\end{equation*}
where $\mathbb{R}^{*}_{>0}$ is the group of positive real numbers under multiplication.

\begin{proof}
Consider the set $\mathbf{K}$ of subshifts in $\TTcalprimebar$ which are minimal, uniquely ergodic, have dimension group rank two with trivial infinitesimal subgroup, do not have linear complexity, and whose automorphism group is generated by the shift map. By Corollary~\ref{cor:ttbartrivauto}, Theorem~\ref{mainfg2}, Theorem~\ref{ttgeneric} and Proposition~\ref{dimrk2}, the set $\mathbf{K}$ is residual in $\TTcalprimebar$, so it suffices to show that any subshift in $\mathbf{K}$ has mapping class group isomorphic to a subgroup of $\textnormal{Aff}(\mathbb{Q})$.

Let $(X,\sigma_{X})$ be such a subshift. Since $(X,\sigma_{X})$ is minimal, it follows from~\cite[Thm. 4.8]{SchmiedingYang2021} that there is a short exact sequence
\begin{equation*}
1 \to K \to \mathcal{M}(\sigma_{X}) \to \textnormal{image} (\Phi_{X}) \to 1
\end{equation*}
where $K$ is a subgroup of $\aut(X,\sigma_{X}) / \langle \sigma_{X} \rangle$. Since $\aut(X,\sigma_{X}) / \langle \sigma_{X} \rangle$ is trivial by assumption, this implies $\mathcal{M}(\sigma_{X})$ is isomorphic to the subgroup $\textnormal{image} (\Phi_{X}) \subset \textnormal{Aut}(\mathcal{G}_{\sigma_{X}})$. By assumption, the group $\mathcal{G}_{\sigma_{X}}$ is rank two. Let $V$ be the rational vector space $\mathcal{G}_{\sigma} \otimes \mathbb{Q}$, and let $v_{2}$ denote the vector $\boldsymbol{1} \otimes 1$ in $V$. Since $\mathcal{G}_{\sigma_{X}}$ is rank two, we may find $g \in \mathcal{G}_{\sigma_{X}}$ such that, letting $v_{1} = g \otimes 1 \in V$, the set $\{v_{1},v_{2}\}$ forms a basis for $V$. For any $[f] \in \mathcal{M}(\sigma_{X})$, we have that $\Phi_{X}([f])$ is an automorphism of $\mathcal{G}_{\sigma_{X}}$ and extends to an automorphism of $V$, which we may represent using the basis $\{v_{1},v_{2}\}$ by a rational matrix $A_{[f]}$; since $\Phi_{X}$ is injective, we may thus identify $\mathcal{M}(\sigma_{X})$ with a subgroup of $GL_{2}(\mathbb{Q})$. We will show this subgroup is isomorphic to a subgroup of $\textnormal{Aff}(\mathbb{Q})$.



First we claim that the map $R_{\mu_X} \colon \mathcal{M}(\sigma_{X}) \to \mathbb{R}^{*}_{>0}$ is trivial, i.e. $R_{\mu_X}([f]) = 1$ for all $[f] \in \mathcal{M}(\sigma_{X})$. In~~\cite[Prop. 4.22]{SchmiedingYang2021} it is shown that if $R_{\mu_X}([f]) \ne 1$ for some $[f] \in \mathcal{M}(\sigma_{X})$, then $(X,\sigma_{X})$ is a self-induced system (see~\cite{DOP2018} for the definition of a self-induced system).
Since $(X,\sigma_{X})$ is expansive, by~\cite[Thm. 14]{DOP2018} this would imply $(X,\sigma_{X})$ is topologically conjugate to a primitive substitution subshift. However, a primitive substitution subshift has linear complexity~\cite[Thm. 2.3]{BertheDelecroix2014}. Since $(X,\sigma_{X})$ by assumption does not have linear complexity and linear complexity is preserved by topological conjugacy, altogether it follows that $R_{\mu_X}$ is the trivial map.\\
\indent Let $H_{\textnormal{aff}}$ denote the subgroup of matrices in $GL_{2}(\mathbb{Q})$ of the form $\begin{pmatrix} s & 0 \\ r & 1 \end{pmatrix}$ where $r,s \in \mathbb{Q}$. We will show that for any $[f] \in \mathcal{M}(\sigma_{X})$, the matrix $A_{[f]}$ lies in $H_{\textnormal{aff}}$; since the group $H_{\textnormal{aff}}$ is isomorphic to $\textnormal{Aff}(\mathbb{Q})$, this completes the proof.\\
\indent Fix now some element $[f] \in \mathcal{M}(\sigma_{X})$. Since $R_{\mu_{X}}$ is the trivial map, we have that $\tau_{\mu_X}(\Phi_{X}(f)([\boldsymbol{1}])) = 1$, or equivalently, that $\Phi_{X}(f)([\boldsymbol{1}]) - [\boldsymbol{1}]$ is an infinitesimal. By assumption, $X$ has no nontrivial infinitesimals, so this implies $\Phi_{X}(f)([\boldsymbol{1}]) = [\boldsymbol{1}] \in \mathcal{G}_{\sigma_{X}}$. Note that in the basis $\{v_{1},v_{2}\}$, in vector notation $\begin{pmatrix} 0 \\ 1 \end{pmatrix}$ corresponds to $v_{2} = \boldsymbol{1} \otimes 1$. Since $\Phi_{X}(f)([\boldsymbol{1}]) = [\boldsymbol{1}]$ , we have that
$$A_{[f]}\begin{pmatrix} 0 \\ 1 \end{pmatrix} = \begin{pmatrix} 0 \\ 1 \end{pmatrix}.$$
This implies $A_{[f]} = \begin{pmatrix} s & 0 \\ r & 1 \end{pmatrix}$ for some $r,s \in \mathbb{Q}$, so $A_{[f]} \in H_{\textnormal{Aff}}$ as desired.
\end{proof}


We note that since the affine group $\textnormal{Aff}(\mathbb{Q})$ is metabelian, Theorem~\ref{thm:genericmcg} in particular implies the mapping class group of a generic subshift in $\TTcalprimebar$ is metabelian.

\bibliographystyle{plain}
\bibliography{GenericBib}

\end{document}